\documentclass[12pt]{article}
\usepackage{amssymb,amsmath,bm,paralist,graphics,epsfig,graphicx,epstopdf,amsthm,geometry}
\usepackage{color}
\usepackage[colorlinks,urlcolor=blue]{hyperref}
\numberwithin{equation}{section}
\newtheorem{theorem}{Theorem}[section]
\newtheorem{lemma}[theorem]{Lemma}

\newtheorem{assumption}{Assumption}
\newtheorem{remark}{Remark}
\newtheorem*{notation}{Notation}

\newcommand{\norm}[1]{\left\Vert#1\right\Vert}

\newcommand{\keywords}[1]{\textbf{Keywords}:#1}

\begin{document}
\date{ }

\title{Non-relativistic limit of the Euler-HMP$_N$ approximation model arising in radiation hydrodynamics}
%\title{Non-relativistic limit of the Euler-HMP$_N$ approximation model arising in radiation hydrodynamics \protect\thanks{Non-relativistic limit of the Euler-HMP$_N$ approximation model arising in radiation hydrodynamics.}}
\maketitle

\vspace{.1in}
\centerline{
Zhiting Ma \footnote{School of Mathematical Sciences,
Peking University, Beijing, China.} 
%\textsuperscript{,} 
%\renewcommand{\thefootnote}{\fnsymbol{footnote}}
%\footnote{E-mail: mazt@math.pku.edu.cn. Corresponding author}
and Wen-An Yong \footnote{Department of Mathematical Sciences, 
Tsinghua University, Beijing, China.}
}

%\author[1]{Zhiting Ma*}
%
%\author[2]{Wen-an Yong}   
%
%\authormark{All \textsc{et al}}
%
%
%\address[1]{\orgdiv{School of Mathematical Sciences}, \orgname{Peking University}, \orgaddress{\state{Beijing}, \country{China}}}
%
%\address[2]{\orgdiv{Department of Mathematical Sciences}, \orgname{Tsinghua University}, \orgaddress{\state{Beijing}, \country{China}}}
%
%
%\corres{*Zhiting Ma, \orgdiv{School of Mathematical Sciences}, \orgname{Peking University}, \orgaddress{\state{Beijing}, \country{China}}. \email{mazt@math.pku.edu.cn}}

\vspace{.4in}

\abstract{In this paper, we are concerned with the non-relativistic limit of a class of computable approximation models for radiation hydrodynamics. The models consist of the compressible Euler equations coupled with moment closure approximations to the radiative transfer equation. They are first-order partial differential equations with source terms. As hyperbolic relaxation systems, they are showed to satisfy the structural stability condition proposed by W.-A. Yong (1999). Base on this, we verify the non-relativistic limit by combining an energy method with a formal asymptotic analysis.}

\vspace{.1in}


\section{Introduction}

Radiation hydrodynamics \cite{mihalas2013foundations} studies interactions of radiation and matters through momentum and energy exchanges.
It is modeled with the compressible Euler equations coupled with a radiation transport equation via an integral-type source \cite{buet2004asymptotic,lowrie1999coupling,pomraning2005equations}:
\begin{equation}\label{intro-eq1}
	\begin{aligned}
		&\partial_t \rho + \textit{div} (\rho \boldsymbol{v}) = 0,\\
		&\partial_t (\rho \boldsymbol{v}) + \textit{div} (\rho \boldsymbol{v} \otimes \boldsymbol{v}) + \nabla p = -\boldsymbol{S}_F, \\
		&\partial_t (\rho E) + \textit{div} (\rho \boldsymbol{v} E + p\boldsymbol{v} ) = - cS_E,\\
		&\partial_t I + c\boldsymbol{\mu} \cdot  \nabla I = S.
	\end{aligned}
\end{equation}
Here the unknowns $\rho$, $\boldsymbol{v}$ and $E$ denote the density, velocity and energy of the fluid, respectively;  $I = I(x, t,\boldsymbol{\mu})\geq 0$ is the radiative intensity depending on the direction variable $\boldsymbol{\mu} \in \mathcal{S}^{D-1}$  as well; the thermodynamics pressure $p=p(\rho, \theta)$ is a smooth function of $\rho$ and temperature $\theta$; $S=S(\rho, \theta, I; c)$ is the source of radiation; $c$ is the speed of light; the source term in \eqref{intro-eq1} is taken to be \cite{buet2004asymptotic}
\begin{equation}\label{2eq:soure-term-mpn}
	S = c \rho \sigma_a(\theta)\bigg(|S^{D-1}|b(\theta) - I\bigg) + \frac{1}{c} \rho \sigma_s(\theta)\bigg(\vert S^{D-1}\vert\int_{S^{D-1}} I \mathrm{d} \boldsymbol{\mu} - I\bigg),
\end{equation}
where $\sigma_a = \sigma_a(\theta)>0$ is the absorption coefficient, $\sigma_s = \sigma_s(\theta)>0$ is the scattering coefficient, and the Planck function $b = b(\theta)$ is smooth and satisfies 
\begin{equation*}
	b(\theta)>0, \qquad b'(\theta)>0;
\end{equation*}
$S_F(S_E)$ characterizes the energy (resp. impulse) exchange between the radiation and matter \cite{Jiang2015Nonrelativistic}:
\begin{equation*}
	\boldsymbol{S}_F = \int_{S^{D-1}} \boldsymbol{\mu} S \mathrm{d} \boldsymbol{\mu},\qquad S_E = \int_{S^{D-1}} S \mathrm{d} \boldsymbol{\mu}.
\end{equation*}

The theory of radiation hydrodynamics has a wide range of applications, including nonlinear pulsation, supernova explosions, stellar winds, and laser fusion \cite{mihalas2013foundations,pomraning2005equations,Zel2002physics}. 
However, the full set of radiation hydrodynamics equations are computationally expensive and numerically difficult to solve since the radiative equation is a high-dimensional integro-differential equation. Various solution methods have been developed. Among them, the moment method is quite attractive due to its numerous advantages such as clear physical interpretation and high efficiency in transitional regimes. It has been regarded as a successful tool to solve radiative equation \cite{minerbo1978maximum,levermore1996moment,CFL2015siam}.

Recently, a new moment method was proposed \cite{fan2020nonlinear,Fan2020nonlinearJcp} for the radiative transfer equation, which is basically the last equation in equations \eqref{intro-eq1}. The resultant model (called the HMP$_N$ model) is globally hyperbolic, and some important physical properties are preserved. 
In this paper, we focus on the equation \eqref{intro-eq1} with the last equation replaced by its HMP$_N$ approximation (and the source terms are treated accordingly). The resultant coupling system will be called the Euler-HMP$_N$ approximation of equations \eqref{intro-eq1}. See Section \ref{sec3-2} for the Euler-HMP$_N$ approximation.
 
The goal of this paper is to investigate the non-relativistic limit of Euler-HMP$_N$ approximation, i.e., the limit as the light speed tends to infinity. We restrict ourself to the mono-dimensional geometry. 
Under quite general assumptions, we prove that as the light speed goes to infinity, the Euler-HMP$_N$ approximation  of equations \eqref{intro-eq1} converges to 
\begin{equation*}
	\begin{aligned}
			&\partial_t \rho + \partial_x(\rho v)=0,\\
			&\partial_t (\rho v) + \partial_x\bigg(\rho v^2 + p + \frac{1}{3}b(\theta) \bigg)=0,\\
			&\partial_t \big(\rho E + b(\theta) \big) + \partial_x \big(\rho E v + p v \big) = \partial_x \bigg(\frac{1}{3\rho\sigma_a(\theta)} \partial_x b(\theta) \bigg)
	\end{aligned}
\end{equation*} 
with corresponding initial data. See details in Section \ref{4-1sec}. 

Note that the non-relativistic limit is a singular perturbation problem. Such singular limit problems have attracted much attention for many years. For instance, Marcati and Milani \cite{Marcati1988Singular,Marcati1990darcy} firstly analyzed the singular limit for weak solutions of hyperbolic balance laws with particular source terms. Bardos et al.   \cite{Bardos1988nonaccretive,Bardos1987Rosseland} studied the limit problem for non-smooth solutions of the closely related nonlinear radiative transfer equations. 
With the well-known compensated compactness theory, Marcati \cite{Marcati2000Hyperbolic} studied general $2\times 2$ systems with applications to multi-dimensional problems and a class of one-dimensional semilinear systems. Recently, for a class of first-order symmetrizable hyperbolic systems, the authors \cite{lattanzio2001hyperbolic,peng2016parabolic} studied the diffusion relaxation limit and derived parabolic type equations.

For the above works, the structural stability condition proposed by Yong \cite{yong1993singular,yong2004entropy} is the key. It is a proper counterpart of the H-theorem for the kinetic equation. Indeed, this condition has been tacitly respected by many well-developed physical theories \cite{yong2008interesting}. Recently, it was shown in \cite{Di2017nm} to be satisfied by the hyperbolic regularization models derived in \cite{CFL2015siam,CFL2014cpam}, which provides a basis for the first author to prove that the models well approximate the Navier-Stokes equations \cite{Ma2021}. 
In contrast, the Biot/squirt (BISQ) model for wave propagation in saturated porous media violates this condition and thus allows exponentially exploding asymptotic solutions \cite{liu2016stability}. On the other hand, this condition also implies that the resultant moment system is compatible with the classical theories \cite{zhao2017stability,Ma2021}. 

In this paper, we verify the structural stability condition for the Euler-HMP$_N$ system and construct formal asymptotic solutions thereof. On the basis of the stability condition, we use the energy method to prove the validity of the asymptotic approximations. Moreover, we conclude the existence of the solution to the Euler-HMP$_N$ systems in the time interval where the approximations are well-defined.

Here, we mention some related works for the equations of radiation hydrodynamics. The system \eqref{intro-eq1} was introduced by Pomraning and Mihalas \cite{pomraning2005equations} in the framework of special relativity. For the radiation hydrodynamics system with the radiation transfer equation replaced by its discrete-ordinate approximations, Rohde and Yong \cite{Rohde2007nonrelativistic} showed the existence of entropy solutions to the Cauchy problems in the framework of functions of bounded variation and investigated the non-relativistic limit of the entropy solutions. In \cite{Fan2016Non-relativistic}, Fan, Li, and Nakamura studied the non-relativistic and low Mach number limits for the Navier-Stokes-Fourier-P1 approximation radiation model. In \cite{Jiang2015Nonrelativistic}, Jiang, Li and Xie studied non-relativistic limit problem of the compressible NSF--P1 approximation radiation hydrodynamics model arising in radiation hydrodynamics. We refer to \cite{Necasova2015Nonrelativistic,Wang2010nonrelativistic,Teleaga2008,Danchin2016} for more references.
 
The paper is organized as follows. Section \ref{2sec} presents a brief introduction of MP$_N$ and HMP$_N$ moment methods for the radiative transfer equation. In Section \ref{3sec}, we verify the structural stability condition for the Euler-HMP$_N$ systems. Section \ref{4sec} is devoted to the non-relativistic limit. In particular, the formal asymptotic expansion is constructed in Subsection \ref{4-1sec} and justified in Subsection \ref{4-2sec}. Finally, we conclude our paper in Section \ref{5sec}.

\section{HMP$_N$ model}\label{2sec}

In this section, we present the HMP$_N$ model proposed by Fan et. \cite{fan2020nonlinear} for the radiative transfer equation (RTE) for a gray medium in the slab geometry:
\begin{equation}\label{2eq:rte-equ}
	\begin{aligned}
		\frac{1}{c}\frac{\partial I}{\partial t} + \mu \frac{\partial I}{\partial x} = \emph{S}(I)
	\end{aligned}.
\end{equation}
Here $I = I(x, t, \mu)\geq 0$ is the specific intensity of radiation, the variable $\mu \in [-1, 1]$ is the cosine of the angle between the photon velocity and the positive $x$-axis, the time variable $t\in \mathrm{R}_+$ and space variable $x\in \Omega$ with $\Omega$ a closed interval,  
and the right-hand side $\emph{S}(I)$ is defined in \eqref{2eq:soure-term-mpn}.

Define the $k$th moment of the specific intensity as
\begin{equation*}
	E_k = \langle I \rangle_{k} \dot{=} \int_{-1}^1 \mu^k I \mathrm{d} \mu, \qquad k \in \mathrm{N}.
\end{equation*}
Multiplying \eqref{2eq:rte-equ} by $\mu^k$ and integrating it with respect to $\mu$ over $[-1, 1]$ yield the moment equations
\begin{equation}\label{moment-equ-rte}
	\frac{1}{c}\frac{\partial E_{k}}{\partial t} + \frac{\partial E_{k+1}}{\partial x} = \langle\emph{S}(I)\rangle_k.
\end{equation}
Notice that the governing equation of $E_{k}$  depends on the $(k+1)$th moment $E_{k+1}$, which indicates that the full system contains an infinite number of equations, so we need to provide a so-called moment closure for the model. A common strategy is to construct an Ansatz: $\hat{I} = \hat{I}(E_0,E_1, \cdots, E_N;\mu)$ with a prescribed integer $N$ such that
\begin{equation*}
	\langle \hat{I}(E_0,E_1, \cdots, E_N;\mu) \rangle_{k} = E_k, \qquad k = 0,1,\cdots, N.
\end{equation*}
Then the moment closure is given by
\begin{equation*}
	E_{N+1} = \langle \hat{I}(E_0,E_1, \cdots, E_N;\mu) \rangle_{N+1}.
\end{equation*}
Based on this strategy, many moment systems have been developed, such as the $P_N$ model \cite{Jeans1917RTE}, the $M_N$ model \cite{levermore1996moment,dubroca1999theoretical}, the positive $P_N$ model \cite{hauck2010positive}, the MP$_N$ model \cite{Fan2020nonlinearJcp}, the  HMP$_N$ model \cite{fan2020nonlinear} and so on. 

In this paper, we focus on the HMP$_N$ model which is based on the MP$_N$ model \cite{Fan2020nonlinearJcp}. 
The latter takes the ansatz of the $M_1$ model (the first order of the $M_N$ model) as a weight function and then constructs the ansatz by expanding the specific intensity around the weight function in terms of orthogonal polynomials in the velocity direction. 
Therefore, we briefly describe the MP$_N$ model.

\subsection{MP$_N$ model}
The construction of the MP$_N$ model starts with the following weight function
\begin{equation}\label{2eq:defin-w}
	\omega^{[\alpha]}(\mu) = \frac{1}{(1+\alpha\mu)^4},\qquad \alpha \in (-1, 1).
\end{equation}
Here $\alpha$ is related to the low-order moment of radiation intensity and its expression will be given later. Having this weight function, we use the Gram-Schmidt orthogonalization to define a series of orthogonal polynomials on the interval $[-1,1]$:
\begin{equation*}
	\phi_{0}^{[\alpha]}(\mu) = 1, \qquad \phi_{j}^{[\alpha]}(\mu) = \mu^j - \sum_{k=0}^{j-1}\frac{\kappa_{j, k}}{\kappa_{k, k}}\phi_{k}^{[\alpha]}(\mu), \qquad j\geq 1,
\end{equation*}
where the coefficients are
\begin{equation}\label{2eq:kappa-define}
	\kappa_{j, k} = \int_{-1}^1 \mu^j \phi_{k}^{[\alpha]}(\mu)\omega^{[\alpha]}(\mu) \mathrm{d} \mu.
\end{equation}
From the orthogonality, it is easy to see that
\begin{equation}\label{2eq:kappa-define-ii}
	\kappa_{j, k} = 0, ~\text{if}~ j < k, \qquad \kappa_{k, k} = \int_{-1}^1 (\phi_{k}^{[\alpha]}(\mu))^2\omega^{[\alpha]}(\mu) \mathrm{d} \mu > 0.
\end{equation}

Set $\Phi_{i}^{[\alpha]}(\mu) = \phi_{i}^{[\alpha]}(\mu) \omega^{[\alpha]}(\mu)$ for $i = 0, 1,\cdots, N$.
The Ansatz for the MP$_N$ model is 
\begin{equation*}
	\hat{I}(E_0,E_1, \cdots, E_N;\mu) \dot{=} \sum_{i=0}^N f_i \Phi_{i}^{[\alpha]}(\mu),
\end{equation*}
where $f_i$ are the expansion coefficients. Thanks to the orthogonality, the coefficients can be expressed as
\begin{equation}\label{2-3equ:f-moment-relation}
	f_i = \frac{1}{\kappa_{i, i}}\bigg(E_i - \sum_{j=0}^{i-1}\kappa_{i, j}f_j\bigg), \qquad 0\leq i \leq N.
\end{equation}
The moment closure form is given by
\begin{equation*}
	E_{N+1} = \sum_{k=0}^N \kappa_{N+1, k}f_k.
\end{equation*}
For the MP$_N$ systems, the parameter $\alpha$ is taken to be
\begin{equation*}
	\alpha = -\frac{3E_1/E_0}{2+\sqrt{4-3(E_1/E_0)^2}}.
\end{equation*}
A simple calculation shows that $f_1 = 0$.

Define the Hilbert space $\mathbb{H}_N^{[\alpha]}$ as  
\begin{equation*}
	\mathbb{H}_N^{[\alpha]} :=\text{span} \bigg\{\Phi_{i}^{[\alpha]}(\mu),i=0,\cdots,N \bigg\}
\end{equation*}
with the inner product 
\begin{equation*}
	\langle \Phi, \Psi \rangle_{\mathbb{H}_N^{[\alpha]}}=\int_{-1}^1 \Phi(\mu) \Psi(\mu)/\omega^{[\alpha]}(\mu) \mathrm{d} \mu.
\end{equation*}
Let $\mathbb{H}$ be the space of all the admissible specific intensities for the RTE. Consider the map from $\mathbb{H}$ to $\mathbb{H}_N^{[\alpha]}$:
\begin{equation*}
	\mathrm{P}:I \rightarrow \hat{I}=\sum_{i=0}^N f_i \Phi_{i}^{[\alpha]}(\mu), \qquad f_i = \frac{\langle I, \Phi_{i}^{[\alpha]} \rangle_{\mathbb{H}_N^{[\alpha]}}} {\langle \Phi_{i}^{[\alpha]}, \Phi_{i}^{[\alpha]} \rangle_{\mathbb{H}_N^{[\alpha]}}} = \frac{\int_{-1}^1 I \phi_{i}^{[\alpha]}\mathrm{d} \mu}{\kappa_{i,i}},
\end{equation*}
where $\kappa_{ii}$ is defined in \eqref{2eq:kappa-define-ii}. Clearly, this map is an orthogonal projection.

Similar to the reduction framework in the literature \cite{CFL2015siam}, the MP$_N$ moment equation can be obtained as
\begin{equation*}
	\frac{1}{c}\mathrm{P}\frac{\partial \mathrm{P}I}{\partial t} + \mathrm{P}\mu\frac{\partial \mathrm{P}I}{\partial x} = \mathrm{P}\emph{S}(\mathrm{P}I).
\end{equation*}
Note that the unknown variables are coefficients
\begin{equation*}
	w = (f_0, \alpha, f_2, \cdots, f_N)^T  
\end{equation*}
of $\mathrm{P}I$ in the basis space $\mathbb{H}_N^{[\alpha]}$.

The MP$_2$ moment model was showed in \cite{Fan2020nonlinearJcp} to be globally hyperbolic and perform well in numerical experiments. But it allows a non-physical characteristic velocity exceeding the speed of light. When $N\geq 3$, the global hyperbolicity fails. For these reasons, the HMP$_N$ moment closure model as a novel hyperbolic regularization was proposed \cite{fan2020nonlinear}. 

\subsection{HMP$_N$ model}\label{4subsec:regularized MP_N model}
 
This class of models uses
\begin{equation}\label{2eq:defin-tilde-w}
	\tilde{\omega}^{[\alpha]}(\mu) = \frac{1}{(1+\alpha\mu)^5},\qquad \alpha \in (-1, 1)
\end{equation}
as the weight function which is different from that of MP$_N$ models. As before, we introduce the orthogonal polynomials with respect to this new weight function:
\begin{equation}\label{2eq:phi_tilde}
	\tilde{\phi}_{0}^{[\alpha]}(\mu) = 1, \qquad \tilde{\phi}_{j}^{[\alpha]}(\mu) = \mu^j - \sum_{k=0}^{j-1}\frac{\tilde{\kappa}_{j, k}}{\tilde{\kappa}_{k, k}}\tilde{\phi}_{k}^{[\alpha]}(\mu), \qquad j\geq 1.
\end{equation}
The coefficients are
\begin{equation}\label{2eq:kappa-tilde-define}
	\tilde{\kappa}_{j, k} = \int_{-1}^1 \mu^j \tilde{\phi}_{k}^{[\alpha]}(\mu)\tilde{\omega}^{[\alpha]}(\mu) \mathrm{d} \mu.
\end{equation}
and the analogue of \eqref{2eq:kappa-define-ii} also holds:
\begin{equation}\label{2-3:kappa-prop}
	\tilde{\kappa}_{j, k} = 0, ~\textit{if}~ j < k, \qquad \tilde{\kappa}_{k, k} = \int_{-1}^1 (\tilde{\phi}_{k}^{[\alpha]}(\mu))^2\tilde{\omega}^{[\alpha]}(\mu) \mathrm{d} \mu > 0.
\end{equation}

Similarly, we have a new Hilbert space
\begin{equation*}
	\tilde{\mathbb{H}}_N^{[\alpha]} :=\text{span} \bigg\{\tilde{\Phi}_{i}^{[\alpha]}(\mu) = \tilde{\phi}_{i}^{[\alpha]}(\mu)\tilde{\omega}^{[\alpha]}(\mu),i=0,\cdots, N \bigg\}
\end{equation*}
with the inner product
\begin{equation}\label{2eq:tilde-h-dot}
	\langle \Phi, \Psi \rangle_{\tilde{\mathbb{H}}_N^{[\alpha]}}=\int_{-1}^1 \Phi(\mu) \Psi(\mu)/\tilde{\omega}^{[\alpha]}(\mu) \mathrm{d} \mu
\end{equation}
and the orthogonal projection from $\mathbb{H}$ to $\tilde{\mathbb{H}}_N^{[\alpha]}$:
\begin{equation*}
	\tilde{\mathrm{P}}:I \rightarrow \hat{I}=\sum_{i=0}^N g_i \tilde{\Phi}_{i}^{[\alpha]}(\mu),\qquad g_i = \frac{\langle I, \tilde{\Phi}_{i}^{[\alpha]}(\mu) \rangle_{\tilde{\mathbb{H}}_N^{[\alpha]}}} {\langle \tilde{\Phi}_{i}^{[\alpha]}, \tilde{\Phi}_{i}^{[\alpha]} \rangle_{\tilde{\mathbb{H}}_N^{[\alpha]}}} = \frac{\int_{-1}^1 I \tilde{\phi}_{i}^{[\alpha]}\mathrm{d} \mu}{\tilde{\kappa}_{i,i}}.
\end{equation*}

Having these preparations, the HMP$_N$ models were constructed in \cite{fan2020nonlinear} as 
\begin{equation*}
	\frac{1}{c}\tilde{\mathrm{P}}\frac{\partial \mathrm{P}I}{\partial t} + \tilde{\mathrm{P}}\mu \tilde{\mathrm{P}}\frac{\partial \mathrm{P}I}{\partial x} = \tilde{\mathrm{P}}\emph{S}(\mathrm{P}I).
\end{equation*}
They can be rewritten as the equations for $w = (f_0, \alpha, f_2, \cdots, f_N)^T$:
\begin{equation}\label{4eq:w_moment_equ}
	\frac{1}{c}\frac{\partial w}{\partial t}+\tilde{D}^{-1}\tilde{M}\tilde{D}\frac{\partial w}{\partial x}=\tilde{D}^{-1}\tilde{\mathrm{S}}.
\end{equation}
Here the matrix $\tilde{D}$ is denoted as
\begin{equation*}
	\begin{aligned}
		\tilde{\mathrm{P}}\frac{\partial \mathrm{P}I}{\partial t} = (\tilde{\Phi}_i^{[\alpha]})^T\tilde{D}\frac{\partial w}{\partial t}
	\end{aligned}
\end{equation*}
and
\begin{equation}\label{2eq:defin-s}
\begin{aligned}
	&\tilde{M}=\tilde{\Lambda}^{-1}\langle \mu\tilde{\Phi}^{[\alpha]},(\tilde{\Phi}^{[\alpha]})^T\rangle_{\tilde{\mathrm{H}}_N^{[\alpha]}},
	\qquad \tilde{\Lambda} = \mathbf{diag}(\tilde{\kappa}_{0,0},\tilde{\kappa}_{1,1},\cdots,\tilde{\kappa}_{N,N}),\\
	&\tilde{S}=(\langle \tilde{\Phi}_i^{[\alpha]},\mathrm{S}(\mathrm{P}I)\rangle_{\tilde{\mathrm{H}}_N^{[\alpha]}}/\tilde{\kappa}_{i,i})_{i=0,\cdots,N}.
\end{aligned}
\end{equation}
The details can be found in the literature \cite{fan2020nonlinear}.

\section{Stability Analysis}\label{3sec}

\subsection{Structural stability condition}\label{stability_condition}

In \cite{yong1999singular}, Yong proposed a structural stability condition for systems of first-order partial differential equations with source terms:
\begin{equation*}
	U_t+\sum^D_{j=1}A_j(U)U_{x_j}=Q(U),
\end{equation*}
where $A_j(U)$ and $Q(U)$ are $n\times n$-matrix and $n$-vector smooth functions of $U\in G\subset R^n $ with state space $G$ open and convex. The subscripts $t$ and $x_j$ refer to the partial derivatives with respect to $t$ and $x_j$.

Set $Q_U=\frac{\partial Q}{\partial U}$ and define the equilibrium manifold
\begin{equation*}
	E:=\{U \in G: Q(U)=0\}.	
\end{equation*} 
The structural stability condition consists of the following three items:
\begin{enumerate}
	\item[(i)] There are an invertible $n \times n$ matrix $P(U)$ and an invertible $r\times r$ matrix $S(U)$, defined on the equilibrium manifold $E$, such that
		\begin{equation*}
			P(U)Q_U(U)=	
			\begin{bmatrix}
				0 & 0\\
				0 & S(U)
			\end{bmatrix}
			P(U), \quad \textit{for} \quad  U \in E.
		\end{equation*}
	\item[(ii)]  There is a symmetric positive definite matrix $A_0(U)$ such that
		\begin{equation*}
			A_0(U)A_j(U)=A_j(U)^TA_0(U), \quad \textit{for} \quad U \in G.
		\end{equation*}
	\item[(iii)] The left-hand side and the source term are coupled in the following way:
		\begin{equation*}
			A_0(U)Q_U(U)+Q_U^T(U)A_0(U)\leq	-P^T(U)
			\begin{bmatrix}
				0 & 0\\
				0 & I
			\end{bmatrix}
			P(U), \quad \textit{for} \quad  U \in E.
		\end{equation*}
\end{enumerate}
Here $I$ is the unit matrix of order $r$.

As shown in \cite{Yong2008}, this set of conditions has been tacitly respected by many well-developed physical theories. Condition (i) is classical for initial value problems of the system of ordinary differential equations (ODE, spatially homogeneous systems), while (ii) means the symmetrizable hyperbolicity of the PDE system. Condition (iii) characterizes a kind of coupling between the ODE and PDE parts. Recently, this structural stability condition was shown in \cite{Di2017nm} to be proper for certain moment closure systems. Furthermore, this condition also implies the existence and stability of the zero relaxation limit of the corresponding initial value problems \cite{yong1999singular}.

\subsection{Stability of the Euler-HMP$_N$ system}\label{sec3-2}

In this subsection, we verify the structural stability condition for the following one-dimensional Euler-HMP$_N$ system
\begin{equation}
	\begin{aligned}\label{4eq:radiation_hy_equ}
		& \partial_t \rho + \partial_x (\rho v) = 0,\\
		& \partial_t (\rho v) + \partial_x (\rho v^2 +p) = \rho\bigg(c\sigma_a(\theta) + \frac{1}{c}\sigma_s(\theta)\bigg) \kappa_{1,0}(\alpha)f_0 ,\\
		& \partial_t (\rho E) + \partial_x (\rho E v + p v) = c^2\rho\sigma_a(\theta)\bigg(\kappa_{0,0}(\alpha) f_0 - b(\theta)\bigg),\\
		& \partial_t w + c\tilde{D}^{-1}\tilde{M}\tilde{D} \partial_x w = c\tilde{D}^{-1}\tilde{S},
	\end{aligned}
\end{equation}
which is the equations \eqref{intro-eq1} with its last equation replaced by the HMP$_N$ approximation \eqref{4eq:w_moment_equ}. Here the following relations have been used:
\begin{equation*}
	\begin{aligned}
		S_F = \int_{-1}^1 \mu S \mathrm{d} \mu = -\rho (c\sigma_a(\theta) + \frac{1}{c}\sigma_s(\theta)) E_1,\qquad
		S_E = \int_{-1}^1 S \mathrm{d} \mu = -c\rho \sigma_a(\theta)(b(\theta) - E_0)
	\end{aligned}
\end{equation*}
with $E_0 = \kappa_{0,0}(\alpha) f_0$ and $E_1 = \kappa_{1,0}(\alpha)f_0$ due to the formula \eqref{2-3equ:f-moment-relation}.

Let $u = (\rho, ~\rho v, ~\rho E)^{T}\in \mathrm{R}^{3}$ be the hydrodynamical variables and $w = (f_0, \alpha, f_2, \cdots, f_N)^T\in \mathrm{R}^{N+1}$ be radiation variables. Denoting $F(u) = (\rho v, ~\rho v^2 + p, ~\rho E v + pv)^T$ and $\varepsilon = 1/c$, we can rewrite \eqref{4eq:radiation_hy_equ} as 
\begin{equation}\label{4eq:U_equ}
	\partial_t U + \frac{1}{\varepsilon}A(U;\varepsilon)\partial_x U = \frac{1}{\varepsilon^2} Q(U;\varepsilon),
\end{equation}
with 
\begin{equation*}
	\begin{aligned}
		&U = \begin{pmatrix}
			u\\w
		\end{pmatrix},\qquad A(U;\varepsilon)=\begin{pmatrix}
			\varepsilon F_u(u)&0\\
			0&\tilde{D}^{-1}\tilde{M}\tilde{D}
		\end{pmatrix}, \qquad Q(U;\varepsilon)=\begin{pmatrix}
			q^{(1)}(U;\varepsilon)\\q^{(2)}(U;\varepsilon)
		\end{pmatrix},\\
		&q^{(1)}(U;\varepsilon) \triangleq \bigg(0,\quad \rho(\varepsilon\sigma_a(\theta) + \varepsilon^3 \sigma_s(\theta)) \kappa_{1,0}(\alpha)f_0, \quad \rho\sigma_a(\theta)(\kappa_{0,0}(\alpha) f_0 - b(\theta))\bigg)^T,\\
		&q^{(2)}(U;\varepsilon) \triangleq \varepsilon\tilde{D}^{-1}\tilde{S}(U;\varepsilon).
	\end{aligned}
\end{equation*}
Note that $\tilde{D}=\tilde{D}(U)$ and $\tilde{M}=\tilde{M}(U)$ are independent of $\varepsilon$. The state space is
\begin{equation*}
	G = \{U = (u, w) \vert \rho > 0, ~\theta > 0, ~\alpha \in (-1,1) \}.
\end{equation*}

Next, we write down the explicit expression of $\tilde{S} = \tilde{S}(U; \varepsilon)$. 
Recall that $PI = \sum_{i=0}^N f_i \Phi_{i}^{[\alpha]}(\mu)$. Then by its definition \eqref{2eq:soure-term-mpn} we have 
\begin{equation*}
	S(PI) = \rho \bigg(\frac{1}{2}c\sigma_a b(\theta) + \frac{1}{2c}\sigma_s \int_{-1}^1 I \mathrm{d}\mu -(c\sigma_a + \frac{1}{c}\sigma_s) \sum_{j=0}^N f_j\Phi_j^{[\alpha]} \bigg).
\end{equation*}
From \cite{fan2020nonlinear} we know that $\Phi_j^{[\alpha]} = \alpha\tilde{\Phi}_{j+1}^{[\alpha]} +\beta_j\tilde{\Phi}_j^{[\alpha]}$ with $\beta_i = \frac{\kappa_{i, i}}{\tilde{\kappa}_{i, i}}$. Thus, we compute the $i$th component of $\tilde{S}$ in equation \eqref{2eq:defin-s} : 
\begin{equation*}\label{2c-suorce-equ}
	\begin{aligned}
		\tilde{S}_i(U; \varepsilon)=& \frac{<\tilde{\Phi}_i^{[\alpha]},S(PI)>_{\tilde{H}_N^{[\alpha]}}}{\tilde{\kappa}_{i,i}}\\
		=& \frac{1}{\tilde{\kappa}_{i,i}}\int_{-1}^1 \tilde{\phi}_i^{[\alpha]}\rho \bigg(\frac{1}{2}c\sigma_a b(\theta) + \frac{1}{2c}\sigma_s \int_{-1}^1 I \mathrm{d}\mu -(c\sigma_a + \frac{1}{c}\sigma_s) \sum_{j=0}^N f_j\Phi_j^{[\alpha]} \bigg)\mathrm{d}\mu\\
		=&\frac{\rho (\frac{1}{\varepsilon} \sigma_a b(\theta) + \varepsilon \sigma_s E_0)}{2\tilde{\kappa}_{i,i}}\int_{-1}^1 \tilde{\phi}_i^{[\alpha]}\mathrm{d}\mu - \frac{\rho(\frac{1}{\varepsilon} \sigma_a + \varepsilon \sigma_s)}{\tilde{\kappa}_{i,i}}\sum_{j=0}^N f_j\int_{-1}^1 \tilde{\phi}_i^{[\alpha]} \Phi_j^{[\alpha]} \mathrm{d}\mu\\
		=&\frac{\rho (\frac{1}{\varepsilon} \sigma_a b(\theta) + \varepsilon \sigma_s E_0)}{2\tilde{\kappa}_{i,i}}\int_{-1}^1 \tilde{\phi}_i^{[\alpha]}\mathrm{d}\mu- \frac{\rho(\frac{1}{\varepsilon}  \sigma_a + \varepsilon \sigma_s)}{\tilde{\kappa}_{i,i}}\sum_{j=0}^N f_j\int_{-1}^1 \tilde{\phi}_i^{[\alpha]} \bigg(\alpha\tilde{\Phi}_{j+1}^{[\alpha]} +\beta_j\tilde{\Phi}_j^{[\alpha]}\bigg) \mathrm{d}\mu\\
		=&\frac{\rho (\frac{1}{\varepsilon} \sigma_a b(\theta) + \varepsilon \sigma_s E_0)}{2\tilde{\kappa}_{i,i}}\int_{-1}^1 \tilde{\phi}_i^{[\alpha]}\mathrm{d}\mu - \frac{\rho(\frac{1}{\varepsilon} \sigma_a + \varepsilon \sigma_s)}{\tilde{\kappa}_{i,i}}\sum_{j=0}^N f_j \bigg(\alpha\delta_{i, j+1}\tilde{\kappa}_{i,i} +\beta_j\delta_{i, j}\tilde{\kappa}_{i,i}\bigg) \\
		=&\frac{1}{\varepsilon} \frac{\rho ( \sigma_a b(\theta) + \varepsilon^2 \sigma_s \kappa_{0,0}f_0)}{2\tilde{\kappa}_{i,i}}\int_{-1}^1 \tilde{\phi}_i^{[\alpha]}\mathrm{d}\mu - \frac{1}{\varepsilon} \rho(\sigma_a + \varepsilon^2 \sigma_s)(\alpha f_{i-1} +\beta_i f_i) .\\
	\end{aligned}
\end{equation*}
Here we have used $E_0 = \int_{-1}^1 I \mathrm{d}\mu= \kappa_{0,0}f_0$ and $\int_{-1}^1 \tilde{\phi}_i^{[\alpha]} \tilde{\Phi}_{j}^{[\alpha]}\mathrm{d}\mu = \delta_{i, j} \tilde{\kappa}_{i, i}$.
Set $\hat{S}(U; \varepsilon) = \varepsilon \tilde{S}(U; \varepsilon)$. We have $q^{(2)}(U;\varepsilon)=\tilde{D}^{-1}\hat{S}(U; \varepsilon)$ and 
\begin{equation}\label{2eq:s-hat}
	\hat{S}_i(U; \varepsilon) = \frac{\rho ( \sigma_a b + \varepsilon^2 \sigma_s \kappa_{0,0}f_0)R_i}{2\tilde{\kappa}_{i,i}} - \rho(\sigma_a + \varepsilon^2 \sigma_s)(\alpha f_{i-1} +\beta_i f_i)
\end{equation}
with $R_i \triangleq \int_{-1}^1 \tilde{\phi}_i^{[\alpha]}\mathrm{d} \mu$.
Note that $\kappa_{0,0}$, $\tilde{\kappa}_{i,i}$, $R_i$ and $\beta_i$ depend on $\alpha$ and $\hat{S}_i(U; \varepsilon)$ is a polynomial of $\varepsilon$. 
Since $f_{-1} = f_1 =0$, $\hat{S}(U; 0)$ can be rewritten as
\begin{equation}\label{2eq:hat-s}
	\begin{aligned}
		\hat{S}_0(U; 0) & = \frac{\rho \sigma_a}{\tilde{\kappa}_{0,0}}(b -\kappa_{0,0} f_0),\qquad
		\hat{S}_1(U; 0) = \frac{\rho\sigma_a b}{2\tilde{\kappa}_{1,1}}R_1 -\rho\sigma_a\alpha f_0, \qquad
		\hat{S}_2(U; 0) = \frac{\rho \sigma_a b}{2\tilde{\kappa}_{2,2}}R_2 - \rho\sigma_a\beta_2f_2 ,\\
		\hat{S}_i(U; 0) &= \frac{\rho \sigma_a b}{2\tilde{\kappa}_{i,i}}R_i - \rho\sigma_a(\alpha f_{i-1} +\beta_i f_i), \qquad \textit{for}~ i = 3,\cdots,N.
	\end{aligned}
\end{equation}
Here $\rho$, $\sigma_a$, $b$, $\kappa_{i,i}$, $\tilde{\kappa}_{i,i} >0$.

For $\kappa_{i, j}$ and $\tilde{\kappa}_{i, j}$, we have the following explicit expressions.
\begin{equation}\label{2lemma:kappa}
	\begin{aligned}
		&\kappa_{0,0} = \int_{-1}^1 w^{[\alpha]}(\mu) \mathrm{d}\mu = \frac{2 (3+\alpha^2)}{3 (1 - \alpha^2)^3}, \qquad \tilde{\kappa}_{0,0} = \int_{-1}^1 \tilde{w}^{[\alpha]}(\mu) \mathrm{d}\mu = \frac{2 \left(\alpha^2+1\right)}{\left(\alpha^2-1\right)^4},\\
		&\kappa_{1,0} = \int_{-1}^1 \mu w^{[\alpha]}(\mu) \mathrm{d}\mu = \frac{8 \alpha}{3 \left(\alpha^2-1\right)^3},\qquad \tilde{\kappa}_{1,0} = \int \mu \tilde{w}^{[\alpha]}(\mu) \mathrm{d}\mu = -\frac{2 \alpha \left(\alpha^2+5\right)}{3 \left(\alpha^2-1\right)^4},\\
		&\tilde{\kappa}_{1,1}(0) = \int_{-1}^1 (\tilde{\phi}_{1}^{[0]})^2(\mu) \tilde{w}^{[0]}(\mu) \mathrm{d}\mu = \frac{2}{3}.
	\end{aligned}
\end{equation}	
Here $\tilde{\phi}_{1}^{[\alpha]}= \tilde{\phi}_1^{[\alpha]}(\mu) = \mu-\tilde{\kappa}_{1, 0}(\alpha)/\tilde{\kappa}_{0, 0}(\alpha)$ according to equations \eqref{2eq:phi_tilde}.
These can be easily checked by using the expressions of $w^{[\alpha]}(\mu)$ and $\tilde{w}^{[\alpha]}(\mu)$ given in \eqref{2eq:defin-w} and \eqref{2eq:defin-tilde-w}.

For $R_i = R_i(\alpha)$ in \eqref{2eq:s-hat}, we have
\begin{lemma}\label{2lemma:R_i}
	$R_0(\alpha) = 2$, $R_1(\alpha) = \frac{2 \alpha (\alpha^2+5)}{3 (\alpha^2+1)}$, $R_i(0) = 0$ and $R'_i(0) = 0$ for $i \geq 2$.
\end{lemma}
\begin{proof}
	According to equations \eqref{2eq:phi_tilde}, we know that $\tilde{\phi}_{0}^{[\alpha]}=1$ and $\tilde{\phi}_1^{[\alpha]}(\mu) = \mu-\tilde{\kappa}_{1, 0}(\alpha)/\tilde{\kappa}_{0, 0}(\alpha)$.
Thereby
	\begin{equation*}
		\begin{aligned}
			R_0(\alpha) = \int_{-1}^1 \tilde{\phi}_0^{[\alpha]}(\mu) \mathrm{d}\mu = 2,\qquad R_1(\alpha) = \int_{-1}^1 \tilde{\phi}_1^{[\alpha]}(\mu) \mathrm{d}\mu = \int_{-1}^1\mu \mathrm{d}\mu - 2\frac{\tilde{\kappa}_{1,0}(\alpha)}{\tilde{\kappa}_{0,0}(\alpha)} = \frac{2\alpha (\alpha^2+5)}{3 (\alpha^2+1)}.
		\end{aligned}
	\end{equation*}
Since $\tilde{w}^{[0]}(\mu)=1 $ in \eqref{2eq:defin-tilde-w}, we have 
% and $\tilde{\phi}_{0}^{[\alpha]}=1$ in \eqref{2eq:defin-tilde-w} and \eqref{2eq:phi_tilde}
\begin{equation*}
	R_i(0) = \int_{-1}^1 \tilde{\phi}_i^{[0]}(\mu) \mathrm{d}\mu = \int_{-1}^1 \mu^0 \tilde{\phi}_i^{[0]}(\mu)\tilde{w}^{[0]}(\mu) \mathrm{d} \mu = \tilde{\kappa}_{0, i}(0) = 0, ~ \text{for}~ i \geq 2.
\end{equation*}
Here we have used the orthogonality of $\tilde{\kappa}_{i, j}$ in \eqref{2-3:kappa-prop}  i.e., $\tilde{\kappa}_{i, j}(\alpha) = 0$ for $i < j$.
%For $R_i'(0)$, we have the following relation  and $\frac{\partial \tilde{w}^{[\alpha]}(\mu)}{\partial \alpha} = - 5\mu \tilde{w}^{[\alpha]}(\mu)$
Based on the expression of $\tilde{\kappa}_{i, j}$ in \eqref{2eq:kappa-tilde-define}, we have
\begin{equation*}
	\tilde{\kappa}'_{0, i}(\alpha) = \int_{-1}^1 \frac{\partial \tilde{\phi}_i^{[\alpha]}(\mu)}{\partial \alpha} \tilde{w}^{[\alpha]}(\mu) \mathrm{d} \mu + \int_{-1}^1 \tilde{\phi}_i^{[\alpha]}(\mu) \frac{\partial \tilde{w}^{[\alpha]}(\mu)}{\partial \alpha} \mathrm{d} \mu.
\end{equation*}
Note that $\frac{\partial \tilde{w}^{[\alpha]}(\mu)}{\partial \alpha} = \frac{- 5\mu}{(1+\alpha\mu)^6}$, thus $\frac{\partial \tilde{w}^{[0]}(\mu)}{\partial \alpha} = - 5\mu $.
Taking $\alpha \rightarrow 0$, we can obtain
	\begin{equation*}
		\begin{aligned}
			R'_i(0) &= \int_{-1}^1 \frac{\partial \tilde{\phi}_i^{[0]}(\mu)}{\partial \alpha} \tilde{w}^{[0]}(\mu) \mathrm{d} \mu =\tilde{\kappa}'_{0, i}(0) - \int_{-1}^1 \tilde{\phi}_i^{[0]}(\mu) \frac{\partial \tilde{w}^{[0]}(\mu)}{\partial \alpha} \mathrm{d} \mu \\
			&= \tilde{\kappa}'_{0, i}(0) + 5 \int_{-1}^1 \mu \tilde{\phi}_i^{[0]}(\mu)\mathrm{d} \mu\\
			&= \tilde{\kappa}'_{0, i}(0) + 5 \int_{-1}^1 \mu \tilde{\phi}_i^{[0]}(\mu) \tilde{w}^{[0]}(\mu) \mathrm{d} \mu\\
			& = \tilde{\kappa}'_{0, i}(0) + 5 \tilde{\kappa}_{1, i}(0).
		\end{aligned}
	\end{equation*}
Similarly, it follows from  the orthogonality of $\tilde{\kappa}_{i, j}$ that
	\begin{equation*}
		R'_i(0) = 0, ~ \textit{for} ~i \geq 2.	
	\end{equation*}
%	The lemma is proved.
\end{proof}

%Then we analyze the equilibrium state $U_{eq}$ of system \eqref{4eq:U_equ}.  
The equilibrium manifold $G_{eq}$ is defined as following 
\begin{equation*}
	G_{eq} = \{ U\in G: Q(U;0)=0 \}.
\end{equation*}
Due to equation \eqref{4eq:U_equ} and the expression of $\hat{S}$ in \eqref{2eq:hat-s}, we know that $U\in G_{eq}$ if and only if $\hat{S}(U; 0)=0$. We denote the equilibrium state as $U_{eq}$.
Using formulas \eqref{2lemma:kappa} and Lemma \ref{2lemma:R_i}, one can obtain the equilibrium state $U_{eq}$ as
\begin{equation}\label{2eq:radiation-equilibrium}
	f_0 = \frac{b(\theta)}{\kappa_{0,0}(0)}=\frac{1}{2}b(\theta), \qquad \alpha = 0, \qquad f_i = 0, \qquad \textit{for}\quad i=2,\cdots,N.
\end{equation}
It can be seen from system \eqref{4eq:U_equ} that the source term of the fluid variable is also zero on the above-mentioned equilibrium manifold. It is worth noting that for any $U_{eq}\in G_{eq}$ and any $\varepsilon$, there are
\begin{equation*}
	Q(U_{eq};\varepsilon)=0.
\end{equation*}

Next, we verify that the Euler-HMP$_N$ system \eqref{4eq:U_equ} satisfies the structural stability condition. 
Throughout this paper, we make the standard thermodynamical assumptions \cite{godlewski1991hyperbolic}:
%First, we review the pressure and internal energy in the equation \eqref{4eq:U_equ} to satisfy the usual thermodynamics assumptions \cite{godlewski1991hyperbolic}:
\begin{equation*}
	p_\theta(\rho, \theta),\quad p_\rho(\rho, \theta),\quad e_\theta(\rho, \theta)>0, \qquad \textit{for} ~ \rho >0, \theta>0.
\end{equation*}
Assume the existence of a specific entropy function $s=s(\rho, e)$ satisfying the classical Gibbs relationship
\begin{equation*}
	\theta \mathrm{d} s = \mathrm{d} e + p \mathrm{d} \nu,\quad \nu := 1/\rho.
\end{equation*}
We take $\eta(u) = -\rho s(\rho, e)$ as the classical entropy function of Euler equations. This means that $\eta_{uu}$ is symmetrizer of Euler equations \cite{rohde2008dissipative}. According to equation \eqref{2eq:defin-s},  $\tilde{\Lambda}\tilde{M} = \langle \mu\tilde{\Phi}^{[\alpha]},(\tilde{\Phi}^{[\alpha]})^T\rangle_{\tilde{\mathrm{H}}_N^{[\alpha]}}$ is symmetric. Therefore, it can be seen that the Euler-HMP$_N$ system \eqref{4eq:U_equ} has the following symmetrizer 
\begin{equation}\label{2eq:radiation-A0}
	A_0(U)=\begin{pmatrix}
		\eta_{uu}&0\\
		0&\tilde{D}^T\tilde{\Lambda}\tilde{D}
	\end{pmatrix}.
\end{equation}
That is, $A_0(U)A(U;\varepsilon) = A(U;\varepsilon)^TA_0(U)$. Note that the symmetrizer $A_0 = A_0(U)$ is independent with $\varepsilon$.

In order to verify the first and third requirement in the structural stability condition, we need to compute $Q_{U}(U_{eq};0)$. Therefore, we now write down the source term of system \eqref{4eq:U_equ} as:
\begin{equation*}
	\begin{aligned}
		&Q(U;0)=\begin{pmatrix}
			q^{(1)}(U;0)\\q^{(2)}(U;0)
		\end{pmatrix},\\
		&q^{(1)}(U;0) =\big(0,\quad 0, \quad \rho\sigma_a(\theta)(\kappa_{0,0}(\alpha) f_0 - b(\theta))\big)^T,\\
		&q^{(2)}(U;0) = \tilde{D}^{-1}(U)\hat{S}(U;0).
	\end{aligned}
\end{equation*}
Set $S_{\rho E} \dot{=} \rho\sigma_a(\theta)(\kappa_{0,0}(\alpha) f_0-b(\theta))$. Resorting to formulas \eqref{2lemma:kappa} and \eqref{2eq:radiation-equilibrium}, we note that, on the equilibrium manifold $G_{eq}$,
\begin{equation*}
	\begin{aligned}
		&\frac{\partial S_{\rho E}}{\partial \rho}(U_{eq};0) = - \rho \sigma_a  b' \theta_{\rho},&  &\frac{\partial S_{\rho E}}{\partial (\rho v)}(U_{eq};0) = - \rho \sigma_a  b' \theta_{\rho v},\\
		&\frac{\partial S_{\rho E}}{\partial (\rho E)}(U_{eq};0)  = - \rho \sigma_a  b' \theta_{\rho E},&  &\frac{\partial S_{\rho E}}{\partial f_0}(U_{eq};0)= 2\rho \sigma_a ,\\
		&\frac{\partial S_{\rho E}}{\partial w}(U_{eq};0) = 0, \qquad \textit{for}~ w\neq f_0.
	\end{aligned}
\end{equation*}
For $q^{(2)}(U;0)$, we know that 
\begin{equation*}
	\begin{aligned}
		\frac{\partial (\tilde{D}^{-1}\hat{S})}{\partial U}(U_{eq};0)=\tilde{D}^{-1}\frac{\partial \hat{S}}{\partial U}(U_{eq};0) + \frac{\partial \tilde{D}^{-1}}{\partial U}\hat{S}(U_{eq};0) =\tilde{D}^{-1}\frac{\partial \hat{S}}{\partial U}(U_{eq};0).
	\end{aligned}
\end{equation*}
Thus we need to compute $\frac{\partial \hat{S}}{\partial U}(U_{eq};0)$. 
Noted that $\hat{S}_0 = \frac{\rho \sigma_a(\theta)}{\tilde{\kappa}_{0,0}(\alpha)}(b(\theta) -\kappa_{0,0}(\alpha) f_0)$ according to equations \eqref{2eq:hat-s}. Using formulas \eqref{2lemma:kappa}, we can obtain 
\begin{equation*}
	\begin{aligned}
		&\frac{\partial \hat{S}_0}{\partial \rho}(U_{eq};0) = \frac{1}{2} \rho \sigma_a  b' \theta_{\rho}, & &\frac{\partial \hat{S}_0}{\partial (\rho v)}(U_{eq};0) = \frac{1}{2} \rho \sigma_a  b' \theta_{\rho v} ,\\
		&\frac{\partial \hat{S}_0}{\partial (\rho E)}(U_{eq};0)  = \frac{1}{2} \rho \sigma_a  b' \theta_{\rho E}, & &\frac{\partial \hat{S}_0}{\partial f_0}(U_{eq};0)= - \rho \sigma_a ,\\
		&\frac{\partial \hat{S}_0}{\partial w}(U_{eq};0) = 0, \quad \textit{for} \quad w\neq f_0.
	\end{aligned}
\end{equation*}
Similarly, for $\hat{S}_1 = \frac{\rho\sigma_a (\theta)b(\theta)}{2\tilde{\kappa}_{1,1}(\alpha)}R_1(\alpha) -\rho \alpha \sigma_a(\theta) f_0$, we have 
\begin{equation*}
	\begin{aligned}
		\frac{\partial \hat{S}_1}{\partial u}(U_{eq};0) =& 0, \\
		\frac{\partial \hat{S}_1}{\partial w}(U_{eq};0)  =& 0, \qquad \textit{for} ~ w \neq \alpha, \\
		\frac{\partial \hat{S}_1}{\partial \alpha}(U_{eq};0)  =& \frac{\rho\sigma_a b(R_1'(0)\tilde{\kappa}_{1,1}(0)-R_1(0) \tilde{\kappa}'_{1,1}(0))}{2\tilde{\kappa}^2_{1,1}(0)}- \frac{1}{2}\rho\sigma_a b= 2\rho\sigma_a b.
	\end{aligned}
\end{equation*}
%In the above formula, $'$ represents the derivative of the function with respect to $\alpha$.
When $i\geq 2$, we know that $\hat{S}_i(U;0)= \frac{\rho\sigma_a(\theta) b(\theta)}{2\tilde{\kappa}_{i,i}(\alpha)}R_i(\alpha) - \rho\sigma_a(\theta) (\alpha f_{i-1} +\beta_i(\alpha) f_i)$. Analogously, it is easy to show that
\begin{equation*}
	\begin{aligned}
		\frac{\partial \hat{S}_i}{\partial u}(U_{eq};0) = 0,
		\qquad \frac{\partial \hat{S}_i}{\partial f_i}(U_{eq};0)  = -\rho \sigma_a \beta_i(0),\qquad \frac{\partial \hat{S}_i}{\partial w}(U_{eq};0)  = 0, \qquad \textit{for}~ w \neq f_i.
	\end{aligned}
\end{equation*}
Resorting to the explicit expression of $\tilde{D}$ in literature \cite{fan2020nonlinear} and equation \eqref{2eq:defin-s}, we can obtain
\begin{equation}\label{2eq:tilde-d}
	\begin{aligned}
		&\tilde{D}^{-1}(U_{eq}) = \mathbf{diag}\bigg(\beta_0^{-1}(0),~(-2b(\theta))^{-1},~\beta_2^{-1}(0),~\cdots,~\beta_N^{-1}(0)\bigg),\\
		&\tilde{D}^T\tilde{\Lambda}\tilde{D}(U_{eq})=\mathbf{diag}\bigg(\beta_0^2(0) \tilde{\kappa}_{0,0}(0),(-2b(\theta))^2\tilde{\kappa}_{1,1}(0),\beta_2^2(0) \tilde{\kappa}_{2,2}(0),\cdots, \beta_N^2(0)\tilde{\kappa}_{N,N}(0)\bigg).\\
	\end{aligned}
\end{equation}
Here $\tilde{D}^{-1}(U_{eq})$ and  $\tilde{D}^T\tilde{\Lambda}\tilde{D}(U_{eq})$ are matrices belonging in $\mathrm{R}^{(N+1)\times(N+1)}$.
In summary, the Jacobian matrix $Q_U(U_{eq};0)$ is
\begin{equation}\label{2eq:jacobi-q}
	\begin{aligned}
		Q_U(U_{eq};0) =\begin{pmatrix}
			Q_1(U_{eq};0) & 0 \\
			0 & -\rho \sigma_a I_{N\times N}
		\end{pmatrix},
	\end{aligned}
\end{equation}
where
\begin{equation}\label{2eq:radiation-Q}
	\begin{aligned}
		Q_1(U_{eq};0)& = \begin{pmatrix}
			0 & 0 & 0 & 0  \\
			0 & 0 & 0 & 0 \\
			-\rho \sigma_a b' \theta_{\rho } & -\rho \sigma_a b' \theta_{\rho v} & -\rho \sigma_a b' \theta_{\rho E} & 2\rho \sigma_a \\
			\frac{1}{2}\rho \sigma_a b' \theta_{\rho } & \frac{1}{2}\rho \sigma_a b' \theta_{\rho v} & \frac{1}{2}\rho \sigma_a b' \theta_{\rho E} & -\rho \sigma_a  \\
		\end{pmatrix}.
	\end{aligned}
\end{equation}
Obviously the rank of $Q_1(U_{eq};0)$ is $1$, so the rank of $Q_U(U_{eq};0)$ is $N+1$. 

Using equation \eqref{2eq:tilde-d}, we can rewrite $A_0(U_{eq})$ in the same block form as
\begin{equation*}
	\begin{aligned}
		&A_0(U_{eq}) = \begin{pmatrix}
			\eta_{uu}&0\\
			0&\tilde{D}^T\tilde{\Lambda}\tilde{D}(U_{eq})
		\end{pmatrix}=\begin{pmatrix}
			\begin{pmatrix}
			\eta_{uu}&0\\
			0 & 2 \\
			\end{pmatrix} & 0\\
			0&\hat{A}_0(U_{eq})_{N\times N}
		\end{pmatrix},\\
		&\hat{A}_0(U_{eq})_{N\times N} = \mathbf{diag}\bigg((-2b(\theta))^2\tilde{\kappa}_{1,1}(0),\quad \beta_2^2(0) \tilde{\kappa}_{2,2}(0),~\cdots, ~\beta_N^2(0) \tilde{\kappa}_{N,N}(0)\bigg).
	\end{aligned}
\end{equation*}
Noted that $\eta_{\rho E} = -\frac{1}{\theta}$ \cite{rohde2008dissipative}, we can obtain 
%From the literature \cite{rohde2008dissipative} we know that there are $\eta_{\rho E} = -\frac{1}{\theta}$ and $\theta_{\rho E}>0$, and then
%\begin{equation*}
%	\eta_{\rho E,\rho}=\frac{\theta_{\rho}}{\theta^2},\qquad \eta_{\rho E,\rho v}=\frac{\theta_{\rho v}}{\theta^2},\qquad \eta_{\rho E,\rho E}=\frac{\theta_{\rho E}}{\theta^2}.
%\end{equation*}
%Use the block form of $A_0(U_{eq})$ and combine with \eqref{2eq:jacobi-q} to get
\begin{equation}\label{2eq:a0-q}
	\begin{aligned}
		A_0(U_{eq})Q_U(U_{eq};0)=\begin{pmatrix}
			H &0\\
			0&-\rho\sigma_a \hat{A}_0(U_{eq})_{N\times N}
		\end{pmatrix},
	\end{aligned}
\end{equation}
where
\begin{equation}\label{2eq:radiation-H}
\begin{aligned}
	H = \begin{pmatrix}
			\eta_{uu}&0\\
			0 & 2 \\
			\end{pmatrix} Q_1(U_{eq};0) = \begin{pmatrix}
		- \rho\sigma_a b' \frac{\theta_{\rho}}{\theta^2}\theta_{\rho} & - \rho\sigma_a b' \frac{\theta_{\rho}}{\theta^2}\theta_{\rho v} & - \rho\sigma_a b' \frac{\theta_{\rho}}{\theta^2}\theta_{\rho E}& 2\rho \sigma_a\frac{\theta_{\rho}}{\theta^2} \\
		-\rho\sigma_a b' \frac{\theta_{\rho v}}{\theta^2}\theta_{\rho} & -\rho\sigma_a b' \frac{\theta_{\rho v}}{\theta^2}\theta_{\rho v} & - \rho\sigma_a b' \frac{\theta_{\rho v}}{\theta^2}\theta_{\rho E}& 2 \rho\sigma_a\frac{\theta_{\rho v}}{\theta^2}\\
		-\rho\sigma_a b' \frac{\theta_{\rho E}}{\theta^2}\theta_{\rho} & - \rho\sigma_a b' \frac{\theta_{\rho E}}{\theta^2}\theta_{\rho v} & -\rho\sigma_a b' \frac{\theta_{\rho E}}{\theta^2}\theta_{\rho E}& 2\rho \sigma_a\frac{\theta_{\rho E}}{\theta^2} \\
		\rho\sigma_a b' \theta_{\rho} & \rho\sigma_a b' \theta_{\rho v} & \rho\sigma_a b' \theta_{\rho E}& -2\rho\sigma_a \\
	\end{pmatrix}.
\end{aligned}
\end{equation}

Take $P\in \mathrm{R}^{(N+4)\times(N+4)}$ as following
\begin{equation}\label{2eq:radiation-P}
	\begin{aligned}
		P = a\begin{pmatrix}
			P_1 & 0\\
			0 & I_{N\times N}
		\end{pmatrix},
		\qquad P_1= \begin{pmatrix}
			1 & 0 & 0 & 0 \\
			0 & 1 & 0 & 0 \\
			0 & 0 & 1 & 2 \\
			-b'\theta_{\rho } & -b'\theta_{\rho v} & -b'\theta_{\rho E} & 2 \\
		\end{pmatrix},
	\end{aligned}
\end{equation}
Here $a$ is an undetermined non-zero constant. Obviously, $P$ is an invertible matrix  since $\det(P) = 2a(1+b'\theta_{\rho E}) \neq 0$. In fact, we have 
\begin{equation*}
	\begin{aligned}
		PQ_{U} = a\begin{pmatrix}
			P_1 & 0\\
			0 & I_{N\times N}
		\end{pmatrix}
		\begin{pmatrix}
			Q_1 & 0\\
			0 & -\rho \sigma_a I_{N\times N}
		\end{pmatrix}=
		a\begin{pmatrix}
			P_1Q_1 & 0\\
			0 & -\rho \sigma_a I_{N\times N}
		\end{pmatrix} .
	\end{aligned}
\end{equation*}
Simple calculation shows that
\begin{equation*}
	P_1Q_1 = \begin{pmatrix}
		0 & 0 \\
		0 & -\rho\sigma_a( 1+b'\theta_{\rho E})
	\end{pmatrix}P_1,
\end{equation*}
Thus the first requirement of structural stability condition is met. Moreover, the third requirement of stability condition need $P$ holds the following inequality.
\begin{equation*}
	\begin{aligned}
		&A_0(U_{eq})Q_U(U_{eq};0) + Q_U^T(U_{eq};0)A_0(U_{eq})+ P^T\begin{pmatrix}
		\mathbf{diag}(0,0,0,1) & 0\\
		0 & I_{N\times N}
		\end{pmatrix} P \leq 0.
	\end{aligned}
\end{equation*}
Due to the expression \eqref{2eq:a0-q}, the above inequality is equivalent to
\begin{equation}\label{2eq:a-constraint1}
	\begin{aligned}
		&H + H^T + a^2 P_1^{T}\mathbf{diag}(0,0,0,1) P_1 \leq 0	\\
		&-2\rho\sigma_a \hat{A}_0(U_{eq})_{N\times N} + a^2 I_{N\times N} \leq 0	
	\end{aligned}
\end{equation}
%and 
%\begin{equation}\label{2eq:a-constraint1}
%	\begin{aligned}
%		&-2\rho\sigma_a \hat{A}_0(U_{eq})_{N\times N} + a^2 I_{N\times N} \leq 0.
%	\end{aligned}
%\end{equation}
%
%According to $Q_U(U_{eq};0)$ and $A_0(U_{eq})Q_U(U_{eq};0)$ expression \eqref{2eq:jacobi-q} and \eqref{2eq:a0-q}, 
Set
\begin{equation*}
	\begin{aligned}
		&K =(-\rho \sigma_a b'\theta_{\rho},~-\rho \sigma_a b'\theta_{\rho v},~-\rho \sigma_a b'\theta_{\rho E}, ~2\rho \sigma_a),\\
		&L =(\frac{\theta_{\rho}}{\theta^2},~\frac{\theta_{\rho v}}{\theta^2},~\frac{\theta_{\rho E}}{\theta^2},~-1).
	\end{aligned}
\end{equation*}
Then, we have
\begin{equation*}
	H = L^T K,\qquad P_1^{T}\mathbf{diag}(0,0,0,1) P_{1} = \frac{1}{\rho^2\sigma_a^2}K^T K.
\end{equation*}
Hence, the first inequality can be rewritten as
\begin{equation*}
	\begin{aligned}
			&H + H^T + a^2P_1^{T}\mathbf{diag}(0,0,0,1) P_1 \\
			&= L^T K + K^T L + \frac{a^2}{\rho^2\sigma_a^2}K^T K\\
			&=(L^T + \frac{a^2}{2\rho^2\sigma_a^2}K^T)K + K^T(L + \frac{a^2}{2\rho^2\sigma_a^2}K).
	\end{aligned}
\end{equation*}
The above matrix is semi-negative definite is equivalent to $(L + \frac{a^2}{2\rho^2\sigma_a^2}K)K^T \leq 0$, which means that
%the following inequality
%\begin{equation}\label{2eq:LK-constrain}
%	\begin{aligned}
%		(L + \frac{a^2}{2\rho^2\sigma_a^2}K)K^T & \leq 0.
%	\end{aligned}
%\end{equation}
%Since
%\begin{equation*}
%	L + \frac{a^2}{2\rho^2\sigma_a^2}K = ((\frac{1}{\theta^2}-\frac{a^2b'}{2\rho\sigma_a})\theta_{\rho },~(\frac{1}{\theta^2}-\frac{a^2b'}{2\rho\sigma_a}) \theta_{\rho v}, ~(\frac{1}{\theta^2}-\frac{a^2b'}{2\rho\sigma_a})\theta_{\rho E},\frac{a^2}{\rho\sigma_a} -1).
%\end{equation*}
%According to the definition of $K$
%\begin{equation*}
%\begin{aligned}
%	&(L + \frac{a^2}{2\rho^2\sigma_a^2}K)K^T\\
%	&= -\rho \sigma_a(\frac{1}{\theta^2}-\frac{a^2b'}{2\rho\sigma_a})b'(\theta_{\rho }^2+\theta_{\rho v}^2+\theta_{\rho E}^2)+2\rho\sigma_a(\frac{a^2}{\rho\sigma_a}-1)\\
%	&=a^2\bigg(2+\frac{b'^2}{2}(\theta_{\rho}^2+\theta_{\rho v}^2+\theta_{\rho E}^2)\bigg)-2\rho\sigma_a-\frac{\rho\sigma_a b'}{\theta^2}(\theta_{\rho }^2+\theta_{\rho v}^2+\theta_{\rho E}^2).
%\end{aligned}
%\end{equation*}
%Thus according to \eqref{2eq:LK-constrain}, we know $a^2$ satisfies 
\begin{equation}\label{2eq:a-constraint3}
	a^2 \leq \rho\sigma_a\frac{4+2 b'(\theta_{\rho }^2+\theta_{\rho v}^2+\theta_{\rho E}^2)/\theta^2}{4 + b'^2(\theta_{\rho }^2+\theta_{\rho v}^2+\theta_{\rho E}^2)}.
\end{equation}

%In summary, 
According to inequalities \eqref{2eq:a-constraint1} and \eqref{2eq:a-constraint3}, 
if the non-zero constant $a$ satisfies the following constraints
\begin{equation*}
	\begin{aligned}
		a^2 \leq \min_{2\leq k\leq N}\bigg\{&2\rho\sigma_a \beta_k(0) \tilde{\kappa}_{k,k}(0),\quad \frac{16}{3}\rho\sigma_a b^2,\quad \rho\sigma_a\frac{4+2 b'(\theta_{\rho }^2+\theta_{\rho v}^2+\theta_{\rho E}^2)/\theta^2}{4 + b'^2(\theta_{\rho }^2+\theta_{\rho v}^2+\theta_{\rho E}^2)}\bigg\},
	\end{aligned}
\end{equation*}
the $P$ matrix defined in \eqref{2eq:radiation-P} satisfies the structural stability condition. And 
\begin{equation*}
	\rho>0, \quad \sigma_a>0, \quad b>0, \quad b'>0, \quad \beta_k(0) \tilde{\kappa}_{k,k}(0)>0.
\end{equation*}
The value space of $a$ is obviously not empty.

Consequently, we conclude the following theorem. 
%In conclusion, we verified that the Euler-HMP$_N$ system \eqref{4eq:radiation_hy_equ} satisfies Yong's structural stability condition. 
\begin{theorem}\label{theorem-stability}
	 The Euler-HMP$_N$ system \eqref{4eq:U_equ} satisfies Yong's structural stability condition which's symmetrizer $A_0$ and $P$ are defined in \eqref{2eq:radiation-A0} and \eqref{2eq:radiation-P}.
\end{theorem}

\section{Non-relativistic limit}\label{4sec}

In this section, we analyze the non-relativistic limit of the radiation hydrodynamics system \eqref{4eq:radiation_hy_equ}. In other word, we focus on singular limits $\varepsilon\rightarrow 0$ of the following system
\begin{equation}\label{4eq:hy_equ}
	\partial_t U + \frac{1}{\varepsilon}A(U;\varepsilon)\partial_x U = \frac{1}{\varepsilon^2} Q(U;\varepsilon).
\end{equation}
Here $U$, $A(U;\varepsilon)$ and $Q(U;\varepsilon)$ are demonstrated in \eqref{4eq:U_equ}.

As we mentioned before, Lattanzio and Yong \cite{lattanzio2001hyperbolic}, Peng and Wasiolek\cite{peng2016parabolic} studied the singular limits of initial-value problems for first-order quasilinear hyperbolic systems with stiff source terms. Under appropriate stability conditions and the existence of approximate solutions, they justified rigorously the validity of the asymptotic expansion on a time interval independent of the parameter. 
However, the system \eqref{4eq:hy_equ} that the coefficient matrix and the source terms both depending on $\varepsilon$ are not considered, which introduce some additional terms.   

For convenience, we rewrite the equations of hydrodynamical variables (the fist three equations in \eqref{4eq:radiation_hy_equ}) as following conservative form 
\begin{equation}\label{4eq:rhov-equ}
	\begin{aligned}
		& \partial_t \rho + \partial_x ( \rho v) = 0, \\
		& \partial_t (\rho v + \varepsilon E_1) + \partial_x ( \rho v^2 +  p +  E_2) = 0, \\
		& \partial_t (\rho E + E_0) + \partial_x (\rho E v +  p v) + \frac{1}{\varepsilon}  \partial_x E_1 = 0.\\
	\end{aligned}
\end{equation}
Here we use the first two moment equations of radiative transfer equation \eqref{moment-equ-rte}:
\begin{equation}\label{equ:moment-E1}
	\begin{aligned}
		&\varepsilon \partial_t E_0 + \partial_x E_1 = S_E,\\
		&\varepsilon \partial_t E_1 + \partial_x E_2 = S_F.
	\end{aligned}
\end{equation}
Owing to the relation \eqref{2-3equ:f-moment-relation}, we know that
\begin{equation*}
	E_0 = \kappa_{0,0}f_0,\qquad E_1 = \kappa_{1,0}f_0,\qquad E_2 = \kappa_{2,2}f_2 + \kappa_{2, 0}f_0.
\end{equation*}

We introduce $\tilde{U} = (\tilde{u}, \tilde{w})^T$ with
\begin{equation}\label{4eq:tilde-u-def}
	\begin{aligned}
		&\tilde{u} = (\rho, ~ \rho v + \varepsilon \kappa_{1,0}f_0, ~ \rho E + \kappa_{0,0}(\alpha)f_0)^T,\\
		&\tilde{w} = (\kappa_{0,0}(\alpha) f_0 - b(\theta), ~ \alpha, ~f_2, ~ \cdots, ~f_N)^T
	\end{aligned}
\end{equation}
and set $\tilde{U}_{eq}=\tilde{U}(U_{eq})$. Then, the systems \eqref{4eq:U_equ} can be rewritten as
\begin{equation}\label{4eq:U_tilde_equ}
	\partial_t \tilde{U} + \frac{1}{\varepsilon}\tilde{A}(\tilde{U};\varepsilon)\partial_x \tilde{U} = \frac{1}{\varepsilon^2}\tilde{Q}(\tilde{U};\varepsilon)
\end{equation}
with
\begin{equation}\label{4eq:a-tilde}
	\begin{aligned}
		\tilde{A}(\tilde{U};\varepsilon) &= D_{U}\tilde{U}\begin{pmatrix}
			\varepsilon F_u(u)&0\\
			0&\tilde{D}^{-1}\tilde{M}\tilde{D}
		\end{pmatrix}(D_{U}\tilde{U})^{-1}, \\
		\tilde{Q}(\tilde{U};\varepsilon) &= D_{U}\tilde{U}Q(U(\tilde{U})) = \begin{pmatrix}
			0 \\ q(\tilde{U};\varepsilon)
			\end{pmatrix}.
	\end{aligned}
\end{equation}
Here, the transformation matrix of $U \rightarrow \tilde{U}$ is 
\begin{equation}\label{4eq:D_U_tilde{U}}
	\begin{aligned}
		D_U\tilde{U} = 
		\begin{pmatrix}
	 		\begin{pmatrix}
				1 & 0 & 0 & 0 & 0\\
				0 & 1 & 0 & \varepsilon \kappa_{1,0} & \varepsilon \kappa_{1,0}'f_0\\
				0 & 0 & 1 & \kappa_{0,0} & \kappa_{0,0}'f_0\\
				-b'\theta_{\rho} & -b'\theta_{\rho v} & -b'\theta_{\rho E} & \kappa_{0,0} & \kappa'_{0,0}f_0\\
				0 & 0 & 0 & 0 & 1
			\end{pmatrix}
			& 0_{5\times(N-1)}\\
			0_{(N-1)\times 5} & I_{(N-1)\times (N-1)}
		\end{pmatrix}.
	\end{aligned}
\end{equation}
Using formulas \eqref{2lemma:kappa}, a routine computation gives rise to the determination of $D_U\tilde{U}$ is
\begin{equation*}
	\begin{aligned}
		\mathrm{det}(D_U\tilde{U}) &= \kappa_{0,0} + b'\theta_{\rho E}\kappa_{0,0} + \varepsilon b'\theta_{\rho v}\kappa_{1,0}\\
		&=\bigg((3+\alpha^2)(1+b'\theta_{\rho E})-\varepsilon 4\alpha b'\theta_{\rho v}\bigg)\frac{2}{3(1-\alpha^2)^3}.
	\end{aligned}
\end{equation*}
Note that $\alpha \in (-1,1)$ and $\theta_{\rho E}>0$ \cite{rohde2008dissipative}. Then when $\varepsilon=0$, we have
\begin{equation*}
	\mathrm{det}(D_U\tilde{U})= \frac{2(3 +\alpha^2)}{3(1-\alpha^2)^3} (1+b'\theta_{\rho E})>0.
\end{equation*}
Thus, there exits $\varepsilon_0>0$ such that $\mathrm{det}(D_U\tilde{U}) \neq 0$ for  $\varepsilon \in [0,\varepsilon_0]$. Therefore, we assume $\varepsilon \in [0, \varepsilon_0]$ for the system \eqref{4eq:U_tilde_equ}.

The system \eqref{4eq:U_tilde_equ} also satisfies Yong's structural stability condition. 
It is apparent from \eqref{2eq:radiation-equilibrium} that $\tilde{w}=0$ on equilibrium manifold.
On the equilibrium manifold, we have 
\begin{equation*}
	\begin{aligned}
		\partial_{\tilde{U}}\tilde{Q} = \partial_{\tilde{U}}\bigg(D_{U}\tilde{U}Q\bigg) = \partial_{U}(D_{U}\tilde{U})D_{U}\tilde{U}Q + D_{U}\tilde{U}Q_{U}(D_{U}\tilde{U})^{-1}  = D_{U}\tilde{U}Q_{U}(D_{U}\tilde{U})^{-1}.
	\end{aligned}
\end{equation*}
Here we use $Q(U_{eq})=0$.
Set $\tilde{P}=P (D_{U}\tilde{U})^{-1}(\tilde{U}_{eq};0)$, in which $P$ expressed in \eqref{2eq:radiation-P}. On the equilibrium manifold, we see that
\begin{equation*}
	\begin{aligned}
		&\kappa_{0,0}(0)=2,\qquad \kappa'_{0,0}(0)=0.\\
%		&\kappa_{1,0}(0)=0,\qquad \kappa'_{1,0}(0)=-\frac{8}{3},\qquad f_0 = \frac{1}{2}b(\theta).
	\end{aligned}
\end{equation*}
A straightforward calculation gives rise to $P=a D_{U}\tilde{U}(\tilde{U}_{eq};0)$, which $a$ is the constant in \eqref{2eq:radiation-P}. Thus $\tilde{P}$ is a scalar matrix. Moreover,
\begin{equation*}
	\tilde{P}\tilde{Q}_{\tilde{U}} \tilde{P}^{-1} = P (D_{U}\tilde{U})^{-1} D_{U}\tilde{U}Q_{U}(D_{U}\tilde{U})^{-1}D_{U}\tilde{U}P^{-1} = PQ_{U}P^{-1}.
\end{equation*}
This means that system \eqref{4eq:U_tilde_equ} satisfies the first requirement of structural stability condition. For the second requirement, the symmetrizer of system \eqref{4eq:U_tilde_equ} is $\tilde{A_0} = \tilde{A_0}(\tilde{U};\varepsilon) = (D_{U}\tilde{U})^{-T}A_0 (D_{U}\tilde{U})^{-1}$.
A simple computation shows that the system also satisfies third requirement of structural stability condition. 

For further discussions, we analyze $\tilde{A}(\tilde{U};0)$. In Appendix \ref{appendix}, we show that 
\begin{equation*}
	\tilde{A}^{11}(\tilde{U}_{eq};0) =0,\qquad \partial_{\tilde{u}} \tilde{A}^{11}(\tilde{U}_{eq};0)=0,
\end{equation*}
and $\tilde{A}^{21}(\tilde{U}_{eq};0)$ is not full--rank matrix. See details in Appendix \ref{appendix}.

For the convenience of writing, we omit the superscript below. Then the system \eqref{4eq:U_tilde_equ} have the following form
\begin{equation}\label{4eq:keyequ}
		\partial_t U + \frac{1}{\varepsilon}A(U;\varepsilon)\partial_x U = \frac{1}{\varepsilon^2}Q(U;\varepsilon),
\end{equation}
with initial conditions
\begin{equation}\label{initial-condition-key}
		U(x,0) = \bar{U}(x,\varepsilon).
\end{equation}
Here $U =(u, w)^T \in G \subset \mathrm{R}^{N+4}$, and
\begin{equation}\label{4eq:u-definition}
	u = (\rho, ~\rho v + \varepsilon \kappa_{1,0}f_0, ~\rho E + \kappa_{0,0}f_0)^T\in\mathrm{R}^{3}, \quad w = (\kappa_{0,0} f_0 - b(\theta), ~\alpha, ~f_2, ~\cdots, ~f_N)^T\in\mathrm{R}^{N+1}.
\end{equation}
Here $A=A(U;\varepsilon)$ and $Q=Q(U;\varepsilon)$ are the respective $n\times n$--matrix function and $n$--vector functions of $(U;\varepsilon)\in G\times [0, \varepsilon_0]$. The parameter $\varepsilon \in [0, \varepsilon_0]$. The state space $G$ is a open convex set, which defined as
\begin{equation*}
	G=\bigg \{U =(u, w) : \rho > 0, ~\theta > 0, ~\alpha \in (-1,1) \bigg \}.
\end{equation*} 
The equilibrium manifold is
\begin{equation*}
	G_{eq}=\bigg\{ U \in G : w = 0 \bigg\}.
\end{equation*} 

\begin{lemma}\label{4lemma1}
The system \eqref{4eq:keyequ} satisfies following properties.
	\begin{enumerate}[(i)]
		\item The source term has the following form: 
		\begin{equation*}
		Q(U;\varepsilon) = \begin{pmatrix}
			0 \\ q(U;\varepsilon)
		\end{pmatrix},\qquad q(U;\varepsilon) \in \mathrm{R}^{(N+1)},
		\end{equation*}
			and 
		\begin{equation*}
			q(U;\varepsilon) = 0 \Leftrightarrow w = 0,
		\end{equation*}
		for all $u$
		\begin{equation*}
		\partial_u q(U_{eq};0) = 0, \qquad \partial_w q(U_{eq};0) ~ \mbox{invertible}
		\end{equation*}
		\item The system \eqref{4eq:keyequ} satisfies Yong's structural stability condition and $P$ is a scalar matrix;
		\item $A^{11}(U_{eq};0) = 0$ and $\partial_u A^{11}(U_{eq};0) = 0$ for all $U_{eq}\in G_{eq}$;
		\item $A_0(U_{eq};0) = \mathrm{diag}(A^{11}_0(U_{eq};0),~A^{22}_0(U_{eq};0))$ is a block diagonal matrix.
	\end{enumerate}		
\end{lemma}
\begin{proof}
Obviously, the first two terms are clearly established. The proof of the third term is exhibited in Appendix \ref{appendix}.
As shown in Theorem 2.2 in \cite{yong1999singular}, the structural stability condition imply that $P^{-T} A_0(U_{eq};0)P^{-1}$ is a block diagonal matrix. Moreover, $P$ is a scalar matrix, so $A_0(U_{eq};0)$ is a block diagonal matrix.
\end{proof}

Through the previous discussion, we can see that the coefficient matrix $A(U;\varepsilon)$ and the source term $Q(U;\varepsilon)$ are smoothly dependent on $U$ and $\varepsilon$. The symmetrizer $A_0$ is also a smooth function of $U$ and $\varepsilon$.

Assuming that the initial value of the equation is periodic and smooth, according to Kato  \cite{kato1975cauchy}, for all integer $s>\frac{3}{2}$, there exists a maximal time $T_\varepsilon>0$ such that problem \eqref{4eq:keyequ}--\eqref{initial-condition-key} admits a unique local-in-time smooth solution $U^\varepsilon$ satisfying 
\begin{equation*}
	U^\varepsilon \in \mathrm{C}([0,T_\varepsilon),H^s)\cap \mathrm{C}^1([0,T_\varepsilon),H^{s-1}).
\end{equation*}
The central problem of the study is to show that $U^\varepsilon$ converges as $\varepsilon \rightarrow 0$ and $\inf T_\varepsilon>0$. To do this, we study the approximate solution of \eqref{4eq:keyequ}.  

We end this section with stating several calculus inequalities in Sobolev spaces \cite{majda2012compressible}, two elementary facts \cite{yong1999singular} related to ordinary differential equations and the notation involved in this paper. Their proofs can be found in \cite{yong1999singular} and references cited therein.

\begin{lemma}[Calculus inequalities]\label{4lemma:calculus_inequ}
	Let $s, s_1,$ and $s_2$ be three nonnegative integers, and $s_0=[D/2]+1$.
	\begin{enumerate}
		\item If $s_3=min\{s_1, s_2,s_1+s_2-s_0\}\geq 0$, then $H^{s_1}H^{s_2}\subset H^{s_3}$. Here the inclusion symbol $\subset$ implies the continuity of the embedding.
		\item Suppose $s>s_0+1, A\in H^s$, and $Q\in H^{s-1}$, Then for all multi-indices $\alpha$ with $\alpha \leq s$, $[A, \partial_\alpha ]Q\equiv A\partial^\alpha Q-\partial^\alpha (A Q)\in L^2$ and
			\begin{equation*}	
				\|A\partial^\alpha Q-\partial^\alpha (A Q)\|\leq C_s\|A\|_s\|Q\|_{|\alpha |-1};
			\end{equation*}
		\item Suppose $s>s_0$, $A\in C_b^s(G)$ and $V\in H^{s}(R^d, G)$. Then $A(V( \cdot))\in H^s$ and
				\begin{equation*}	
					\|A(V(\cdot))\|_s\leq C_s|A|_s(1+\|V\|_s^s).
				\end{equation*}
	\end{enumerate}
\end{lemma}
Here and below $C_s$ denotes a generic constant depending only on $s, n$ and $D$, and $|A|_s$ stand for $\sup_{u\in G_0,|\alpha|\leq s}|\partial^\alpha_u A(u)|$.

\begin{lemma}\cite{yong1999singular}\label{4lemma:decay}
	Suppose $A(x,\tau)\in C([0,\infty),H^s)$ with $s>\frac{3}{2}$, $f(x,\tau)\in C([0,\infty),L^2)$, $\parallel f(\tau) \parallel$ decays exponentially to zero as $\tau$ goes to infinity, and $E(x),S(x)\in L^\infty$ are uniformly positive definite symmetric matrices such that for all sufficiently large $\tau$ and fo all $x$, 
	\begin{equation*}
		E(x)A(x,\tau)+A^T(x,\tau)E(x)\leq -S(x).
	\end{equation*}
	If $V(x,\tau)\in C^1([0,\infty),L^2)$ satisfies
	\begin{equation*}
		\frac{\mathrm{d} V}{\mathrm{d} \tau} = A(x,\tau)V+f(x,\tau),
	\end{equation*}
	then $\parallel V\parallel$ decays exponentially to zero as $\tau$ goes to infinity. Moreover, if $V(x,\tau),f(x,\tau)\in C([0,\infty),H^s)$ and $\parallel f(\tau) \parallel_s$ decays exponentially to zero as $\tau$ goes to infinity, then $\parallel V\parallel_s$ decays exponentially to zero as $\tau$ goes to infinity.
\end{lemma}

\begin{lemma}\cite{yong1999singular}\label{4lemma:gronwall}
	Suppose $\psi(t)$ is a positive $C^1$--function of $t\in [0,T)$ with $T\leq \infty$, $m>1$ and $b_1(t),b_2(t)$ are integrable on $[0,T)$. If
	\begin{equation*}
		\psi'(t) \leq b_1(t)\psi^m(t) + b_2(t)\psi(t),
	\end{equation*}
	then there exists $\delta >0$, depending only on $m$, $C_{1b}$ and $C_{2b}$, such that
	\begin{equation*}
		\sup_{t\in[0,T)}\psi(t) \leq e^{C_{1b}},
	\end{equation*}
	whenever $\psi(0)\in (0,\delta]$. Here
	\begin{equation*}
		C_{1b}=\sup_{t\in [0,T)}\int_0^t b_1(s)\mathrm{d} s,\qquad C_{2b}=\int_0^T \max\{b_2(t),0\}\mathrm{d}t.
	\end{equation*}
\end{lemma}

\begin{notation}
	The superscript $'T'$ denotes the transpose of a vector or matrix. $|U|$ denotes some norm of a vector or matrix. $L_2=L_2(\Omega)$ is the space of square integrable (vector- or matrix-valued) functions on $\Omega$. For a non-negative integer s, $H_s=H_s(\Omega)$ is defined as the space of functions whose distribution derivatives of order $\leq s$ are all in $L_2$. We use $\parallel U\parallel_s$ to denote the standard norm of $U\in H_s$, and $\parallel U\parallel=\parallel U\parallel_0$. When $A$ is a function of another variable $t$ as well as $x$, we write $\parallel A(\cdot, t)\parallel_s$ to recall that the norm is taken with respect to $s$ while $t$ is viewed as a parameter. In addition, we denote by $C([0, T], X)$ the space of continuous functions on $[0, T]$ with values in a Banach space $X$.
\end{notation}

\subsection{Formal asymptotic expansions}\label{4-1sec}

We construct such an approximate solution for the equation \eqref{4eq:keyequ} by an asymptotic expansion with initial layer corrections of the form
\begin{equation}\label{4eq:U-expansion}
	U_\varepsilon^m = \sum_{k=0}^m \varepsilon^k\bigg(U_k(x, t)+\mathrm{I}_k(x, \tau)\bigg),\qquad  m\in \mathrm{N},
\end{equation}
where $\tau = t/\varepsilon^2$  is a fast time.  Here $\sum_{k=0}^m \varepsilon^k U_k(x, t)$ is the outer expansion and $\sum_{k=0}^m \varepsilon^k \mathrm{I}_k(x, \tau)$ is the initial-layer correction. As a correction,  $\mathrm{I}_k(x, \tau)$ will be significant only near $t=0$. Thus the $\mathrm{I}_k(x, \tau)$ are required to decay to zero as $\tau$ goes to infinity, since the latter happens as $\varepsilon$ tends to zero whenever $t \geq \delta>0$ with $\delta$ arbitrary but fixed. This natural requirement is similar to the traditional matching principle in \cite{eckhaus1977matching}. Once the outer expansion and the initial-layer correction are found, the formal asymptotic approximation is defined as the above truncation \eqref{4eq:U-expansion}.
%
%Throughout this paper, $s> \frac{3}{2}$ is an integer and $C>0$ stands for a generic constant independent of $\varepsilon$. 
We assume there exists an approximate solution $U_\varepsilon^m$ to \eqref{4eq:keyequ}--\eqref{initial-condition-key} defined on a time interval $[0, T_m]$, with $T_m>0$ independent of $\varepsilon$.

The properties of the approximate solution strongly depend on its leading profile  $(u_0, w_0)$, which is a formal limit of $U_\varepsilon^m$. From the eqautions \eqref{4eq:U-expansion} and \eqref{4eq:keyequ}, we can obtain 
\begin{equation*}
	\begin{aligned}
		\varepsilon^{-2}: &q(u_0, w_0; 0) = 0,\\
		\varepsilon^{-1}: &\begin{pmatrix}
			A^{11}(u_0, w_0;0) & A^{12}(u_0, w_0;0)\\
			A^{21}(u_0, w_0;0) & A^{22}(u_0, w_0;0)
		\end{pmatrix}\partial_x \begin{pmatrix}
			u_0\\w_0
		\end{pmatrix} = \begin{pmatrix}
			0\\
			q_w(u_0, w_0;0)w_1
		\end{pmatrix},\\
		\varepsilon^0 ~: &\partial_t u_0 + A^{11}(u_0, w_0;0)\partial_x u_1 + A^{12}(u_0, w_0;0)\partial_x w_1 \\
		&+ \bigg(A^{11}_u(u_0, w_0;0)u_1 + A^{11}_w(u_0, w_0;0)w_1 + A^{11}_\varepsilon(u_0, w_0;0)\bigg)\partial_x u_0 = 0.
	\end{aligned}
\end{equation*}
According to the Lemma \ref{4lemma1}, we have
\begin{equation*}
	\begin{aligned}
		w_0 = 0, \qquad w_1 = q_w^{-1}(u_0, 0;0)A^{21}(u_0, 0;0)\partial_x u_0,\\
%		\partial_t u_0 + A^{11}_w(u_0, 0;0)w_1 + A^{11}_\varepsilon(u_0, 0;0)\partial_x u_0 = 0.\\
		\partial_t u_0 + A^{12}(u_0, 0;0)\partial_x w_1 + A^{11}_w(u_0, 0;0)w_1\partial_x u_0  + A^{11}_\varepsilon(u_0, 0;0)\partial_x u_0 = 0.
	\end{aligned}
\end{equation*}
Here we use the properties that $A^{11}(U_{eq};0) = 0$ and $\partial_u A^{11}(U_{eq};0) = 0$ for all $U_{eq}\in G_{eq}$.
Applying the relation of $w_1$ to the equation of $u_0$, we obtain  
\begin{equation}\label{4eq:u0-equ}
	\begin{aligned}
%		\partial_t u_0 &+ \bigg(A^{11}_w(U_0;0)q_w^{-1}(U_0;0)A^{21}(U_0;0) + A^{11}_\varepsilon(U_0;0)\bigg)\partial_x u_0 = 0
		\partial_t u_0 & + A^{12}(U_0;0)\partial_x \bigg(q_w^{-1}(U_0;0)A^{21}(U_0;0)\partial_x u_0 \bigg)\\
		 &+ \bigg(A^{11}_w(U_0;0)q_w^{-1}(U_0;0)A^{21}(U_0;0)\partial_x u_0 + A^{11}_\varepsilon(U_0;0)\bigg)\partial_x u_0=0.
	\end{aligned}
\end{equation}
The equation \eqref{4eq:u0-equ} can be rewritten as
\begin{equation*}
	\begin{aligned}
		\partial_t u_0 &+ A^{12}(U_0;0)q_w^{-1}(U_0;0)A^{21}(U_0;0)\partial^2_{xx}u_0\\
		& + A^{12}(U_0;0)\partial_u \bigg(q_w^{-1}(U_0;0)A^{21}(U_0;0)\bigg) \partial_x u_0\\
		& + \bigg(A^{11}_w(U_0;0)q_w^{-1}(U_0;0)A^{21}(U_0;0)\partial_x u_0 + A^{11}_\varepsilon(U_0;0)\bigg)\partial_x u_0= 0.
	\end{aligned}
\end{equation*}
Since $A^{21}(U_0;0)$ is not full-rank matrix according to Appendix \ref{appendix}, we know that the equation of $u_0$ \eqref{4eq:u0-equ} is not strictly parabolic. Its proof is quite similar to those proved in \cite{lattanzio2001hyperbolic,peng2016parabolic}.
%$A^{12}(U_0;0)q_w^{-1}(U_0;0)A^{21}(U_0;0)$ is a semi-negative definite matrix. Consequently, the equation of $u_0$ \eqref{4eq:u0-equ} is not strictly parabolic.

Here we derive the specific form of the equation which $u_0$ satisfies. 
%To this end, the variables involved in the equation \eqref{4eq:rhov-equ} are expanded according to $\varepsilon$. Then we obtain
Expanding the variables into a power series of $\varepsilon$ which involved in the equation \eqref{4eq:rhov-equ} yields 
\begin{equation}\label{4eq:expansion-u0}
	\begin{aligned}
		\rho &= \rho^0 + \varepsilon \rho^1 + \cdots, \quad && v = v^0 + \varepsilon v^1 + \cdots,\\
		E &= E^0 + \varepsilon E^1 + \cdots, && \theta = \theta^0 + \varepsilon \theta^1 + \cdots,\\
		p &= p^0+ \varepsilon p^1 + \cdots, && f_0 =f_0^0 + \varepsilon f_0^1 +\cdots,\\
		\alpha &= \alpha^0 + \varepsilon \alpha^1 + \cdots, && f_2 = f_2^0 + \varepsilon f_2^1+\cdots,
	\end{aligned}
\end{equation}
where $\theta = \theta(\rho, v, E)$, $p = p(\rho,\theta)$.
%$\theta^0 = \theta(\rho^0, v^0, E^0)$, $p^0 = p(\rho^0,\theta^0)$, $\theta^1$ and $p^1$ are the corresponding first-order terms. 
According to the definition of equilibrium state in $\eqref{2eq:radiation-equilibrium}$, we know that $f_0^0 = \frac{1}{2}b(\theta^0)$, $\alpha^0 = 0$, $f^0_2 = 0$.

%From the definition \eqref{4eq:u-definition} of $u$, we know $u_0 = (\rho^0, \rho^0 v^0, \rho^0 E^0+\kappa_{0,0}(0)f_0^0)$.
Using the equation \eqref{4eq:rhov-equ}, we arrive at 
\begin{equation}\label{4eq:rho-0-equ}
	\begin{aligned}
			&\partial_t \rho^0 + \partial_x(\rho^0 v^0)=0,\\
			&\partial_t (\rho^0 v^0) + \partial_x\bigg(\rho^0 (v^0)^2 + p^0 +  \kappa_{2,2}(0)f_2^0 +  \kappa_{2,0}(0)f_0^0\bigg)=0,\\
			&\partial_t \bigg(\rho^0 E^0 + \kappa_{0,0}(0)f_0^0\bigg) + \partial_x \bigg(\rho^0 E^0 v^0 + p^0 v^0 + \kappa'_{1,0}(0)\alpha^1 f_0^0 + \kappa_{1,0}(0)f_0^1\bigg)=0.
	\end{aligned}
\end{equation}
Here $\kappa_{2,0}(0)=\frac{2}{3}$, $\kappa_{0,0}(0)=2$, $\kappa'_{1,0}(0)=-\frac{8}{3}$, $\kappa_{1,0}(0)=0$ due to the formulas \eqref{2lemma:kappa}. To get a closed system, we also need the expression of $\alpha^1$.

In order to obtain the expression of $\alpha^1$, we analyze equation of $E_1$ in \eqref{equ:moment-E1}:
\begin{equation*}
	\partial_{t}(\kappa_{1,0}f_0) + \frac{1}{\varepsilon} \partial_x(\kappa_{2,2}f_2+\kappa_{2,0}f_0) = -\rho\bigg(\frac{1}{\varepsilon^2}\sigma_a(\theta)+\sigma_s(\theta)\bigg)\kappa_{1,0}(\alpha)f_0.
\end{equation*}
Putting the expansion \eqref{4eq:expansion-u0} into above equation, the identification of $O(\varepsilon^{-1})$ yields
%In the above equation, the variables are expanded according to $\varepsilon$, and the coefficient of $\varepsilon^{-1}$ is zero, we have 
%\begin{equation*}
%	\begin{aligned}
%		\partial_x(\kappa_{2,2}(0)f_2^0+\kappa_{2,0}(0)f_0^0)=&-\rho^0 \sigma_a(\theta^0) \bigg(\kappa'_{1,0}(0)\alpha^1f_0^0 + \kappa_{1,0}(0)f_0^1\bigg)\\
%		&- \rho^1 \sigma_a(\theta^0) \kappa_{1,0}(0) f_0^0 - \rho^0 \sigma'_{a}\theta^1\kappa_{1,0}(0)f_0^0.
%	\end{aligned}
%\end{equation*}
\begin{equation*}
	\begin{aligned}
		\partial_x(\kappa_{2,0}(0)f_0^0)=&-\rho^0 \sigma_a(\theta^0) \bigg(\kappa'_{1,0}(0)\alpha^1f_0^0 \bigg).
	\end{aligned}
\end{equation*}
Here we used $f_2^0=0$ and $\kappa_{1,0}(0)=0$.
%Here $\kappa_{2,0}(0)=\frac{2}{3}$, $f_0^0 = \frac{1}{2}b(\theta^0)$. 
Combining $\kappa_{2,0}(0)=\frac{2}{3}$ and $f_0^0 = \frac{1}{2}b(\theta^0)$, we see that 
%The equation \eqref{4eq:rho-0-equ} involves the term of $\alpha^1$ as $\kappa'_{1,0}(0)\alpha^1f_0^0$, then we can obtain 
\begin{equation*}
	\kappa'_{1,0}(0)\alpha^1f_0^0 =-\frac{1}{3\rho^0\sigma_a(\theta^0)} \partial_x(b(\theta^0)).
\end{equation*}

%Then we can obtain the $\varepsilon \rightarrow 0$ equation, which means that $u_0$ satisfies
%\begin{equation*}
%	\begin{aligned}
%			&\partial_t \rho^0 + \partial_x(\rho^0 v^0)=0,\\
%			&\partial_t (\rho^0 v^0) + \partial_x\bigg(\rho^0 (v^0)^2 + p^0 + \frac{1}{3}b(\theta^0) \bigg)=0,\\
%			&\partial_t \bigg(\rho^0 E^0 + b(\theta^0) \bigg) + \partial_x \bigg(\rho^0 E^0 v^0 + p^0 v^0 -\frac{1}{3\rho^0\sigma_a(\theta^0)} \partial_x(b(\theta^0)) \bigg)=0.
%	\end{aligned}
%\end{equation*}
Omit the superscript in above equation, the non-relativistic limit equation can be obtained as
\begin{equation*}
	\begin{aligned}
			&\partial_t \rho + \partial_x(\rho v)=0,\\
			&\partial_t (\rho v) + \partial_x\bigg(\rho v^2 + p + \frac{1}{3}b(\theta) \bigg)=0,\\
			&\partial_t \big(\rho E + b(\theta) \big) + \partial_x \big(\rho E v + p v \big) = \partial_x \bigg(\frac{1}{3\rho\sigma_a(\theta)} \partial_x b(\theta) \bigg).
	\end{aligned}
\end{equation*}
In \cite{buet2004asymptotic,lowrie1999coupling,ferguson2017equilibrium}, the authors also obtain the zero-order approximation of the radiation hydrodynamics system and the system is also hyperbolic-parabolic form which is  similar to above equation. However, there is no rigorous proof of the singular limit.
%\todo{This is similar to the equation satisfied by the zero-order approximation of the radiation hydrodynamics system obtained in the literature \cite{buet2004asymptotic,lowrie1999coupling,ferguson2017equilibrium}.}

Below we derive the equations satisfied by the other coefficients in the asymptotic solution \eqref{4eq:U-expansion}. To do this, we consider the residual
\begin{equation}\label{4eq:asp-err}
	R(U_\varepsilon^m) = \partial_t U_\varepsilon^m + \frac{1}{\varepsilon}A(U_\varepsilon^m;\varepsilon)\partial_x U_\varepsilon^m - \frac{1}{\varepsilon}Q(U_\varepsilon^m;\varepsilon).
\end{equation}
Using Tayler expansion, we have
\begin{equation*}
	\begin{aligned}
		A(U_\varepsilon^m;\varepsilon) = A(U_\varepsilon^m;0)+\sum_{l=1}^{+\infty}\varepsilon^l \partial_\varepsilon^l A(U_\varepsilon^m;0),\\
		Q(U_\varepsilon^m;\varepsilon) = Q(U_\varepsilon^m;0)+\sum_{l=1}^{+\infty}\varepsilon^l \partial_\varepsilon^l Q(U_\varepsilon^m;0).
	\end{aligned}
\end{equation*}
Remark that for $W = \sum_{k=0}^{+\infty}\varepsilon^k W_k$ and a sufficiently smooth function $H$, we have formally \cite{lattanzio2001hyperbolic}
\begin{equation*}
	H(W) = H(\sum_{k=0}^{+\infty}\varepsilon^k W_k) = H(W_0) + \sum_{k=1}^{+\infty}\varepsilon^k [\partial_W H(W_0)W_k + \mathrm{C}(H,k,\underline{W})],
\end{equation*}
where coefficients $\mathrm{C}(H,k,\underline{W})$ are completely determined by the given function $H$ and the first $k$ components $\underline{W}=(W_0, W_1,W_2,\cdots, W_{k-1})$. Moreover, $\mathrm{C}(H,1,\underline{W}) = 0$ and $\mathrm{C}(H,k,\underline{W})$ is linear with respect to $W_{k-1}$ for $k\geq 3$.

\subsubsection{Outer Expansions}
As a formal solution, the outer expansion $\sum_{k=0}^\infty \varepsilon^k U_k(x, t)$ asymptotically satisfies the system \eqref{4eq:keyequ}. Thus, we have
\begin{equation}\label{4eq:outer-expansion}
	\begin{aligned}
		&R(\sum_{k=0}^\infty \varepsilon^k U_k)\\
		&= -\varepsilon^{-2} Q(U_0;0) + \varepsilon^{-1}\bigg[A(U_0;0)\partial_x U_0 - \partial_\varepsilon Q(U_0;0) - Q_U(U_0;0)U_1\bigg]\\
		&~ + \sum_{k=0}^\infty \varepsilon^k \partial_t U_k + \sum_{k=0}^\infty \varepsilon^k \sum_{l=0}^{k+1} \frac{1}{l!}\partial_\varepsilon^l A(U_0;0)\partial_x U_{k+1-l}\\
		&~ + \sum_{k=0}^\infty \varepsilon^k \sum_{l=0}^{k}\sum_{j=0}^{k-l} \frac{1}{l!}\bigg[\partial_U(\partial_\varepsilon^l A(U_0;0))U_{k+1-l-j} + \mathrm{C}(\partial_\varepsilon^l A(\cdot~;0),k+1-l-j,\underline{U})\bigg]\partial_x U_{j}\\
		&~ - \sum_{k=0}^\infty \varepsilon^k \sum_{l=0}^{k+1}\frac{1}{l!}[\partial_U(\partial_\varepsilon^lQ(U_0;0))U_{k+2-l} + \mathrm{C}(\partial_\varepsilon^l Q(\cdot~;0);k+2-l,\underline{U})] \\
		&~ - \sum_{k=0}^\infty \varepsilon^k \frac{1}{(k+2)!}\partial_\varepsilon^{k+2}Q(U_0;0)
	\end{aligned}
\end{equation}
vanishes. This happens when each term of the last expansion is zero, i.e.,
\begin{equation}\label{eq:equ0}
	\begin{aligned}
		\varepsilon^{-2}: &Q(U_0;0) = 0,\\
		\varepsilon^{-1}: &A(U_0;0)\partial_x U_0 - \partial_\varepsilon Q(U_0;0) - Q_U(U_0;0)U_1=0,\\
		\varepsilon^k ~: &\partial_t U_k + \sum_{l=0}^{k+1} \frac{1}{l!}\partial_\varepsilon^lA(U_0;0)\partial_x U_{k+1-l} \\
		&+ \sum_{l=0}^{k}\sum_{j=0}^{k-l} \frac{1}{l!}\bigg[\partial_U(\partial_\varepsilon^l A(U_0;0))U_{k+1-l-j} + \mathrm{C}(\partial_\varepsilon^l A(\cdot~;0),k+1-l-j,\underline{U})\bigg]\partial_x U_{j}\\
		&= \sum_{l=0}^{k+1}\frac{1}{l!}[\partial_U(\partial_\varepsilon^lQ(U_0;0))U_{k+2-l} + \mathrm{C}(\partial_\varepsilon^l Q(\cdot~;0) ;k+2-l,\underline{U})]\\
		&+\frac{1}{(k+2)!}\partial_\varepsilon^{k+2}Q(U_0;0).
	\end{aligned}
\end{equation}
According to Lemma \eqref{4lemma1}, $u_0,w_0$ and $w_1$ satisfy
\begin{equation}\label{4eq:equ1}
	\begin{aligned}
		&Q(u_0,w_0;0)=0 \Rightarrow w_0 = 0, \qquad \partial_\varepsilon Q(u_0,0;0)=0,\\
		&w_1 = q_w^{-1}(u_0, 0;0)A^{21}(u_0, 0;0)\partial_x u_0,\\
		&\partial_t u_0 + A^{12}(u_0, 0;0)\partial_x w_1 + A^{11}_w(u_0, 0;0)w_1\partial_x u_0  + A^{11}_\varepsilon(u_0, 0;0)\partial_x u_0 = 0.
	\end{aligned}
\end{equation}
The above equations can be rewritten with the $u$--, $w$--components as
\begin{equation}\label{4eq:equ2}
	\begin{aligned}
		&\partial_t u_k + \sum_{l=0}^{k+1} \frac{1}{l!}\partial_\varepsilon^l( A^{11}(U_0;0)\partial_x u_{k+1-l}+ A^{12}(U_0;0)\partial_x w_{k+1-l})\\
		& + \sum_{l=0}^k\sum_{j=0}^{k-l}\frac{1}{l!}\bigg[\partial_u(\partial_\varepsilon^l A^{11}(U_0;0))u_{k+1-l-j}\partial_x u_{j} + \partial_w(\partial_\varepsilon^l A^{11}(U_0;0))w_{k+1-l-j}\partial_x u_{j} \\
		&+ \partial_u(\partial_\varepsilon^l A^{12}(U_0;0))u_{k+1-l-j}\partial_x w_{j} + \partial_w(\partial_\varepsilon^l A^{12}(U_0;0))w_{k+1-l-j}\partial_x w_{j} \\
		&+ \mathrm{C}(\partial_\varepsilon^l A^{11}(\cdot~;0),k+1-l-j,\underline{U})\partial_x u_{j} + \mathrm{C}(\partial_\varepsilon^l A^{12}(\cdot~;0),k+1-l-j,\underline{U})\partial_x w_{j} \bigg]=0.
	\end{aligned}
\end{equation}
and
\begin{equation}\label{4eq:equ3}
	\begin{aligned}
		&\partial_t w_k + \sum_{l=0}^{k+1} \frac{1}{l!}\partial_\varepsilon^l( A^{21}(U_0;0)\partial_x u_{k+1-l}+ A^{22}(U_0;0)\partial_x w_{k+1-l})\\
		& + \sum_{l=0}^k\sum_{j=0}^{k-l}\frac{1}{l!}\bigg[\partial_u(\partial_\varepsilon^l A^{21}(U_0;0))u_{k+1-l-j}\partial_x u_{j} + \partial_w(\partial_\varepsilon^l A^{21}(U_0;0))w_{k+1-l-j}\partial_x u_{j} \\
		&+ \partial_u(\partial_\varepsilon^l A^{22}(U_0;0))u_{k+1-l-j}\partial_x w_{j} + \partial_w(\partial_\varepsilon^l A^{22}(U_0;0))w_{k+1-l-j}\partial_x w_{j} \\
		&+ \mathrm{C}(\partial_\varepsilon^l A^{21}(\cdot~;0),k+1-l-j,\underline{U})\partial_x u_{j} + \mathrm{C}(\partial_\varepsilon^l A^{22}(\cdot~;0),k+1-l-j,\underline{U})\partial_x w_{j} \bigg]\\
		&-\sum_{l=0}^{k+1}\frac{1}{l!}[\partial_u(\partial_\varepsilon^l q(U_0;0))u_{k+2-l} + \partial_w(\partial_\varepsilon^l q(U_0;0))w_{k+2-l} + \mathrm{C}(\partial_\varepsilon^l q(\cdot~;0) ;k+2-l,\underline{U})]\\
		&-\frac{1}{(k+2)!}\partial_\varepsilon^{k+2}q(U_0;0) = 0.
	\end{aligned}
\end{equation}

Obviously, the equations in \eqref{eq:equ0} need to be rewritten to determine $U_k$ inductively. Equation \eqref{4eq:equ1} shows that $U_0$ lies on the equilibrium manifold $G_{eq}$. According to equations \eqref{4eq:equ1}, we have found the equations for $u_0$, $w_0$ and $w_1$. From the Lemma \eqref{4lemma1}, we know $A^{11}(U_0;0)=0$ and $w_0 =0$. Hence the equations of $u_k$  \eqref{4eq:equ2} may depend on $U_0,\cdots, U_k$, $w_{k+1}$ and their first-order derivatives, but are independent of $u_{k+1}$. From the equations \eqref{4eq:equ3}, we can see $w_k$ depend on $U_0, \cdots, U_{k+1}$ and $w_{k+2}$. The equations of $w_k$ are independent of $u_{k+2}$ due to the fact: since  $q(U_0;0)=0$, we know $\partial_u(\partial_\varepsilon^l q(U_0;0))u_{k+2-l}=0$ when $l=0$. Moreover, $\partial_w(\partial_\varepsilon^l q(U_0;0))w_{k+2-l}=\partial_w q(U_0;0) w_{k+2}$ when $l=0$. Therefore, \eqref{4eq:equ3} give an expression of $w_{k+2}$ as a function of $U_0,\cdots, U_{k+1}$ and of the known quantities and their derivatives. 

Up to now, we have found the equations for $u_0$,$w_0$ and $w_1$. Assume inductively that we have equations for $u_i$, $w_i$ and $w_{i+1}$ for $i=0,\cdots,k$. The equations \eqref{4eq:equ3} gives an expression of $w_{i+2}$ of function of $u_{k+1}$, $\partial_x u_{k+1}$ and of the known quantities and their derivatives. With this expression, the equation for $u_{k+1}$ can be derived from the relation \eqref{4eq:equ2}.

Assume $U_0,\cdots,U_{k-1}$ are known. From equation \eqref{4eq:equ2}, we know that the equations of $u_k$ can be rewritten as
\begin{equation*}
	\partial_t u_k + A^{12}(U_0;0)\partial_x w_{k+1} + \cdots =0.
\end{equation*}
What is omitted here and in the following equations is the derivative term of the known quantity and the known quantity $U_0,\cdots, U_{k-1}$. \eqref{4eq:equ3} allows to express $w_{k+1}$ as
\begin{equation*}
	w_{k+1} = q_w(U_0;0)^{-1}A^{21}(U_0;0)+\cdots.
\end{equation*}
Hence, the coefficient of the second derivative in the equation of $u_k$ is still $A^{12}(U_0;0)q_w(U_0;0)^{-1}A^{21}(U_0;0)$, which is the same as $u_0$, so the equation of $u_k$ is not strictly parabolic.

From previous discussions, it remains to find initial data for the coefficients $U_k$. For this purpose, we turn to consider the composite expansion.

\subsubsection{Composite Expansions}
Since $t= \varepsilon^2 \tau$, we have formally
\begin{equation*}
	\begin{aligned}
		\sum_{k=0}^\infty \varepsilon^k U_{k}(x,t) = \sum_{k=0}^\infty \varepsilon^k U_{k}(x,\varepsilon^2 \tau) = \sum_{k=0}^\infty \varepsilon^k \mathrm{P}_{k}(x,\tau),
	\end{aligned}
\end{equation*}
where
\begin{equation*}
	\mathrm{P}_{k}(x,\tau) = \sum_{h=0}^{[k/2]}\frac{\tau^h}{h!}\frac{\partial^h U_{k-2h}}{\partial t^h}(x,0)
\end{equation*}
is a polynomial of degree $[k/2]$ in $\tau$. Particularly, $\mathrm{P}_{0}(x,\tau) = U_0(x,0)$.

The composite expansion $U_\varepsilon^m$ in \eqref{4eq:U-expansion} becomes  
\begin{equation}\label{4eq:tau-expansion}
	\begin{aligned}
		\sum_{k=0}^m \varepsilon^k(U_k(x, t)+\mathrm{I}_k(x, \tau)) = \sum_{k=0}^m \varepsilon^k(\mathrm{P}_{k}(x,\tau)+\mathrm{I}_k(x, \tau)),
	\end{aligned}
\end{equation}
which is just the traditional inner expansion \cite{eckhaus1977matching}. Now write \eqref{4eq:keyequ} in variables $(x,\tau)$ as follows
\begin{equation*}
	\frac{1}{\varepsilon^2}\partial_{\tau} U + \frac{1}{\varepsilon}A(U;\varepsilon)\partial_x U = \frac{1}{\varepsilon^2}Q(U;\varepsilon).
\end{equation*}
The corrected formal solution should asymptotically satisfy the equations \eqref{4eq:keyequ}. Namely, the formal asymptotic expansion
\begin{equation}\label{4eq:r-inner-expansion}
	\begin{aligned}
		&R(\sum_{k=0}^\infty \varepsilon^k (\mathrm{P}_{k}(x,\tau)+\mathrm{I}_k(x, \tau)))\\
		&= -\varepsilon^{-2} Q(\mathrm{P}_0 + \mathrm{I}_0;0) + \varepsilon^{-1}\bigg[A(\mathrm{P}_0 + \mathrm{I}_0;0)\partial_x (\mathrm{P}_0 + \mathrm{I}_0) - \partial_\varepsilon Q(\mathrm{P}_0 + \mathrm{I}_0;0)\\
		&~ - Q_U(\mathrm{P}_0 + \mathrm{I}_0;0)(\mathrm{P}_1 + \mathrm{I}_1)\bigg]\\
		&~ + \sum_{k=0}^\infty \varepsilon^k \partial_t (\mathrm{P}_{k+2} + \mathrm{I}_{k+2})  - \sum_{k=0}^\infty \varepsilon^k \frac{1}{(k+2)!}\partial_\varepsilon^{k+2}Q(\mathrm{P}_0 + \mathrm{I}_0;0)\\
		&~ + \sum_{k=0}^\infty \varepsilon^k \sum_{l=0}^{k+1} \frac{1}{l!}\partial_\varepsilon^lA(\mathrm{P}_0 + \mathrm{I}_0;0)\partial_x (\mathrm{P}_{k+1-l} + \mathrm{I}_{k+1-l})\\
		&~ + \sum_{k=0}^\infty \varepsilon^k \sum_{l=0}^{k}\sum_{j=0}^{k-l} \frac{1}{l!}\bigg[\partial_U(\partial_\varepsilon^l A(\mathrm{P}_0 + \mathrm{I}_0;0))(\mathrm{P}_{k+1-l-j} + \mathrm{I}_{k+1-l-j})\\
		&~ + \mathrm{C}(\partial_\varepsilon^l A(\cdot~;0),k+1-l-j,\underline{\mathrm{I}+\mathrm{P}})\bigg]\partial_x (\mathrm{P}_{j} + \mathrm{I}_{j})\\
		&~ - \sum_{k=0}^\infty \varepsilon^k \sum_{l=0}^{k+1}\frac{1}{l!}[\partial_U(\partial_\varepsilon^lQ(\mathrm{P}_0 + \mathrm{I}_0;0))(\mathrm{P}_{k+2-l} + \mathrm{I}_{k+2-l})\\
		&~ + \mathrm{C}(\partial_\varepsilon^l Q(\cdot~;0);k+2-l,\underline{\mathrm{I}+\mathrm{P}})] \\
	\end{aligned}
\end{equation}
vanishes. This happens when each term of the last expansion is zero, i.e.,
\begin{equation}\label{4eq:inner-expansion}
	\begin{aligned}
		\partial_{\tau}(\mathrm{P}_0 + \mathrm{I}_0) =& Q(\mathrm{P}_0 + \mathrm{I}_0;0),\\
		\partial_{\tau}(\mathrm{P}_1 + \mathrm{I}_1) =& - A(\mathrm{P}_0 + \mathrm{I}_0;0)\partial_{x}(\mathrm{P}_0 + \mathrm{I}_0) + \partial_U Q(\mathrm{P}_0 + \mathrm{I}_0;0)(\mathrm{P}_1 + \mathrm{I}_1)\\
		& + \partial_\varepsilon Q(\mathrm{P}_0 + \mathrm{I}_0;0)(\mathrm{P}_0 + \mathrm{I}_0), \\
		\partial_{\tau}(\mathrm{P}_k + \mathrm{I}_k) =& \partial_U Q(\mathrm{P}_0 + \mathrm{I}_0;0)(\mathrm{P}_k + \mathrm{I}_k) + \mathrm{F}(k, \underline{\mathrm{I}+\mathrm{P}}),\qquad k\geq 2,
	\end{aligned}
\end{equation}
where
\begin{equation}\label{4eq:f_i_p-equ}
	\begin{aligned}
		&\mathrm{F}(k, \underline{\mathrm{I}+\mathrm{P}})\\
		=&  \frac{1}{k!}\partial_\varepsilon^{k}Q(\mathrm{P}_0 + \mathrm{I}_0;0) - \sum_{l=0}^{k-1} \frac{1}{l!}\partial_\varepsilon^lA(\mathrm{P}_0 + \mathrm{I}_0;0)\partial_x (\mathrm{P}_{k-1-l} + \mathrm{I}_{k-1-l})\\
		& -  \sum_{l=0}^{k-2}\sum_{j=0}^{k-2-l} \frac{1}{l!}\bigg[\partial_U(\partial_\varepsilon^l A(\mathrm{P}_0 + \mathrm{I}_0;0))(\mathrm{P}_{k-1-l-j} + \mathrm{I}_{k-1-l-j})\\
		& + \mathrm{C}(\partial_\varepsilon^l A(\cdot~;0),k-1-l-j,\underline{\mathrm{I}+\mathrm{P}})\bigg]\partial_x (\mathrm{P}_{j} + \mathrm{I}_{j})\\
		& + \sum_{l=1}^{k-1}\frac{1}{l!}[\partial_U(\partial_\varepsilon^lQ(\mathrm{P}_0 + \mathrm{I}_0;0))(\mathrm{P}_{k-l} + \mathrm{I}_{k-l}) + \mathrm{C}(\partial_\varepsilon^l Q(\cdot~;0);k-l,\underline{\mathrm{I}+\mathrm{P}})].	
	\end{aligned}
\end{equation}
Here $\mathrm{F}(k, \underline{\mathrm{I}+\mathrm{P}})$  depend only on the first $k$ terms of the inner expansion, which is $U_0$, $\mathrm{I}_0$, $\cdots$,$U_{k-1}$ and $\mathrm{I}_{k-1}$.

According to the definition of $\mathrm{P}_k$, $\sum_{k=0}^\infty \varepsilon^k \mathrm{P}_{k}(x,\tau)$ is also a solution of \eqref{4eq:keyequ}. Hence, we obtain as above
\begin{equation}\label{4eq:p-outer-expansion}
	\begin{aligned}
		&\partial_{\tau}\mathrm{P}_0 = Q(\mathrm{P}_0;0),\\
		&\partial_{\tau}\mathrm{P}_1 =- A(\mathrm{P}_0 ;0)\partial_{x}\mathrm{P}_0 + \partial_U Q(\mathrm{P}_0;0)\mathrm{P}_1+ \partial_\varepsilon Q(\mathrm{P}_0;0)\mathrm{P}_0 ,\\
		&\partial_{\tau}\mathrm{P}_k= \partial_U Q(\mathrm{P}_0;0)\mathrm{P}_k + \mathrm{F}(k, \underline{\mathrm{I}}).
	\end{aligned}
\end{equation}
Note that 
\begin{equation*}
	\mathrm{P}_0(x, \tau) = U_0(x;0),\qquad Q(\mathrm{P}_0;0) = 0.
\end{equation*}
We find from \eqref{4eq:inner-expansion} and \eqref{4eq:p-outer-expansion} that
\begin{equation}\label{4eq:I-equ}
	\begin{aligned}
		&\partial_{\tau}\mathrm{I}_0 = Q(\mathrm{P}_0+\mathrm{I}_0;0),\\
		&\partial_{\tau}\mathrm{I}_k= \partial_U Q(\mathrm{P}_0+\mathrm{I}_0;0)\mathrm{I}_k + [\partial_U Q(\mathrm{P}_0+\mathrm{I}_0;0)-\partial_U Q(\mathrm{P}_0;0)]\mathrm{P}_k + \mathrm{G}_k,
	\end{aligned}
\end{equation}
where
\begin{equation*}
	\mathrm{G}_k = \mathrm{F}(k, \underline{\mathrm{P}+\mathrm{I}}) - \mathrm{F}(k, \underline{\mathrm{P}})
\end{equation*}
with
\begin{equation*}
	\mathrm{F}(1, \underline{\mathrm{P}+\mathrm{I}}) = \partial_\varepsilon Q(\mathrm{P}_0 + \mathrm{I}_0;0)(\mathrm{P}_0 + \mathrm{I}_0) - A(\mathrm{P}_0 + \mathrm{I}_0;0)\partial_{x}(\mathrm{P}_0 + \mathrm{I}_0).  
\end{equation*}
According to the expression of $\mathrm{F}(k, \underline{\mathrm{I}+\mathrm{P}})$      \eqref{4eq:f_i_p-equ}, $\mathrm{G}_k$ depend only on $U_0$, $\mathrm{I}_0$, $\cdots$, $U_{k-1}$, $\mathrm{I}_{k-1}$.

\subsubsection{Initial Data for the Outer Expansion}
Now we determine the initial conditions for $U_k$. Assuming $\bar{U}(x;\varepsilon)$ has a formal asymptotic expansion as follows 
\begin{equation*}
	\bar{U}(x;\varepsilon) = \sum_{k=0}^\infty\varepsilon^k \bar{U}_k(x),\qquad \bar{U}_k(x) = (\bar{u}_k(x), \bar{w}_k(x))^T.
\end{equation*}
If the composite expansion \eqref{4eq:U-expansion} is a solution of \eqref{4eq:keyequ} and \eqref{initial-condition-key}, we should have
\begin{equation*}
	U_{k}(x,0)+\mathrm{I}(x,0) = \bar{U}_k(x),
\end{equation*}
or equivalently
\begin{equation}\label{4eq:initial-value}
	\begin{aligned}
		u_k(x,0)+\mathrm{I}_k^{I}(x,0)&=\bar{u}_k(x),\\
		w_k(x,0)+\mathrm{I}_k^{II}(x,0)&=\bar{w}_k(x).
	\end{aligned}
\end{equation}
From $Q^{I} = 0$ and the first equation of \eqref{4eq:I-equ}, we have $\partial_\tau \mathrm{I}_0^{I} = 0$. 
Meanwhile, since $\mathrm{I}_0$ satisfies $\mathrm{I}_0(x,+\infty)=0$, we know $\mathrm{I}_0^{I}(x,\tau)=0$, which means that there is no zero-th order initial layer for $u$. Together with $w_0=0$, we obtain
\begin{equation*}
	u_0(x,0) = \bar{u}_0(x),\qquad \mathrm{I}_0^{II}(x,0)=\bar{w}_0(x).
\end{equation*}
According to \eqref{4eq:I-equ}, $\mathrm{I}_0^{II}$ satisfies
\begin{equation}\label{4eq:i02-equ}
	\begin{aligned}
		\partial_\tau \mathrm{I}_0^{II}&=q(\bar{u}_0(x),\mathrm{I}_0^{II};0)\\
		\mathrm{I}_0^{II}(x, 0)&=\bar{w}_0(x).
	\end{aligned}
\end{equation}
Here and below the superscript 'I'(or 'II') stands for the first 3 (or last n+1) components of a vector in $\mathcal{R}^{n+4}$.

\begin{lemma}\label{4lemma:I0limit}
	Let $\bar{w}_0$ be sufficiently small. Then there exists a unique global smooth solution $\mathrm{I}_0$ satisfying
	\begin{equation}\label{4eq:I02}
		\parallel \mathrm{I}_0 \parallel_{s+m} \rightarrow 0, \qquad \mbox{exponentially as}~ \tau \rightarrow +\infty . 
	\end{equation}
\end{lemma}
\begin{proof}
	By Lemma \ref{4lemma1}, the system \eqref{4eq:keyequ} satisfies the structural stability condition and $P$ is a scalar matrix. 
	Then $q_w$ satisfies
	\begin{equation*}
	A_0^{22}(u,0;0)q_w(u,0;0) + q_w(u,0;0)^T A_0^{22}(u,0;0)\leq -I.
	\end{equation*}
	Therefore, for sufficiently small data $\bar{w}_0$, there is a unique global solution $\mathrm{I}_0^{II}(x,\tau)$ (see \cite{arnold1974equations}). Thanks to Lemma \ref{4lemma1}, $0\in \mathcal{R}^{n+1}$ is a fixed point for \eqref{4eq:I02}. Moreover, $q_w(u,0;0)$ is stable due to above equation. Hence $0\in \mathcal{R}^{n+1}$ is locally asymptotically stable for \eqref{4eq:I02}.
%	By Lemma \ref{4lemma:decay}, we know that $\mathrm{I}_0^{II}(x,\tau)$ decays to zero as $\tau$ goes to infinity. 
	By induction, for all $\alpha$ with $\alpha \leq s+m$, $\partial_x^\alpha \mathrm{I}_0^{II}(x,\tau)$ satisfies a linear ordinary differential equation of the form
	\begin{equation*}
		\partial_t Y = \partial_w q(\bar{u}_0(x),\mathrm{I}_0^{II};0)Y + g_\alpha(x,\tau).
	\end{equation*}
	Meanwhile, $g_\alpha(x,\tau)$ decays to zero as $\tau \rightarrow +\infty$. Thanks to Lemma \ref{4lemma:decay}, we see the exponential decay of $\parallel \mathrm{I}_0 \parallel_{s+m} \rightarrow 0$. 
\end{proof}

Assume that, for $k\geq 1$ and for any $i\leq k-1$, $\mathrm{I}_i$ exists globally in time and $\parallel \mathrm{I}_i(\cdot,\tau)\parallel_{s+m-i}$ decays exponentially fast to zero as $\tau\rightarrow +\infty$. Then so does $\parallel \mathrm{G}_k\parallel_{s+m-k}$ since $\mathrm{G}_k=\mathrm{F}(k, \underline{\mathrm{P}+\mathrm{I}}) - \mathrm{F}(k, \underline{\mathrm{P}})$ is a function of $\mathrm{I}_i,\mathrm{P}_i$$(0\leq i\leq k-1)$ and their first-order derivatives with respect to $x$. Because the $u$--component of $Q$ is 0 for $k\geq 1$, the first $3$ equations in \eqref{4eq:I-equ} are
\begin{equation}\label{4eq:I1-equ}
	\begin{aligned}
		\partial_\tau \mathrm{I}_k^{I} = \mathrm{G}^{I}_k,
	\end{aligned}
\end{equation}
Hence, 
\begin{equation*}
	\mathrm{I}_k^{I}(x,\tau) = \mathrm{I}_k^{I}(x,0) + \int_0^\tau \mathrm{G}^{I}_k(x, \tau')\mathrm{d} \tau',
\end{equation*}
which admits a limit $0$ as $\tau$ goes to infinity. Therefore
\begin{equation*}
	\mathrm{I}_k^{I}(x,\tau) = -\int_\tau^{+\infty} \mathrm{G}^{I}(k, x,\tau')\mathrm{d} \tau'.
\end{equation*}
and 
\begin{equation*}
	\mathrm{I}_k^{I}(x,\tau) = -\int_\tau^{+\infty} \mathrm{G}^{I}(k, x,\tau')\mathrm{d} \tau', \qquad \text{exponentially as}~ \tau \rightarrow +\infty.
\end{equation*}
In particular,
%In view of the matching principle, we take
\begin{equation}\label{4eq:I1-initial}
	\mathrm{I}_k^{I}(x,0) = -\int_0^{+\infty} \mathrm{G}^{I}_k(x,\tau')\mathrm{d} \tau',
\end{equation}
%Therefore,
%\begin{equation*}
%	\mathrm{I}_k^{I}(x,\tau) = -\int_\tau^{+\infty} \mathrm{G}^{I}(k, x,\tau')\mathrm{d} \tau'.
%\end{equation*}
%Thus, when $\tau$ goes to infinity, we see that
%\begin{equation*}
%	\parallel \mathrm{I}_k^{I}(x,\tau) \parallel_{s+m-k}\rightarrow 0,
%\end{equation*}
Together with \eqref{4eq:initial-value} it determines the initial value of $u_k$:
\begin{equation}\label{4eq:uk-initial}
	u_k(x,0)= \bar{u}(x) + \int_\tau^{+\infty} \mathrm{G}^{I}(k, x,\tau')\mathrm{d} \tau'.
\end{equation}

Furthermore, we can rewrite the remaining equation in \eqref{4eq:I-equ} as
\begin{equation}\label{4eq:Ik-2-equ}
	\begin{aligned}
			\partial_\tau \mathrm{I}_k^{II} =& \partial_w q(\mathrm{P}_0+\mathrm{I}_0;0)\mathrm{I}^{II}_k + \partial_u q(\mathrm{P}_0+\mathrm{I}_0;0)\mathrm{I}^{I}_k + [\partial_u q(\mathrm{P}_0+\mathrm{I}_0;0)-\partial_u q(\mathrm{P}_0;0)]\mathrm{P}^{I}_k\\
			& + [\partial_w q(\mathrm{P}_0+\mathrm{I}_0;0)-\partial_w q(\mathrm{P}_0;0)]\mathrm{P}^{II}_k  + \mathrm{G}^{II}(k) \\
			\equiv & \partial_w q(\mathrm{P}_0+\mathrm{I}_0;0)\mathrm{I}^{II}_k + \mathrm{G}'.
	\end{aligned}
\end{equation}
We know that $\parallel \mathrm{G}'\parallel_{s+m-k} $ decays exponentially fast to zero as $\tau \rightarrow +\infty $ from the definition of $\mathrm{G}$ and Lemma \ref{4lemma:I0limit}. Thanks to Lemma \ref{4lemma:decay}, we see the exponential decay of $\parallel \mathrm{I}_k^{II}(x,\tau) \parallel_{s+m-k}\rightarrow 0$. Hence, the inductive process is complete.

Now we describe a procedure to determine the coefficients of the expansion \eqref{4eq:U-expansion} using equations \eqref{4eq:outer-expansion} and \eqref{4eq:inner-expansion}. 
Based on previous analysis, $\mathrm{I}_0, U_0$ and $w_1$ are known. Then we can solve \eqref{4eq:Ik-2-equ} with the initial value providing in \eqref{4eq:initial-value} to obtained $\mathrm{I}_1^{II}$. 
The value of $\mathrm{I}^{I}_1$ can be determined by the equation \eqref{4eq:I1-equ} and initial value \eqref{4eq:I1-initial}. Hence, we can determine $U_0, U_1, \mathrm{I}_0, \mathrm{I}_1$ since the equation and initial value of $u_1$ are known. Assume inductively that $U_i, \mathrm{I}_i,  w_{i+1}$ with $i\leq k$ have been obtained. Then we can solve \eqref{4eq:Ik-2-equ} with the initial value providing in \eqref{4eq:initial-value} to obtained $\mathrm{I}^{II}_{k+1}$. And the equation \eqref{4eq:I1-equ} and the initial value \eqref{4eq:I1-initial} give the value of $\mathrm{I}^{I}_{k+1}$. Thus, $\mathrm{I}_{k+1}$ are completely determined. Moreover, \eqref{4eq:equ3} gives an expression of  $w_{i+2}$ as a function of $u_{k+1}$, $\partial_x u_{k+1}$ and of the known quantities and their derivatives. With this expression, the equation for $u_{k+1}$ can be derived from \eqref{4eq:equ2} together the initial value \eqref{4eq:uk-initial}. 

Therefore, we obtain  $U_{k+1}, \mathrm{I}_{k+1}$ and $w_{k+2}$. Hence, the inductive process is complete. 
In conclusion, we have determined all coefficients in expansions \eqref{4eq:U-expansion} and $\parallel \mathrm{I}_k \parallel_{s+m-k}$ decays exponentially to zero as $\tau\rightarrow +\infty$.

\subsubsection{Residual estimation}
The next lemma is concerning the residual of the formal approximation $R(U_\varepsilon^m)$.
\begin{theorem}\label{4theorem:asp-err}
	Let $R(U_\varepsilon^m)$ be defined by \eqref{4eq:asp-err}. Then
	\begin{equation*}
		R(U_\varepsilon^m) = \varepsilon^{m-1}Q_U(U_0;0)U_{m+1}+\varepsilon^{m-1}F_m,
	\end{equation*}
	where $Q_U(U_0;0)U_{m+1}$ is completely determined by the first $m$ terms of the outer expansion. And $F_m$ satisfies
	\begin{equation}\label{4eq:fm-err}
		\parallel F_m\parallel_s \leq C\varepsilon + Ce^{-\mu \tau},
	\end{equation}
	with $\mu\geq 0$ and $C$ constants independent of $\varepsilon$ .
\end{theorem}
\begin{proof}
	The proof of this theorem mainly refers to the literature \cite{yong1999singular} and \cite{lattanzio2001hyperbolic}. From the relation in \eqref{4eq:outer-expansion}, we have
	\begin{equation*}
		R(\sum_{k=0}^m \varepsilon^k U_k)= \varepsilon^{m-1}Q_U(U_0;0)U_{m+1} + O(\varepsilon^m),
	\end{equation*}
	where
	\begin{equation*}
		\begin{aligned}
			&Q_U(U_0;0)U_{m+1}\\
			=& \partial_t U_{m-1} +  \sum_{l=0}^{m-1} \frac{1}{l!}\partial_\varepsilon^lA(U_0;0)\partial_x U_{m-l} \\
			&+ \sum_{l=0}^{m}\sum_{j=0}^{m-1-l} \frac{1}{l!}\bigg[\partial_U(\partial_\varepsilon^l A(U_0;0))U_{m-l-j} + \mathrm{C}(\partial_\varepsilon^l A(\cdot~;0),m-l-j,\underline{U})\bigg]\partial_x U_{j}\\
			&- \sum_{l=1}^{m}\frac{1}{l!}[\partial_U(\partial_\varepsilon^lQ(U_0;0))U_{m+1-l} +  \mathrm{C}(\partial_\varepsilon^l Q(\cdot~;0) ;m+1-l,\underline{U})]\\
			&- \frac{1}{(m+1)!}\partial_\varepsilon^{m+1}Q(U_0;0).
		\end{aligned}
	\end{equation*}
	Then $Q_U(U_0;0)U_{m+1}$ depend only on $U_0,\cdots, U_{m}$. Define $F_m$ as
	\begin{equation*}
		\varepsilon^{m-1} F_m = R(U_\varepsilon^m) - \varepsilon^{m-1}Q_U(U_0;0)U_{m+1}.
	\end{equation*}
	With this definition, we only need to prove \eqref{4eq:fm-err}.
	
	To this end, consider the Taylor expansion with respect to $\varepsilon$ at $\varepsilon=0$:
	\begin{equation*}
		\sum_{k=0}^m \varepsilon^k U_k(x,t) = \sum_{k=0}^m \varepsilon^k U_k(x,\varepsilon^2\tau)=\sum_{k=0}^m \varepsilon^k \mathrm{P}_k(x,\tau) + \varepsilon^{m+1} \tilde{\mathrm{P}}(x,t,\tau,\varepsilon),
	\end{equation*}
	where $\tilde{\mathrm{P}}(x,t,\tau,\varepsilon) = O(1)\tau^{1+[m/2]}$. Thus, we can write
	\begin{equation*}
		\begin{aligned}
			U_\varepsilon^m &= \sum_{k=0}^m \varepsilon^k(U_k(x, t) + \mathrm{I}(x,\tau)) \\
			&= \sum_{k=0}^m \varepsilon^k (\mathrm{P}_k(x,\tau) + \mathrm{I}_k(x,\tau)) + \varepsilon^{m+1} \tilde{\mathrm{P}}(x,t,\tau,\varepsilon),
		\end{aligned}
	\end{equation*}
	In the spirit of the relation \eqref{4eq:r-inner-expansion} for the inner expansion, we deduce from the definition of $R(U_\varepsilon^m)$ that
	\begin{equation*}
		\begin{aligned}
			R(U_\varepsilon^m) &= \varepsilon^{m-1}[\tilde{\mathrm{P}}_{\tau} + C(\varepsilon,\tilde{\mathrm{P}};\mathrm{I}_0+\mathrm{P}_0,\cdots,\mathrm{I}_m+\mathrm{P}_m)]\\
			R(\sum_{k=0}^m U_k) &= \varepsilon^{m-1}[\tilde{\mathrm{P}}_{\tau} + C(\varepsilon,\tilde{\mathrm{P}};\mathrm{P}_0,\cdots,\mathrm{P}_m)].
		\end{aligned}
	\end{equation*}
	Here $C(\varepsilon,\tilde{\mathrm{P}};\mathrm{P}_0,\cdots,\mathrm{P}_m)$ depends smoothly on the $\varepsilon,\tilde{\mathrm{P}};\mathrm{P}_0,\cdots,\mathrm{P}_m$ and their first-order derivatives with respect to.
	
	Furthermore, it follows from the definition of $F_m$ that
	\begin{equation*}
		\begin{aligned}
			F_m &= \varepsilon^{-(m-1)}R(U_\varepsilon^m)-\varepsilon^{-(m-1)}R(\sum_{k=0}^m U_k) + O(\varepsilon)\\
			&=C_U(\varepsilon,\tilde{\mathrm{P}};\cdot)\mathrm{I} + O(\varepsilon),
		\end{aligned}
	\end{equation*}
	where $C_U(\varepsilon, \tilde{\mathrm{P}};\cdot)$ denotes the Fr\'echet derivative of the operator $C(\varepsilon, \tilde{\mathrm{P}};\cdot)$. Finally, the estimate in \eqref{4eq:fm-err} follows from the decay property of the $\mathrm{I}$ when $\tau$ tends to infinity.
\end{proof}

\subsection{Justification of formal expansions}\label{4-2sec}
Having constructed formal asymptotic approximations $U_\varepsilon^m$ for the initial-value problem \eqref{4eq:keyequ} and \eqref{initial-condition-key}, we prove here the validity of the approximations under Lemma \eqref{4lemma1} and under some regularity assumptions on the given data. For the sake of exactness, we refer to next remark and make the following assumption.

\begin{assumption}\label{4ass1}
Let $s>\frac{3}{2}$. 
\begin{enumerate}
	\item There exists a convex open set $G_0\subset\subset G$ satisfying $G_0\subset\subset G$ such that $\bar{U}(x; \varepsilon)\in G_0$ for all $\varepsilon >0$ and all $x\in\Omega$, and $\bar{U}(\cdot; \varepsilon) \in H^s$ is periodic on $\Omega$;
	\item $A(U;\varepsilon),Q(U;\varepsilon),P(U;\varepsilon),A_0(U;\varepsilon)$ are smooth function of $U\in G,\varepsilon\in [0,\varepsilon_0]$;
	\item $Q_U(U_0;0)U_{m+1} \in C([0, T_m],H^s)$;
	\item $U_\varepsilon^m$ takes value in $\bar{G}_0$ and satisfies $U_\varepsilon^m \in C([0,T_m],H^{s+1})\bigcap C^1([0,T_m],H^s) $. For sufficiently small $\varepsilon>0$,  
		\begin{equation}\label{initial-err}
			\parallel U_\varepsilon^m(0,\cdot) - \bar{U}(\cdot,\varepsilon)\parallel_s \leq c\varepsilon^m,
		\end{equation}
		and
		\begin{equation}\label{boundary-condition}
 			\sup_{0\leq t\leq T_m}\parallel U_\varepsilon^m - U_0 \parallel_s \leq C \varepsilon + C \varepsilon^2 B_\varepsilon(t), \qquad \parallel \partial_t U_\varepsilon^m\parallel_s \leq c + cB_\varepsilon(t),
		\end{equation}
		where $B_\varepsilon(t) =\varepsilon^{-2} e^{-\frac{\mu t}{\varepsilon^2}}$ and $\mu>0$ is a constant independent of $\varepsilon$.
\end{enumerate}
\end{assumption}

\begin{remark}
	The first assumption is necessary to apply the existence theorem, see \cite{kato1975cauchy}. The second assumption is obviously. The next can be verified by using the existence theory for parabolic system in \cite{Taylor}.
	\eqref{initial-err} is a natural condition on the initial data. It stands for initial errors. In the above subsection, we have constructed $U_k$ and $\mathrm{I}_k$. Now we show that, for any fixed $m\in \mathcal{N}$, the approximate solution $U_\varepsilon^m$ defined by \eqref{4eq:U-expansion} satisfies \eqref{boundary-condition}. Indeed, since $\mathrm{I}_0^I=0$ and $\mathrm{I}^{II}_0(\cdot, \tau)$ decays exponentially fast to zero as $\tau \rightarrow +\infty$ with $\tau=t/\varepsilon^2$, thus $\norm{\mathrm{I}_0}_s \leq C e^{-\frac{\mu t}{\varepsilon^2}}$ with $\mu>0$ a constant independent of $\varepsilon$. Meanwhile 
	\begin{equation*}
		\begin{aligned}
			\partial_t \mathrm{I}^{II}_0(\cdot, \tau) = \varepsilon^{-2} \partial_{\tau} \mathrm{I}^{II}_0(\cdot, \tau) = \varepsilon^{-2} q(\bar{u}_0(x),\mathrm{I}_0^{II};0).
		\end{aligned}
	\end{equation*}
	Therefore
	\begin{equation*}
		\begin{aligned}
			\norm{U_\varepsilon^m - U_0}_s &= \norm{\sum_{k=1}^m \varepsilon^k U_k(\cdot, t) + \sum_{k=0}^m \varepsilon^k \mathrm{I}_k(\cdot, t/\varepsilon^2)}_s \leq C \varepsilon + C \varepsilon^2 B_\varepsilon(t),\\
			\norm{\partial_t U_\varepsilon^m}_s &= \norm{\partial_t \mathrm{I}_0(\cdot, \tau) + \sum_{k=0}^m \varepsilon^k U_k(\cdot, t) + \sum_{k=1}^m \varepsilon^k \mathrm{I}_k(\cdot, t/\varepsilon^2)}_s \leq c + cB_\varepsilon(t).
		\end{aligned}
	\end{equation*}	
\end{remark}

Fix $\varepsilon >0$ and recall assumption  \ref{4ass1}. According to Theorem 2.1 in \cite{majda2012compressible}, for any convex open set $G_1$ satisfying $G_0\subset\subset G_1\subset\subset G$, there exist  $T_\varepsilon>0$ such that that initial value problem \eqref{4eq:keyequ} and \eqref{initial-condition-key} for the symmetrizable hyperbolic system has a unique $H^s$--solution $U^\varepsilon$ satisfying $U^\varepsilon \in C([0,T_\varepsilon],H^{s})$ and $U^\varepsilon \in \bar{G}_1$. Without loss of generality, we assume that $T_\varepsilon$ is the maximal time interval where the $H^s$--solution $U^\varepsilon$ take value in $\bar{G}_1$. Note that $T_\varepsilon$ may shrink to zero as so does $\varepsilon$.

In order to show $T_\varepsilon \geq T_m$, we state our main result.

\begin{theorem}\label{4theorem-1}
	Under the assumption \ref{4ass1} with $m>2$, suppose $s>\frac{3}{2}$ is a integer, $[0, T_\varepsilon]$ is the maximal time interval where \eqref{4eq:keyequ} has a solution $U^\varepsilon \in C([0,T_\varepsilon],H^{s})$  with values in a convex set $\bar{G}_1$, and $[0,T_m]$ a time interval where the asymptotic approximation $U_\varepsilon^m$ of the form \eqref{4eq:U-expansion}. 
	
	Then there exists a constant $K$, independent of $\varepsilon$ but dependent on $T_m$, such that
	\begin{equation*}
		\parallel U^\varepsilon(t) - U_\varepsilon^m(t) \parallel_s \leq K \varepsilon^m,
	\end{equation*}
	for sufficiently small $\varepsilon$ and $t \in [0, \min\{T_m, T_\varepsilon\})$.
\end{theorem}
Before proving this theorem, we remark that $m>2$  is required by the following proof (see \eqref{4eq:v-err}) below). However, since 
\begin{equation*}
	U_\varepsilon^m(x,t) = U_\varepsilon^{m_0}(x,t) + \sum_{k=m_0+1}^m \varepsilon^k \bigg(U_k(x,t) + \mathrm{I}(x, t/\varepsilon^2)\bigg),
\end{equation*}
we have
\begin{equation*}
	\parallel U^\varepsilon(t) - U_\varepsilon^{m_0}(t) \parallel_s \leq \parallel U^\varepsilon(t) - U_\varepsilon^m(t) \parallel_s + \sum_{k=m_0+1}^m \varepsilon^k \parallel U_k(x,t) + \mathrm{I}(x, t/\varepsilon^2) \parallel_s.
\end{equation*}
and thus
\begin{equation*}
	\parallel U^\varepsilon(t) - U_\varepsilon^{m_0}(t) \parallel_s = O(\varepsilon^{m_0+1})
\end{equation*}
for any $m_0 \leq m$ provided that the coefficients of $\varepsilon^k$ in the sum are bounded.

In addition, on the basis of Theorem \ref{4theorem-1}, we use exactly the same argument in \cite{yong1999singular} to obtain
\begin{theorem}\label{4theorem-3}
	The hypotheses of Theorem \ref{4theorem-1} imply $T_\varepsilon \geq T_m$.
\end{theorem}
\begin{proof}
	If $T_\varepsilon \leq T_m$, then Theorem \ref{4theorem-1} gives
	\begin{equation*}
		\parallel U^\varepsilon(T_\varepsilon) - U_\varepsilon^m(T_\varepsilon) \parallel_s \leq K \varepsilon^m.
	\end{equation*}
	Thus, it follows from the embedding inequality that $U^\varepsilon(T_\varepsilon) \in G_0$ if if $\varepsilon$ is small enough. Now we could apply Theorem 2.1 in \cite{majda2012compressible}, beginning at the time $T_\varepsilon$, to continue this solution beyond $T_\varepsilon$. This is a contradiction. Therefore $T_\varepsilon \geq T_m$.
\end{proof}

Now we prove the Theorem \ref{4theorem-1}.

$\mathbf{The}$ $\mathbf{Proof}$ $\mathbf{of}$ $\mathbf{Theorem}$ $\mathbf{\ref{4theorem-1}}:$
Let $T_* = \min\{T_\varepsilon, T_m\}$, then both the exact solution $U^\varepsilon$ and the approximate solution $U_\varepsilon^m$ are defined on time interval $[0, T_*)$, satisfy equation \eqref{4eq:keyequ} and   
\begin{equation*}
	\begin{aligned}
%		\partial_t U^\varepsilon + \frac{1}{\varepsilon}A(U^\varepsilon;\varepsilon)\partial_x U^\varepsilon &= \frac{1}{\varepsilon^2}Q(U^\varepsilon;\varepsilon),\\
		\partial_t U^m_\varepsilon + \frac{1}{\varepsilon}A(U^m_\varepsilon;\varepsilon)\partial_x U^m_\varepsilon = \frac{1}{\varepsilon^2}Q(U^m_\varepsilon;\varepsilon) + R^m_\varepsilon.
	\end{aligned}
\end{equation*}
On $[0, T_*)$, we define
\begin{equation*}
	V = U^\varepsilon - U_\varepsilon^m,
\end{equation*}
then
\begin{equation}\label{4eq:V-equation}
\begin{aligned}
	\partial_t V + \frac{1}{\varepsilon}A(U^{\varepsilon};\varepsilon)\partial_x V =& \frac{1}{\varepsilon}\big(A(U^m_\varepsilon;\varepsilon)-A(U^\varepsilon;\varepsilon)\big)\partial_x U_\varepsilon^m + \frac{1}{\varepsilon^2}\big(Q(U^\varepsilon;\varepsilon)-Q(U_\varepsilon^m;\varepsilon)\big) - R^m_\varepsilon.
\end{aligned}
\end{equation}
Applying $\partial_x^\alpha$ to the last equation for multi-index $\alpha$ satisfying $|\alpha|\leq s$, and setting $V_\alpha = \partial^\alpha V$, we get 
\begin{equation*}
	\begin{aligned}
		&\partial_t V_\alpha + \frac{1}{\varepsilon}A(U^{\varepsilon};\varepsilon)\partial_x V_\alpha \\
		=& \frac{1}{\varepsilon}\bigg\{\big(A(U^m_\varepsilon;\varepsilon)-A(U^\varepsilon;\varepsilon)\big)\partial_x U_\varepsilon^m\bigg\}_\alpha + \frac{1}{\varepsilon^2}\bigg\{\big(Q(U^\varepsilon;\varepsilon)-Q(U_\varepsilon^m;\varepsilon)\big)\bigg\}_\alpha \\
		&+ \frac{1}{\varepsilon}\bigg\{A(U^{\varepsilon};\varepsilon)\partial_x V_\alpha - [A(U^{\varepsilon};\varepsilon)\partial_x V]_\alpha\bigg\}- (R^m_\varepsilon)_\alpha.\\
	\end{aligned}
\end{equation*}
We consider the energy norm $e(V_\alpha(x,t)) = V_\alpha^T A_0(U^\varepsilon;\varepsilon)V_\alpha$. Multiplying the last equation by $2V_\alpha^T A_0(U^\varepsilon;\varepsilon)$ and integrating over $x\in \Omega$ yields 
\begin{equation}\label{4eq:err}
	\begin{aligned}
		\frac{\mathrm{d}}{\mathrm{d} t}\int e(V_\alpha) \mathrm{d}x 
		=& \frac{2}{\varepsilon} \int V_\alpha^T A_0(U^\varepsilon;\varepsilon)\bigg\{\bigg(A(U^m_\varepsilon;\varepsilon)-A(U^\varepsilon;\varepsilon)\bigg)\partial_x U_\varepsilon^m\bigg\}_\alpha \mathrm{d}x  \\
		& + \frac{2}{\varepsilon^2} \int V_\alpha^T A_0(U^\varepsilon;\varepsilon)\bigg\{\bigg(Q(U^\varepsilon;\varepsilon)-Q(U_\varepsilon^m;\varepsilon)\bigg)\bigg\}_\alpha \mathrm{d}x \\
		& + \frac{2}{\varepsilon} \int V_\alpha^T A_0(U^\varepsilon;\varepsilon)\bigg(A(U^{\varepsilon};\varepsilon)\partial_x V_\alpha - [A(U^{\varepsilon};\varepsilon)\partial_x V]_\alpha\bigg) \mathrm{d}x \\
		& - 2 \int V_\alpha^T A_0(U^\varepsilon;\varepsilon)\partial_\alpha (R^m_\varepsilon) \mathrm{d}x  \\ 
		& + \int V_\alpha^T \bigg(\partial_t (A_0(U^\varepsilon;\varepsilon)) + \frac{1}{\varepsilon} \partial_x (A_0(U^\varepsilon;\varepsilon)A(U^\varepsilon;\varepsilon))\bigg)V_\alpha \mathrm{d}x \\
		\triangleq & I_1^\alpha + I_2^\alpha + I_3^\alpha + I_4^\alpha + I_5^\alpha.
	\end{aligned}
\end{equation}
Next we estimate each term in the right-hand side of \eqref{4eq:err}. Firstly,
\begin{equation*}
	\begin{aligned}
			I_1^\alpha =& \frac{2}{\varepsilon}\int  V_\alpha^T A_0(U^\varepsilon;\varepsilon)\bigg\{\bigg(A(U^m_\varepsilon;\varepsilon)-A(U^\varepsilon;\varepsilon)\bigg)\partial_x U_\varepsilon^m\bigg\}_\alpha \mathrm{d}x\\
			=& \frac{2}{\varepsilon}\int V_\alpha^T \bigg(A_0(U^\varepsilon;\varepsilon) - A_0(U_0; 0)\bigg)\bigg\{\bigg(A(U^m_\varepsilon;\varepsilon)-A(U^\varepsilon;\varepsilon)\bigg)\partial_x U_\varepsilon^m\bigg\}_\alpha \mathrm{d}x\\
			&+\frac{2}{\varepsilon}\int V_\alpha^T A_0(U_0;0)\bigg\{\bigg(A(U^m_\varepsilon;\varepsilon)-A(U^\varepsilon;\varepsilon)\bigg)\partial_x U_\varepsilon^m\bigg\}_\alpha \mathrm{d}x. 
	\end{aligned}
\end{equation*}
Recall that $U^\varepsilon$ and $U_0$ takes
values in a compact subset $G_1$. By using Assumption \ref{4ass1}, we can obtain
\begin{equation}\label{4eq:A0-err}
	\begin{aligned}
		&A_0(U^\varepsilon;\varepsilon) - A_0(U_0;0)\\
		&= A_0(U^\varepsilon;\varepsilon) - A_0(U^\varepsilon;0) + A_0(U^\varepsilon;0) - A_0(U_0;0)\\
		&= -\varepsilon \int_0^1 A_{0\varepsilon}(U^\varepsilon;\theta\varepsilon) \mathrm{d} \theta - |U^\varepsilon - U_0|\int_0^1 A_{0U}(U^\varepsilon + \theta(U_0 - U^\varepsilon);0)\mathrm{d} \theta\\
		&\leq C\varepsilon + C |U^\varepsilon - U_0|\\
		&\leq C(\varepsilon + |U^\varepsilon - U_\varepsilon^m| + |U^m_\varepsilon - U_0|) \\
		&\leq C\varepsilon + C\varepsilon^2 \triangle + C\varepsilon^2 B_\varepsilon(t),
	\end{aligned}
\end{equation}
where $\triangle \dot{=} \parallel U_\varepsilon^m - U^\varepsilon \parallel_s/ \varepsilon^2 $. Here and below, $C$ is a generic constant which may change from line to line. Since
\begin{equation*}
	\begin{aligned}
		A(U^\varepsilon;\varepsilon)-A(U_\varepsilon^m;\varepsilon)= -\int_{0}^1 A_{U}(U_\varepsilon^m + \theta(U^\varepsilon-U_\varepsilon^m);\varepsilon)V \mathrm{d} \theta.
	\end{aligned}
\end{equation*}
Using the calculus inequalities in Sobolev spaces \eqref{4lemma:calculus_inequ}, we get
\begin{equation*}
	\begin{aligned}
		&\parallel \bigg\{\bigg(A(U^m_\varepsilon;\varepsilon)-A(U^\varepsilon;\varepsilon)\bigg)\partial_x U_\varepsilon^m\bigg\}_\alpha \parallel
		\leq C\parallel V \parallel_\alpha,
	\end{aligned}
\end{equation*}
For the first term of $I_1^\alpha$, we have
\begin{equation*}
	\begin{aligned}
		&\int \frac{2}{\varepsilon}V_\alpha^T \bigg(A_0(U^\varepsilon;\varepsilon) - A_0(U_0; 0)\bigg)\bigg\{\bigg(A(U^m_\varepsilon;\varepsilon)-A(U^\varepsilon;\varepsilon)\bigg)\partial_x U_\varepsilon^m\bigg\}_\alpha \mathrm{d}x \\
		&\leq C(1 + \varepsilon \triangle + \varepsilon B_\varepsilon(t))\parallel V\parallel_s^2.
	\end{aligned}
\end{equation*}
According to Lemma \ref{4lemma1}, we know that $A_0(U_0; 0)$ is a block diagonal matrix, then the second term of $I_1^\alpha$ can be rewritten as
\begin{equation}\label{e2-1}
	\begin{aligned}
		&\frac{2}{\varepsilon} \int V_\alpha^T A_0(U_0;0)\bigg\{\bigg(A(U^m_\varepsilon;\varepsilon)-A(U^\varepsilon;\varepsilon)\bigg)\partial_x U_\varepsilon^m\bigg\}_\alpha \mathrm{d}x\\
		=& \frac{2}{\varepsilon}\int V_\alpha^{IT} A^{11}_0(U_0;0)\bigg\{\bigg(A^{11}(U^m_\varepsilon;\varepsilon)-A^{11}(U^\varepsilon;\varepsilon)\bigg)\partial_x u_\varepsilon^m\bigg\}_\alpha \mathrm{d}x\\
		& + \frac{2}{\varepsilon}\int V_\alpha^{IT} A^{11}_0(U_0;0)\bigg\{\bigg(A^{12}(U^m_\varepsilon;\varepsilon)-A^{12}(U^\varepsilon;\varepsilon)\bigg)\partial_x w_\varepsilon^m\bigg\}_\alpha \mathrm{d}x\\
		& + \frac{2}{\varepsilon}\int V_\alpha^{IIT} A^{22}_0(U_0;0)\bigg\{\bigg(A^{21}(U^m_\varepsilon;\varepsilon)-A^{21}(U^\varepsilon;\varepsilon)\bigg)\partial_x u_\varepsilon^m\bigg\}_\alpha \mathrm{d}x\\
		& + \frac{2}{\varepsilon}\int V_\alpha^{IIT} A^{22}_0(U_0;0)\bigg\{\bigg(A^{22}(U^m_\varepsilon;\varepsilon)-A^{22}(U^\varepsilon;\varepsilon)\bigg)\partial_x w_\varepsilon^m\bigg\}_\alpha \mathrm{d}x.
	\end{aligned}
\end{equation}
The last two terms on the right-hand side are bounded by
\begin{equation*}
	\frac{C}{\varepsilon} \parallel V\parallel_s \parallel V_\alpha^{II}\parallel \leq \frac{\delta}{\varepsilon^2}\parallel V_\alpha^{II} \parallel^2 + C\parallel V\parallel_s^2.
\end{equation*}
Since $w_0 = 0$, the Assumption \ref{4ass1} yields $\parallel w_\varepsilon^m\parallel_s \leq C(\varepsilon + \varepsilon^2 B_\varepsilon(t))$. Therefore
\begin{equation*}
	\begin{aligned}
		&\int V_\alpha^{IT} A^{11}_0(U_0;0)\bigg\{\bigg(A^{12}(U^m_\varepsilon;\varepsilon)-A^{12}(U^\varepsilon;\varepsilon)\bigg)\partial_x w_\varepsilon^m\bigg\}_\alpha \mathrm{d}x\\
		&~ \leq C(\varepsilon + \varepsilon^2 B_\varepsilon(t)) \parallel V\parallel_s^2.
	\end{aligned}
\end{equation*}
For the first term in \eqref{e2-1}, we have
\begin{equation}\label{4-92}
	\begin{aligned}
		&A^{11}(U^m_\varepsilon;\varepsilon)-A^{11}(U^\varepsilon;\varepsilon)\\
		=& A^{11}(u^m_\varepsilon, w^m_\varepsilon;\varepsilon)-A^{11}(u^\varepsilon,w^m_\varepsilon;\varepsilon) + A^{11}(u^\varepsilon,w^m_\varepsilon;\varepsilon)-A^{11}(u^\varepsilon,w^\varepsilon;\varepsilon)\\
		=& -\int_0^1 \partial_u A^{11}(u_\varepsilon^m+\theta(u^\varepsilon -u_\varepsilon^m),w^m_\varepsilon;\varepsilon)V^{I} \mathrm{d} \theta\\
		& - \int_0^1 \partial_w A^{11}(u^\varepsilon; w_\varepsilon^m+\theta(w^\varepsilon -w_\varepsilon^m);\varepsilon)V^{II} \mathrm{d} \theta.
	\end{aligned}
\end{equation}
The second integral above is easily estimated due to the appearance of $\parallel V^{II}\parallel $. The first one can be treated due to condition $\partial_u A^{11}(U_0;0)=0$ in Lemma \ref{4lemma1}. Precisely, we write
\begin{equation*}
	\begin{aligned}
		&\partial_u A^{11}(u_\varepsilon^m+\theta(u^\varepsilon -u_\varepsilon^m),w^m_\varepsilon;\varepsilon) \\
		=& \partial_u A^{11}(u_\varepsilon^m+\theta(u^\varepsilon -u_\varepsilon^m),w^m_\varepsilon;\varepsilon) - \partial_u A^{11}(u_0,w^m_\varepsilon;\varepsilon)\\
		& + \partial_u A^{11}(u_0,w^m_\varepsilon;\varepsilon)- \partial_u A^{11}(u_0,0;\varepsilon) + \partial_u A^{11}(u_0,0;\varepsilon)- \partial_u A^{11}(u_0,0;0)\\
		=& \int_0^1 \partial^2_{uu} A^{11}(u(\theta,\theta'),w^m_\varepsilon;\varepsilon)(u_\varepsilon^m - u_0 +\theta(u^\varepsilon -u_\varepsilon^m))\mathrm{d} \theta'\\
		& + \int_0^1 \partial^2_{uw} A^{11}(u_0; \theta' w_\varepsilon^m;\varepsilon)w_\varepsilon^m \mathrm{d} \theta' + \int _0^1 \partial^2_{u\varepsilon}A^{11}(u_0,0;\theta' \varepsilon)\varepsilon \mathrm{d} \theta'.
	\end{aligned}
\end{equation*}
Here $u(\theta,\theta')=(1-\theta')u_0+\theta' u_\varepsilon^m+\theta(u^\varepsilon -u_\varepsilon^m)$. The integral in  \eqref{4-92} can be rewritten as
\begin{equation*}
	\begin{aligned}
		&\int_0^1  \partial_u A^{11}(u_\varepsilon^m+\theta(u^\varepsilon -u_\varepsilon^m),w^m_\varepsilon;\varepsilon)V^{I} \mathrm{d} \theta \\
		&= \int_0^1\int_0^1 \partial^2_{uu} A^{11}(u(\theta,\theta'),w^m_\varepsilon;\varepsilon)((u_\varepsilon^m - u_0 +\theta(u^\varepsilon -u_\varepsilon^m)),V^{I}) \mathrm{d} \theta \mathrm{d} \theta'\\
		&+ \int_0^1\int_0^1 \partial^2_{uw} A^{11}(u_0, \theta' w_\varepsilon^m;\varepsilon)(V^{I},w_\varepsilon^m) \mathrm{d} \theta\mathrm{d} \theta' + \int_0^1\int_0^1 \partial^2_{u\varepsilon} A^{11}(u_0, 0;\theta' \varepsilon)(\varepsilon,V^{I}) \mathrm{d} \theta\mathrm{d} \theta'.
	\end{aligned}
\end{equation*}
Since $\parallel w_\varepsilon^m \parallel $ and $\varepsilon$ both can be bounded by $C(\varepsilon + \varepsilon^2 B_\varepsilon(t))$, and 
\begin{equation*}
	\norm{u_\varepsilon^m -u_0}_s \leq C \varepsilon + C \varepsilon^2 B_\varepsilon(t),\qquad \norm{u^\varepsilon - u_\varepsilon^m}_s \leq C \varepsilon^2 \triangle,
\end{equation*}
then
\begin{equation*}
	\begin{aligned}
		&\int V_\alpha^{IT} A^{11}_0(U_0;0)\bigg\{\bigg(A^{11}(U^m_\varepsilon;\varepsilon)-A^{11}(U^\varepsilon;\varepsilon)\bigg)\partial_x u_\varepsilon^m\bigg\}_\alpha \mathrm{d}x \\
		&\leq \frac{\delta}{\varepsilon}\parallel V_\alpha^{II}\parallel + C(\varepsilon + \varepsilon^2 \triangle + \varepsilon^2 B_\varepsilon(t))\parallel V\parallel_s^2.
	\end{aligned}
\end{equation*}
Therefore 
\begin{equation}\label{4eq:err1}
	I_1^\alpha \leq \frac{\delta}{\varepsilon^2}\parallel V_\alpha^{II}\parallel  ^2 + C(1 + \varepsilon\triangle + \varepsilon B_\varepsilon(t))\parallel V\parallel_s^2.
\end{equation}

The second item is
\begin{equation*}
	I_2^\alpha = \frac{2}{\varepsilon^2} \int V_\alpha^T A_0(U^\varepsilon;\varepsilon)\bigg\{\bigg(Q(U^\varepsilon;\varepsilon)-Q(U_\varepsilon^m;\varepsilon)\bigg)\bigg\}_\alpha \mathrm{d}x.
\end{equation*}
We first rewrite $Q(U^\varepsilon;\varepsilon)-Q(U_\varepsilon^m;\varepsilon)$ as
\begin{equation*}
	\begin{aligned}
		&Q(U^\varepsilon;\varepsilon)-Q(U_\varepsilon^m;\varepsilon)\\
		=& Q_U(U_0;0)V+ \varepsilon \partial_\varepsilon Q_U(U_0;0)V\\
		& + [Q(U^\varepsilon;0) - Q(U_\varepsilon^m;0)-Q_U(U_0;0)V]\\
		& + \varepsilon[\partial_\varepsilon Q(U^\varepsilon;0) - \partial_\varepsilon Q(U_\varepsilon^m;0) - \partial_\varepsilon Q_u(U_0;0)V]\\
		& + [Q(U^\varepsilon;\varepsilon)-Q(U^\varepsilon;0)-\varepsilon \partial_\varepsilon Q(U^\varepsilon;0) - \big(Q(U^m_\varepsilon;\varepsilon)-Q(U^m_\varepsilon;0)-\varepsilon \partial_\varepsilon Q(U^m_\varepsilon;0)\big)]
	\end{aligned}
\end{equation*}
which implies that
\begin{equation*}
	I_2^\alpha = I_{21}^\alpha+I_{22}^\alpha+I_{23}^\alpha+I_{24}^\alpha+I_{25}^\alpha
\end{equation*}
with the natural correspondence for $I_{21}^\alpha$, $\cdots$, $I_{25}^\alpha$. Now we estimate each of these terms. For $I_{21}^\alpha$ we write
\begin{equation*}
	\begin{aligned}
		I_{21}^\alpha =& \frac{2}{\varepsilon^2}\int V_\alpha^T A_0(U^\varepsilon;\varepsilon)\{Q_U(U_0;0)V \}_\alpha \mathrm{d}x\\
		=&\frac{2}{\varepsilon^2} \int V_\alpha^T A_0(U_0;0)Q_U(U_0;0)V_\alpha \mathrm{d}x\\
		&+ \frac{2}{\varepsilon^2} \int V_\alpha^T A_0(U_0;0)[\partial^\alpha(Q_U(U_0;0)V) - Q_U(U_0;0)V_\alpha] \mathrm{d}x\\
		&+ \frac{2}{\varepsilon^2} \int V_\alpha^T [A_0(U^\varepsilon;\varepsilon) - A_0(U_0;0)]\{Q_U(U_0;0)V \}_\alpha \mathrm{d}x.
	\end{aligned}
\end{equation*}
From the structural stability conditions in Lemma \ref{4lemma1}, we can see
\begin{equation*}
	\begin{aligned}
		2\int V_\alpha^T A_0(U_0;0) Q_U(U_0;0) V_\alpha &= V_\alpha^T \bigg(A_0(U_0;0) Q_U(U_0;0) + Q_U^T(U_0;0)A_0(U_0;0) \bigg) V_\alpha \mathrm{d}x\\
		&\leq -c_0\parallel V_\alpha^{II} \parallel^2
	\end{aligned}
\end{equation*}
with $c_0$ a positive constant. Since $Q_U(U_0;0) = \mathrm{diag}(0,q_w(U_0,0))$, we have
\begin{equation*}
	\begin{aligned}		
		&\int V_\alpha^T A_0(U_0;0) [\partial^\alpha (Q_U(U_0;0)V) - Q_U(U_0;0)V_\alpha] \mathrm{d}x\\
		&= \int V_\alpha ^{\uppercase\expandafter {\romannumeral2}T}A^{22}_0(U_0;0) [\partial^\alpha(q_w(U_0;0) V^{II}) - q_w(U_0;0) V_\alpha^{II}] \mathrm{d}x\\
		&\leq C \parallel V_\alpha^{II}\parallel \parallel \partial^\alpha(q_w(U_0;0) V^{II}) - q_w(U_0;0) V_\alpha^{II} \parallel \\
		&\leq C \parallel V_\alpha^{II}\parallel \parallel q_w(U_0;0)\parallel_s \parallel V^{II}\parallel_{|\alpha|-1} \\
		& \leq \frac{\delta}{4}\parallel V_\alpha^{II}\parallel^2 + C \parallel V^{II}\parallel_{|\alpha|-1}^2.
	\end{aligned}
\end{equation*}
Note that the above term vanishes when $\alpha=0$. Here we use the calculus inequalities (Lemma \eqref{4lemma:calculus_inequ}). And for the remaining terms, we will use the calculus inequalities in Sobolev spaces repeatedly.
For the third item in $I_{21}^\alpha$, we have
\begin{equation*}
	\begin{aligned}
		&V_\alpha^T [A_0(U^\varepsilon;\varepsilon) - A_0(U_0;0)]\partial^\alpha(Q_U(U_0;0)V)\\
		=& V_\alpha^{IT}[A^{12}_0(U^\varepsilon;\varepsilon) - A^{12}_0(U_0;0)]\partial^\alpha(q_w(U_0;0)V^{II})\\
		&+ V_\alpha^{IIT}[A^{22}_0(U^\varepsilon;\varepsilon) - A^{22}_0(U_0;0)]\partial^\alpha(q_w(U_0;0)V^{II})\\
	\end{aligned}
\end{equation*}
Using \eqref{4eq:A0-err}, we have
\begin{equation*}
	\begin{aligned}
		&\int V_\alpha^T [A_0(U^\varepsilon;\varepsilon) - A_0(U_0;0)]\partial^\alpha(Q_U(U_0;0)V) \mathrm{d}x\\
		&\leq C(\varepsilon + \varepsilon^2 \triangle + \varepsilon^2 B_\varepsilon(t))\parallel V^{II}_\alpha\parallel \parallel V \parallel_s \\ 
		& \leq \frac{\delta}{4} \parallel V^{II}_\alpha\parallel^2 + C (\varepsilon + \varepsilon^2 \triangle + \varepsilon^2 B_\varepsilon(t))^2\parallel V \parallel_s^2.
	\end{aligned}
\end{equation*}
Therefore
\begin{equation*}
	I_{21}^\alpha \leq \frac{\delta - c_0}{\varepsilon^2} \norm{V_\alpha^{II}}^2 + \frac{C}{\varepsilon^2} \norm{V^{II}}^2_{|\alpha|-1} + C(1 + \varepsilon \triangle + \varepsilon B_\varepsilon(t))^2 \norm{V}_s^2.
\end{equation*}
Now, we consider $I_{22}^\alpha$,
% Recall that $A^{12}_0(U_0;0) =0$.
\begin{equation*}
	\begin{aligned}
		I_{22}^\alpha = & \frac{2}{\varepsilon^2} \int V_\alpha^T A_0(U^\varepsilon;\varepsilon)\bigg\{\varepsilon \partial_\varepsilon Q_U(U_0;0)V\bigg\}_\alpha \mathrm{d}x\\
		=&\frac{2}{\varepsilon} \int V_\alpha^{IT} \bigg( A^{12}_0(U^\varepsilon;\varepsilon) - A^{12}_0(U_0;0) \bigg) \{\partial_\varepsilon q_u(U_0;0)V^{I}\}_\alpha \mathrm{d}x\\
		& + \frac{2}{\varepsilon} \int V_\alpha^{IT} A^{12}_0(U^\varepsilon;\varepsilon) \{\partial_\varepsilon q_w(U_0;0)V^{II}\}_\alpha \mathrm{d}x\\
		&+ \frac{2}{\varepsilon} \int V_\alpha^{IIT}A^{22}_0(U^\varepsilon;\varepsilon)\bigg\{\partial_\varepsilon q_u(U_0;0)V^{I} + \partial_\varepsilon q_w(U_0;0)V^{II}\bigg\}_\alpha \mathrm{d}x.
%		\leq & \frac{\delta}{\varepsilon^2} \parallel V_\alpha^{II} \parallel^2 + C\parallel V \parallel^2_s.
	\end{aligned}
\end{equation*}
The first term is bounded by $C(1 + \varepsilon \triangle + \varepsilon B_\varepsilon(t)) \norm{V}_s^2$ and the remaining terms are dominated by $\frac{C}{\varepsilon}\norm{V}_s \norm{V_\alpha^{II}}$.
Hence
\begin{equation*}
	\begin{aligned}
		I_{22}^\alpha  \leq  \frac{\delta}{\varepsilon^2}\parallel V_\alpha^{II} \parallel^2 + C(1 + \varepsilon \triangle + \varepsilon B_\varepsilon(t)) \norm{V}_s^2
	\end{aligned}
\end{equation*}
Moreover, $I_{23}^\alpha$ can be rewritten as
\begin{equation*}
	\begin{aligned}
		I_{23}^\alpha = & \frac{2}{\varepsilon^2} \int V_\alpha^T A_0(U^\varepsilon;\varepsilon)\bigg\{[Q(U^\varepsilon;0)  - Q(U_\varepsilon^m;0) - Q_U(U_0;0)V]\bigg\}_\alpha \mathrm{d}x\\
		= & \frac{2}{\varepsilon^2} \int V_\alpha^T (A_0(U^\varepsilon;\varepsilon) - A_0(U_0;0)) \bigg\{[Q(U^\varepsilon;0) - Q(U_\varepsilon^m;0)-Q_U(U_0;0)V]\bigg\}_\alpha \mathrm{d}x\\
		& + \frac{2}{\varepsilon^2} \int V_\alpha^{IIT} A^{22}_0(U_0;0)\bigg\{[q(U^\varepsilon;0) - q(U_\varepsilon^m;0)-q_U(U_0;0)V]\bigg\}_\alpha \mathrm{d}x.
	\end{aligned}
\end{equation*}
According to Lemma \ref{4lemma1}, we know
\begin{equation*}
	\begin{aligned}
		&Q(U^\varepsilon;0) - Q(U_\varepsilon^m;0)-Q_U(U_0;0)V \\
		&=\begin{pmatrix}
			0 \\ q(U^\varepsilon;0) - q(U_\varepsilon^m;0)-q_U(U_\varepsilon^m;0)V
		\end{pmatrix}  + \begin{pmatrix}
			0 \\ q_U(U_\varepsilon^m;0)V - q_U(U_0;0)V
		\end{pmatrix},
	\end{aligned}
\end{equation*}
By the Taylor formula, it is clear that
\begin{equation*}
	\begin{aligned}
		\parallel \partial_\alpha (q(U^\varepsilon;0) - q(U_\varepsilon^m;0)-q_U(U_\varepsilon^m;0)V)\parallel \leq C\parallel V \parallel_s^2 = C \varepsilon^2\triangle \parallel V \parallel_s,
	\end{aligned}
\end{equation*}
\begin{equation*}
	\begin{aligned}
		\parallel \partial_\alpha (q_U(U_\varepsilon^m;0)V - q_U(U_0;0)V )\parallel \leq C(\varepsilon + \varepsilon^2 B_\varepsilon(t))\parallel V\parallel_s.
	\end{aligned}
\end{equation*}
Using \eqref{4eq:A0-err}, we obtain 
\begin{equation*}
	\begin{aligned}
		& \frac{2}{\varepsilon^2} \int V_\alpha^T (A_0(U^\varepsilon;\varepsilon) - A_0(U_0;0)) \bigg\{[Q(U^\varepsilon;0) - Q(U_\varepsilon^m;0)-Q_U(U_0;0)V]\bigg\}_\alpha \mathrm{d}x\\
		& \leq \frac{C}{\varepsilon^2} |A_0(U^\varepsilon;\varepsilon) - A_0(U_0;0)| \norm{V}_s \parallel \bigg\{[Q(U^\varepsilon;0) - Q(U_\varepsilon^m;0)-Q_U(U_0;0)V]\bigg\}_\alpha \parallel \\
		& \leq C(1 + \varepsilon \triangle + \varepsilon B_\varepsilon(t))^2 \parallel V \parallel_s^2,
	\end{aligned}
\end{equation*}
and
\begin{equation*}
	\begin{aligned}
		&\frac{2}{\varepsilon^2} \int V_\alpha^{IIT} A^{22}_0(U_0;0)\bigg\{[q(U^\varepsilon;0) - q(U_\varepsilon^m;0)-q_U(U_0;0)V]\bigg\}_\alpha \mathrm{d}x\\
		&\leq \frac{C}{\varepsilon^2}(\varepsilon + \varepsilon^2 \triangle + \varepsilon^2 B_\varepsilon(t))\parallel V_\alpha^{II} \parallel \parallel V \parallel_s\\
		&\leq \frac{\delta}{\varepsilon^2}\parallel V_\alpha^{II} \parallel + C(1 + \varepsilon \triangle + \varepsilon B_\varepsilon(t))^2 \parallel V \parallel_s^2.
	\end{aligned}
\end{equation*}
Therefore
\begin{equation*}
	I_{23}^\alpha \leq \frac{\delta}{\varepsilon^2}\parallel V_\alpha^{II} \parallel + C(1 + \varepsilon \triangle +  B_\varepsilon(t))^2 \parallel V \parallel_s^2.
\end{equation*}
Similarly, $I_{24}^\alpha$ can be bounded by
\begin{equation*}
	\begin{aligned}
		I_{24}^\alpha &= \frac{2}{\varepsilon^2} \int V_\alpha^T A_0(U^\varepsilon;\varepsilon)\bigg\{\varepsilon\bigg[\partial_\varepsilon Q(U^\varepsilon;0) - \partial_\varepsilon Q(U_\varepsilon^m;0) - \partial_\varepsilon Q_u(U_0;0)V\bigg]\bigg\}_\alpha \mathrm{d}x\\
		&\leq \frac{\delta}{\varepsilon^2}\parallel V_\alpha^{II} \parallel + C(1 + \varepsilon \triangle + \varepsilon B_\varepsilon(t))^2 \parallel V \parallel_s^2.
	\end{aligned}
\end{equation*}
For the last term in $I_2^\alpha$, since
\begin{equation*}
\begin{aligned}
	&Q(U^\varepsilon;\varepsilon)-Q(U^\varepsilon;0)-\varepsilon \partial_\varepsilon Q(U^\varepsilon;0) - (Q(U^m_\varepsilon;\varepsilon)-Q(U^m_\varepsilon;0)-\varepsilon \partial_\varepsilon Q(U^m_\varepsilon;0)) \\
	& = \varepsilon^2 \int_0^1\int_0^1 \partial^2_{\varepsilon\varepsilon}Q_U(\tau U^\varepsilon+(1-\tau)U_\varepsilon^m;\theta\varepsilon) V \mathrm{d} \tau \mathrm{d} \theta.
\end{aligned}
\end{equation*}
We have
\begin{equation*}
	I_{25}^\alpha \leq \frac{2}{\varepsilon^2} C \varepsilon^2 \parallel V\parallel_s^2 \leq C \parallel V\parallel_s^2.
\end{equation*}
Note that $(1 + \varepsilon \triangle + \varepsilon B_\varepsilon(t))^2 \leq C(1 + \triangle^2 + B_\varepsilon(t))$. Therefore
\begin{equation}\label{4eq:err2}
	I_2^\alpha \leq \frac{4\delta - c_0}{\varepsilon^2}\parallel V_\alpha^{II} \parallel^2 + \frac{C}{\varepsilon^2} \parallel V^{II}\parallel^2_{|\alpha|-1} + C(1 + \triangle^2 + B_\varepsilon(t)) \parallel V\parallel_s^2.
\end{equation}

Next we estimate $I_3^\alpha$. To this end, we observe
\begin{equation*}
	\begin{aligned}
		 I_3^\alpha =& \frac{2}{\varepsilon} \int V_\alpha^T A_0(U^\varepsilon;\varepsilon)\bigg(A(U^{\varepsilon};\varepsilon)\partial_x V_\alpha - [A(U^{\varepsilon};\varepsilon)\partial_x V]_\alpha\bigg) \mathrm{d}x\\
		 =& \frac{2}{\varepsilon} \int V_\alpha^T (A_0(U^\varepsilon;\varepsilon) - A_0(U_0;0))\bigg(A(U_0;0)\partial_x V_\alpha - [A(U_0;0)\partial_x V]_\alpha\bigg) \mathrm{d}x\\
		 & + \frac{2}{\varepsilon} \int V_\alpha^T A_0(U^\varepsilon;\varepsilon)\bigg((A(U^{\varepsilon};\varepsilon) - A(U_0;0))\partial_x V_\alpha - [(A(U^{\varepsilon};\varepsilon) - A(U_0;0))\partial_x V]_\alpha\bigg) \mathrm{d}x\\
		 & + \frac{2}{\varepsilon} \int V_\alpha^T A_0(U_0;0)\bigg(A(U_0;0)\partial_x V_\alpha - [A(U_0;0)\partial_x V]_\alpha\bigg) \mathrm{d}x.
	\end{aligned}
\end{equation*}
Using the calculus inequalities and \eqref{4eq:A0-err}, the first two term in above equation can be bounded by
\begin{equation*}
	C(1+\varepsilon \triangle + \varepsilon B_\varepsilon(t))\parallel V \parallel_s^2,
\end{equation*}
According to Lemma \ref{4lemma1} we know that $A^{11}(U_0;0) \equiv 0$. Thus, the last term of $I_3^\alpha$ can be rewritten as
\begin{equation*}
	\begin{aligned}
		&\frac{2}{\varepsilon} \int V_\alpha^T A_0(U_0;0)(A(U_0;0)\partial_x V_\alpha - [A(U_0;0)\partial_x V]_\alpha ) \mathrm{d}x\\
		= & \frac{2}{\varepsilon} \int V_\alpha^{IT} A^{11}_0(U_0;0)(A^{12}(U_0;0)\partial_x V^{II}_\alpha - [A^{12}(U_0;0)\partial_x V^{II}]_\alpha) \mathrm{d}x\\
		& + \frac{2}{\varepsilon} \int V_\alpha^{IIT} A^{22}_0(U_0;0)(A^{21}(U_0;0)\partial_x V^{I}_\alpha - [A^{21}(U_0;0)\partial_x V^{I}]_\alpha) \mathrm{d}x\\
		& + \frac{2}{\varepsilon} \int V_\alpha^{IIT} A^{22}_0(U_0;0)(A^{22}(U_0;0)\partial_x V^{II}_\alpha - [A^{22}(U_0;0)\partial_x V^{II}]_\alpha) \mathrm{d}x,
	\end{aligned}
\end{equation*}
in which each term on the right-hand side contains $V^{II}$. By the calculus inequalities, it is easy to see that
\begin{equation*}
\begin{aligned}
	&\frac{2}{\varepsilon} \int V_\alpha^T A_0(U_0;0)(A(U_0;0)\partial_x V_\alpha - [A(U_0;0)\partial_x V]_\alpha) \mathrm{d}x\\
	& \leq \frac{C}{\varepsilon} \parallel V^{II}\parallel_{|\alpha|}\parallel V\parallel_s\\
	& \leq \frac{\delta}{\varepsilon^2} \parallel V^{II}\parallel^2_\alpha + C \parallel V \parallel^2_s.
\end{aligned}
\end{equation*}
This implies
\begin{equation}\label{4eq:err3}
	I_3^\alpha \leq \frac{\delta}{\varepsilon^2} \parallel V^{II}\parallel^2_\alpha + C(1+\varepsilon \triangle + \varepsilon B_\varepsilon(t))\parallel V \parallel_s^2.
\end{equation}

For $I_4^\alpha$, we have
\begin{equation*}
	\begin{aligned}
		I_4^\alpha &= -2 \int V_\alpha^T A_0(U^\varepsilon;\varepsilon)\partial_\alpha (R^m_\varepsilon) \mathrm{d}x\\
		&= -2 \int V_\alpha^T (A_0(U^\varepsilon;\varepsilon)-A_0(U_0; 0))\partial_\alpha (R^m_\varepsilon) \mathrm{d}x -  2 \int V_\alpha^T A_0(U_0;0)\partial_\alpha (R^m_\varepsilon) \mathrm{d}x.
	\end{aligned}
\end{equation*}
Using the $\parallel F_m\parallel_s \leq C\varepsilon + Ce^{-\mu \tau} = C(\varepsilon + \varepsilon^2 B_\varepsilon(t))$ in Theorem \ref{4theorem:asp-err} and Assumption  \ref{4ass1}, we obtain
\begin{equation*}
	\begin{aligned}
		& -2 \int V_\alpha^T (A_0(U^\varepsilon;\varepsilon)-A_0(U_0;0))\partial_\alpha (R^m_\varepsilon) \mathrm{d}x\\
		&\leq C \parallel U^\varepsilon - U_0\parallel_s \parallel R^m_\varepsilon \parallel_s \norm{V}_s \\
		&\leq \varepsilon^m (1 + \varepsilon \triangle + \varepsilon B_\varepsilon(t))(\parallel Q_{U}(U_0;0)U_{m+1}\parallel_s + \parallel F_m \parallel_s) \norm{V}_s\\
		& \leq C(1 + \varepsilon \triangle + \varepsilon B_\varepsilon(t))^2\norm{V}_s^2 + C\varepsilon^{2m}.
	\end{aligned}
\end{equation*}
The remaining term can be bounded by
\begin{equation*}
	\begin{aligned}
		&-2 \int V_\alpha^T A_0(U_0;0)\partial_\alpha (R^m_\varepsilon) \mathrm{d}x\\
		&=-\int \varepsilon^{m-1} V_\alpha^{IIT} A^{22}_0(U_0;0)\partial_\alpha (q_w(U_0;0)w_{m+1}) \mathrm{d}x + \int \varepsilon^{m-1} V_\alpha^T A_0(U_0;0)\partial_\alpha F_m \mathrm{d}x\\
		&\leq \frac{\delta}{\varepsilon^2}\parallel V_\alpha^{II} \parallel^2 + C(1 + \varepsilon B_\varepsilon(t))^2\parallel V_\alpha \parallel^2 + C\varepsilon^{2m}.
	\end{aligned}
\end{equation*}
Therefore
\begin{equation}\label{4eq:err4}
	I_4^\alpha \leq\frac{\delta}{\varepsilon^2}\parallel V_\alpha^{II} \parallel + C(1 + \triangle^2 + B_\varepsilon(t)) \norm{V}_s^2 + C\varepsilon^{2m}.
\end{equation}

The last term is 
\begin{equation*}
	\begin{aligned}
		I_5^\alpha &= \int V_\alpha^T \bigg(\partial_t (A_0(U^\varepsilon;\varepsilon)) + \frac{1}{\varepsilon} \partial_x (A_0(U^\varepsilon;\varepsilon)A(U^\varepsilon;\varepsilon))\bigg)V_\alpha \mathrm{d}x.
	\end{aligned}
\end{equation*}
And we have
\begin{equation*}
	V_\alpha^T \partial_t A_0(U^\varepsilon;\varepsilon))V_\alpha \leq C |\partial_t V + \partial_t U_\varepsilon^m |\parallel V \parallel_s^2.
\end{equation*}
The equation of $V$ implies 
\begin{equation*}
	\begin{aligned}
		\partial_t V =& - \frac{1}{\varepsilon}A(U^{\varepsilon};\varepsilon)\partial_x V + \frac{1}{\varepsilon}\bigg(A(U^m_\varepsilon;\varepsilon)-A(U^\varepsilon;\varepsilon)\bigg)\partial_x U_\varepsilon^m\\
		& + \frac{1}{\varepsilon^2}\bigg(Q(U^\varepsilon;\varepsilon)-Q(U_\varepsilon^m;\varepsilon)\bigg) - R^m_\varepsilon.
	\end{aligned}
\end{equation*}
Since $\parallel V \parallel_s = \varepsilon^2 \triangle$, we have
\begin{equation*}
	\begin{aligned}
		|- \frac{1}{\varepsilon}A(U^{\varepsilon};\varepsilon)\partial_x V |\leq C\frac{1}{\varepsilon} \parallel V \parallel_s = \varepsilon \triangle,\\
		|\frac{1}{\varepsilon}\bigg(A(U^m_\varepsilon;\varepsilon)-A(U^\varepsilon;\varepsilon)\bigg)\partial_x U_\varepsilon^m| \leq C\frac{1}{\varepsilon} \parallel V \parallel_s = \varepsilon \triangle, \\
		|R^m_\varepsilon|\leq C \varepsilon^{m-1}\leq C.
	\end{aligned}
\end{equation*}
Moreover,
\begin{equation*}
	\begin{aligned}
		&Q(U^\varepsilon;\varepsilon)-Q(U_\varepsilon^m;\varepsilon) \\
		= & Q(U^\varepsilon;\varepsilon)-Q(U_\varepsilon^m;\varepsilon) - \partial_U Q(U_\varepsilon^m;\varepsilon)V \\
		&+ (\partial_U Q(U_\varepsilon^m;\varepsilon) - \partial_U Q(U_0;\varepsilon))V\\
		& + (\partial_U Q(U_0;\varepsilon)-\partial_U Q(U_0;0)) V + \partial_U Q(U_0;0)V.
	\end{aligned}
\end{equation*}
From the Lemma \ref{4lemma1}, we obtain that
\begin{equation*}
	\begin{aligned}
		\partial_U Q(U_0;0) = \begin{pmatrix}
			0 & 0\\
			0 & q_w(U_0;0)
		\end{pmatrix},
	\qquad \partial_U Q(U_0;0) V = \begin{pmatrix}
			0 \\
			q_w(U_0;0) V^{II}
		\end{pmatrix}.
	\end{aligned}
\end{equation*}
Then
\begin{equation*}
	\begin{aligned}
		&|Q(U^\varepsilon;\varepsilon)-Q(U_\varepsilon^m;\varepsilon)|\\
		&\leq C(\parallel V\parallel_s^2 + \norm{U_\varepsilon^m -U_0}_\infty \parallel V\parallel_s  + \varepsilon \parallel V \parallel_s + \parallel V^{II} \parallel_s )\\
		&\leq C(\norm{U_\varepsilon^m -U_0}_s^2 + \parallel V\parallel_s^2 + \parallel V^{II} \parallel_s + \varepsilon^2)\\
		&\leq C\varepsilon^2(1 + \varepsilon B_\varepsilon(t))^2 + \varepsilon^4 \triangle^2 + C\parallel V^{II} \parallel_s .
	\end{aligned}
\end{equation*}
Thus
\begin{equation*}
	|\partial_t V\parallel \leq C(1 + \triangle^2 + B_\varepsilon(t)) +\frac{C}{\varepsilon^2} \norm{V^{II}}_s .
\end{equation*}
Noting $\parallel \partial_t U_\varepsilon^m\parallel_s \leq C + C B_\varepsilon(t)$ from Assumption \ref{4ass1}, we obtain
\begin{equation*}
	\begin{aligned}
		&\int V_\alpha^T \partial_t A_0(U^\varepsilon;\varepsilon))V_\alpha \mathrm{d}x\\
		&\leq  C(1 + \triangle^2 + B_\varepsilon(t)) \norm{V}_s^2 + \frac{C}{\varepsilon^2}\parallel V^{II} \parallel_s \parallel V \parallel_s^2.
	\end{aligned}
\end{equation*}
Setting $\hat{A}(U^\varepsilon;\varepsilon)=A_0(U^\varepsilon;\varepsilon)A(U^\varepsilon;\varepsilon)$, the second term of $I_5^\alpha$ can be treated as
\begin{equation*}
	\begin{aligned}
	 	&\frac{1}{\varepsilon}\int V_\alpha^T \partial_x (A_0(U^\varepsilon;\varepsilon)A(U^\varepsilon;\varepsilon))V_\alpha \mathrm{d}x\\
		=& \frac{1}{\varepsilon} \int V^{IT}_\alpha\partial_x \hat{A}^{11}(U^\varepsilon;\varepsilon)V^{I}_\alpha \mathrm{d}x + \frac{2}{\varepsilon} \int V^{IT}_\alpha\partial_x \hat{A}^{12}(U^\varepsilon;\varepsilon)V^{II}_\alpha \mathrm{d}x\\
		 &+  \frac{1}{\varepsilon} \int V^{IIT}_\alpha\partial_x \hat{A}^{22}(U^\varepsilon;\varepsilon)V^{II}_\alpha \mathrm{d}x.
	\end{aligned}
\end{equation*}
Obviously, the last two terms are bounded by
\begin{equation*}
	\frac{C}{\varepsilon}\parallel V \parallel_s \parallel V^{II}_\alpha \parallel \leq \frac{\delta}{\varepsilon^2}\parallel V^{II}_\alpha \parallel^2 + C \parallel V \parallel_s^2.
\end{equation*}
Since $A^{11}(U_0;0) = 0$ and $A_0(U_0;0)$ is a block-diagonal matrix, we have $\hat{A}^{11}(U_0;0) = 0$. The first term can be bounded by
\begin{equation*}
	\begin{aligned}
		& \frac{1}{\varepsilon} \int V^{IT}_\alpha\partial_x \hat{A}^{11}(U^\varepsilon;\varepsilon)V^{I}_\alpha \mathrm{d}x =  \frac{1}{\varepsilon} \int V^{IT}_\alpha\partial_x (\hat{A}^{11}(U^\varepsilon;\varepsilon) -\hat{A}^{11}(U_0;0)) V^{I}_\alpha \mathrm{d}x\\
		&\leq C \frac{1}{\varepsilon} \parallel \partial_x(U^\varepsilon - U_0)\parallel_{\infty} \parallel V\parallel_s^2 \\
		&\leq C(1 + \varepsilon \triangle +  \varepsilon B_\varepsilon(t))\parallel V\parallel_s^2.
	\end{aligned}
\end{equation*}  
Therefore 
\begin{equation}\label{4eq:err5}
	\begin{aligned}
		I_5^\alpha \leq & \frac{\delta}{\varepsilon^2}\parallel V^{II}_\alpha \parallel^2 + C(1+ \triangle^2 + B_\varepsilon(t)) \norm{V}_s^2 + \frac{C}{\varepsilon^2}\parallel V^{II} \parallel_s\parallel V\parallel_s^2.
	\end{aligned}
\end{equation}

%Since
%\begin{equation*}
%	\begin{aligned}
%		&(1 + \varepsilon \triangle + \varepsilon B_\varepsilon(t))^2\\
%		& = 1 + 2\varepsilon \triangle + 2\varepsilon B_\varepsilon(t) + 2\varepsilon^2\triangle B_\varepsilon(t) + \varepsilon^2 \triangle^2 + \varepsilon^2 B_\varepsilon(t)^2\\
%		& \leq C(1+ \triangle^2 + B_\varepsilon(t)),
%	\end{aligned} 
%\end{equation*}
%substituting \eqref{4eq:err1},\eqref{4eq:err2},\eqref{4eq:err3},\eqref{4eq:err4},\eqref{4eq:err5} 
substituting \eqref{4eq:err1}--\eqref{4eq:err5} into the inequality \eqref{4eq:err} yields
\begin{equation*}
	\begin{aligned}
		&\frac{\mathrm{d} }{\mathrm{d} t}\int e(V_\alpha) \mathrm{d} x + \frac{c_0 - 8\delta }{\varepsilon^2}\parallel V_\alpha^{II} \parallel^2 \\
		& \leq \frac{C}{\varepsilon^2} \parallel V^{II}\parallel^2_{|\alpha|-1} + C(1 +\triangle^2 + B_\varepsilon(t))\parallel V\parallel_s^2 + \frac{C}{\varepsilon^2}\parallel V^{II} \parallel_s\parallel V\parallel_s^2 + C\varepsilon^{2m}.
	\end{aligned}
\end{equation*}
Let $\delta$ to be sufficiently small such that $c_1 = c_0 -8\delta \in (0,c_0) $. then we have
\begin{equation}\label{4eq:estimate}
	\begin{aligned}
		&\frac{\mathrm{d} }{\mathrm{d} t}\int_\Omega e(V_\alpha) \mathrm{d} x + \frac{c_1}{\varepsilon^2}\parallel V_\alpha^{II} \parallel^2 \\
		&\leq \frac{C}{\varepsilon^2} \parallel V^{II}\parallel^2_{|\alpha|-1} + C(1 +\triangle^2 + B_\varepsilon(t))\parallel V\parallel_s^2 + \frac{C}{\varepsilon^2}\parallel V^{II} \parallel_s\parallel V\parallel_s^2 + C\varepsilon^{2m}.
	\end{aligned}
\end{equation}
Recall $\parallel V^{II}\parallel^2_{-1}=0$. when $|\alpha|=1$, we see that $ \frac{C}{\varepsilon^2} \parallel V^{II}\parallel^2_{|\alpha|-1}$ on the right-hand side of\eqref{4eq:estimate} can be controlled by $\frac{c_1}{\varepsilon^2}\parallel V_\alpha^{II} \parallel^2$. More generally, let $\eta\in (0,1]$, Multiplying \eqref{4eq:estimate} by $\eta^{|\alpha|}$ and summing up the equalities for all index $\alpha$ with $|\alpha|\leq s$ yields 
\begin{equation*}
	\begin{aligned}
		&\frac{\mathrm{d} }{\mathrm{d} t} \sum_{|\alpha|\leq s} \eta^{|\alpha|} \int_\Omega e(V_\alpha) \mathrm{d} x + \frac{c_1}{\varepsilon^2}\sum_{|\alpha|\leq s}\eta^{|\alpha|}\parallel V_\alpha^{II} \parallel^2 \\
		\leq & \frac{C}{\varepsilon^2} \sum_{|\alpha|\leq s-1}\eta^{|\alpha|+1}\parallel V^{II}\parallel^2_{|\alpha|} + C(1 +\triangle^2 + B_\varepsilon(t))\parallel V\parallel_s^2\\
		& + \frac{C}{\varepsilon^2}\parallel V^{II} \parallel_s\parallel V\parallel_s^2  + C\varepsilon^{2m},
	\end{aligned}
\end{equation*}
in which $C$ is independent of $\eta$. Let $\eta$ be suitably small. Then
\begin{equation*}
	\begin{aligned}
		\frac{C}{\varepsilon^2} \sum_{|\alpha|\leq s-1}\eta^{|\alpha|}\parallel V^{II}\parallel^2_{|\alpha|+1} \leq \frac{c_1}{2\varepsilon^2}\sum_{|\alpha|\leq s}\eta^{|\alpha|}\parallel V_\alpha^{II} \parallel^2
	\end{aligned}
\end{equation*}
and
\begin{equation*}
	\begin{aligned}
		\frac{c_1 \eta^s}{ 2\varepsilon^2}\parallel V^{II}\parallel^2_{s} \leq \frac{c_1}{2\varepsilon^2}\sum_{|\alpha|\leq s}\eta^{|\alpha|}\parallel V_\alpha^{II} \parallel^2.
	\end{aligned}
\end{equation*}
Therefore 
\begin{equation*}
	\begin{aligned}
		&\frac{\mathrm{d} }{\mathrm{d} t} \sum_{|\alpha|\leq s} \eta^{|\alpha|} \int_\Omega e(V_\alpha) \mathrm{d} x  + \frac{c_1 \eta^s}{2\varepsilon^2} \parallel V^{II} \parallel^2_s \\
		&\leq  C(1 +\triangle^2 + B_\varepsilon(t))\parallel V\parallel_s^2  + \frac{C}{\varepsilon^2}\parallel V^{II} \parallel_s\parallel V\parallel_s^2 + C\varepsilon^{2m}.
	\end{aligned}
\end{equation*}
By the Young inequality, we have
\begin{equation*}
	\begin{aligned}
		C\parallel V^{II} \parallel_s \parallel V\parallel_s^2 &\leq \frac{c_1 \eta^s}{4}\parallel V^{II} \parallel_s^2 + \frac{c^2}{c_1 \eta^s}\parallel V\parallel_s^4\\
		&\leq \frac{c_1 \eta^s}{4}\parallel V^{II} \parallel_s^2 + \frac{c^2}{c_1 \eta^s}\varepsilon^2\triangle \parallel V\parallel_s^2.
	\end{aligned}
\end{equation*}
Thus,
\begin{equation*}
	\begin{aligned}
		&\frac{\mathrm{d} }{\mathrm{d} t} \sum_{|\alpha|\leq s} \eta^{|\alpha|} \int_\Omega e(V_\alpha) \mathrm{d} x  + \frac{c_1 \eta^s}{4\varepsilon^2} \parallel V^{II} \parallel^2_s \\
		&\leq C(1 +\triangle^2 + B_\varepsilon(t) + \frac{1}{\eta^{2s}})\parallel V\parallel_s^2 + C\varepsilon^{2m}.
	\end{aligned}
\end{equation*}
Note that
\begin{equation*}
	C^{-1} |V_\alpha|^2 \leq e(V_\alpha) \leq C |V_\alpha|^2.
\end{equation*}
Now we fix $\eta>0$. Integrating this inequality over $[0,T_m]$ and noting that $ \sum_{|\alpha|\leq s} \eta^{|\alpha|} \int_\Omega e(V_\alpha) \mathrm{d} x$ is equivalent to $\parallel V_\alpha(T) \parallel^2$, we use $\parallel V(0)\parallel_s = O(\varepsilon^m)$ to obtain
\begin{equation*}
	\begin{aligned}
		&\parallel V_\alpha(T) \parallel^2 + \frac{1}{\varepsilon^2} \int_0^T \parallel V^{II}(t) \parallel^2_s \mathrm{d} t \leq C T\varepsilon^{2m} + \int_0^T C(1 +\triangle^2 + B_\varepsilon(t))\parallel V(t)\parallel_s^2 \mathrm{d} t.
	\end{aligned}
\end{equation*}
Then
\begin{equation*}
	\begin{aligned}
		\parallel V(T) \parallel_s^2 \leq C T\varepsilon^{2m} + \int_0^T C(1 +\triangle^2 + B_\varepsilon(t))\parallel V(t)\parallel_s^2 \mathrm{d} t.
	\end{aligned}
\end{equation*}
We apply Gronwall's lemma to above equation to get
\begin{equation}\label{4eq:v-err}
	\begin{aligned}
		\parallel V(T) \parallel_s^2 \leq C T_m\varepsilon^{2m} \exp\bigg[ \int_0^T C(1 +\triangle^2 + B_\varepsilon(t))\mathrm{d} t\bigg].
	\end{aligned}
\end{equation}
Since $\parallel V\parallel_s = \varepsilon^2 \triangle$, it follows from above equation that
\begin{equation*}
	\triangle(T)^2 \leq C T_m\varepsilon^{2m-4} \exp\bigg[ \int_0^T C(1 +\triangle^2 + B_\varepsilon(t))\mathrm{d} t\bigg]\equiv \Phi(T).
\end{equation*}
Thus,
\begin{equation*}
	\Phi'(T) = C(1 +\triangle^2 + B_\varepsilon(t)) \Phi(T) \leq C(1 + B_\varepsilon(t)) \Phi(T) + C\Phi^2(T).
\end{equation*}
because of $\int_0^T B_\varepsilon(t) \leq \frac{1}{2\mu}$. Applying the nonlinear Gronwall-type inequality in Lemma  \ref{4lemma:gronwall} to the last inequality yields
\begin{equation*}
	\triangle(T)^2 \leq \sup_{[0,T_m]}\Phi(T)\leq C \exp\bigg[ \int_0^T C(1 + B_\varepsilon(t))\mathrm{d} t\bigg].
\end{equation*}
if we assume $m>2$ and choose $\varepsilon$ so small that $\Phi(0) = C T_m \varepsilon^{2m-4}<\delta$. Then there exists a constant $C$, independent of $\varepsilon$, such that
\begin{equation*}
	\triangle(T)\leq C.
\end{equation*}
for any $T\in [0,\min\{T_\varepsilon, T_m\})$.
Because of \eqref{4eq:v-err}, there exists a constant $K>0$, independent of $\varepsilon$, such that
\begin{equation*}
	\parallel V \parallel_s \leq K \varepsilon^m.
\end{equation*}
This completes the proof of Theorem \ref{4theorem-1}.

\section{Conclusions}\label{5sec}
In this work, we study the stability of the radiation hydrodynamics system, which the HMP$_N$ moment method is adopted for radiative transfer equation. Our results give a rigorous derivation of the widely used macroscopic model in radiation hydrodynamics. This work shows the importance of Yong's structural stability condition in analyzing the compatibility of hyperbolic relaxation systems. For non-relativistic limit, we investigate the singular limit problem for system \eqref{4eq:hy_equ}. The process can be promoted for other quasilinear hyperbolic system with source term.

\section{Appendix}\label{appendix}
In this Appendix, we prove $\tilde{A}^{11}(\tilde{U}_{eq};0)=0$ and $\tilde{A}_{\tilde{u}}^{11}(\tilde{U}_{eq};0)=0$ for Lemma \ref{4lemma1}. 

Take value on the equilibrium state and set $\varepsilon =0$ in $\eqref{4eq:a-tilde}$. We can obtain
\begin{equation}\label{4eq:tilde-a-eq}
	\begin{aligned}
		\tilde{A}(\tilde{U}_{eq};0) = D_{U}\tilde{U}(\tilde{U}_{eq};0)\begin{pmatrix}
			0_{3\times 3}&0_{3\times (N+1)}\\
			0_{(N+1)\times 3}&\tilde{D}^{-1}\tilde{M}\tilde{D}(\tilde{U}_{eq})
		\end{pmatrix}(D_{U}\tilde{U})^{-1}(\tilde{U}_{eq};0),		
	\end{aligned}
\end{equation}
where $\tilde{D}^{-1}\tilde{M}\tilde{D}(\tilde{U}_{eq})\in \mathrm{R}^{(N+1)\times (N+1)}$. 
%Using \eqref{4eq:D_U_tilde{U}} and 
%\eqref{2eq:radiation-P} $P=a D_{U}\tilde{U}(\tilde{U}_{eq};0)$ 
%and the expression of $P$, 
From the above discussion in Section \ref{4sec}, we know that
\begin{equation*}
	D_{U}\tilde{U}(\tilde{U}_{eq};0) = \begin{pmatrix}
		P_1 & 0\\
		0 & I_{N\times N}
	\end{pmatrix},
\end{equation*}
in which $P_1\in \mathrm{R}^{4\times 4}$ defined in \eqref{2eq:radiation-P}. 
%Divide the matrix $\tilde{A}(\tilde{U}_{eq};0)$ as the corresponding to partition of $D_{U}\tilde{U}(\tilde{U}_{eq};0)$. 
Then $\tilde{A}(\tilde{U}_{eq};0)$ can be rewritten as
\begin{equation*}
	\begin{aligned}
		\tilde{A}(\tilde{U}_{eq};0) =\begin{pmatrix}
		P_1 & 0\\
		0 & I_{N\times N}
	\end{pmatrix}
	\begin{pmatrix}
			0_{3\times 3} & 0_{3\times 1}&0_{3\times N}\\
			0_{1\times 3} & g_1 & g_2\\
			0_{N\times 3} & g_3 & g_4\\
		\end{pmatrix}\begin{pmatrix}
		P_1^{-1} & 0\\
		0 & I_{N\times N}
	\end{pmatrix}, 		
	\end{aligned}
\end{equation*}
where $g_1$ is the first element in the upper left corner of $\tilde{D}^{-1}\tilde{M}\tilde{D}(\tilde{U}_{eq})$, $g_2\in \mathrm{R}^{1\times N} $, $g_3\in \mathrm{R}^{N\times 1}$, $g_4\in \mathrm{R}^{N\times N}$ are corresponding block of matrix $\tilde{D}^{-1}\tilde{M}\tilde{D}(\tilde{U}_{eq})$.

Next, we calculate $g_1$ and the first component of vector $g_3$. Note that $\tilde{D}(\tilde{U}_{eq})$ is diagonal matrix which showed in \eqref{2eq:tilde-d}, $\tilde{M}(\alpha)=\tilde{\Lambda}^{-1}\langle \mu\tilde{\Phi}^{[\alpha]},(\tilde{\Phi}^{[\alpha]})^T\rangle_{\tilde{\mathrm{H}}_N^{[\alpha]}}$ and $\tilde{\Lambda}(\alpha) = \mathrm{diag}(\tilde{\kappa}_{0,0},\tilde{\kappa}_{1,1},\cdots,\tilde{\kappa}_{N,N})$ due to \eqref{2eq:defin-s}. It follows from the definition of the inner product of $\tilde{\mathrm{H}}_N^{[\alpha]}$ \eqref{2eq:tilde-h-dot} that
\begin{equation*}
	\begin{aligned}
		\tilde{M}_{11}(\alpha) &=\tilde{\kappa}^{-1}_{0,0}(\alpha) \int_{-1}^1 \mu \tilde{\Phi}^{[\alpha]}_0(\mu) \tilde{\Phi}^{[\alpha]}_0(\mu)/\tilde{w}^{[\alpha]}(\mu) \mathrm{d} \mu\\
		&=\tilde{\kappa}^{-1}_{0,0}(\alpha)\int_{-1}^1 \mu \tilde{w}^{[\alpha]}(\mu) \mathrm{d} \mu =\tilde{\kappa}^{-1}_{0,0}(\alpha)\tilde{\kappa}_{1, 0}(\alpha),\\
		\tilde{M}_{21}(\alpha) &=\tilde{\kappa}^{-1}_{1,1}(\alpha)\int_{-1}^1 \mu \tilde{\Phi}^{[\alpha]}_1(\mu) \tilde{\Phi}^{[\alpha]}_0(\mu)/\tilde{w}^{[\alpha]}(\mu) \mathrm{d} \mu\\
		&=\tilde{\kappa}^{-1}_{1,1}(\alpha)\int_{-1}^1 \mu \tilde{\phi}^{[\alpha]}_1(\mu) \tilde{w}^{[\alpha]}(\mu)  \mathrm{d} \mu =\tilde{\kappa}^{-1}_{1,1}(\alpha)\tilde{\kappa}_{1, 1}(\alpha) =1.
	\end{aligned}
\end{equation*}
Hence, $g_1 =\tilde{M}_{11}(0)= \tilde{\kappa}^{-1}_{0,0}(0)\tilde{\kappa}_{1, 0}(0) =0$. 
%Since $\tilde{M}_{21}=1$, thus t
The first components of $g_3$ is $(-2b(\theta)\tilde{M}_{21}\beta_0^{-1}\neq 0$. Thus $g_3$ is not zero. Therefore
\begin{equation}\label{4eq:a-eq-tilde}
	\begin{aligned}
		\tilde{A}(\tilde{U}_{eq};0) &=\begin{pmatrix}
			P_1 & 0\\
			0 & I_{N\times N}
		\end{pmatrix}
		\begin{pmatrix} \begin{pmatrix}
			0_{3\times 3} & 0_{3\times 1}\\
			0_{1\times 3} & 0 \end{pmatrix} & \begin{pmatrix}  0_{3\times N}\\ g_2\end{pmatrix} \\
			\begin{pmatrix} 0_{N\times 3} & g_3 \end{pmatrix} & g_4
		\end{pmatrix} 
		\begin{pmatrix}
		P_1^{-1} & 0\\
		0 & I_{N\times N}
		\end{pmatrix} \\
		&=\begin{pmatrix}
			0_{4\times 4} & P_1 \begin{pmatrix}
			0_{3\times N}\\g_2\end{pmatrix}\\
			\begin{pmatrix}
			0_{N\times 3} & g_3\end{pmatrix}P_1^{-1} & g_4
		\end{pmatrix}.	
	\end{aligned}
\end{equation}
Since $g_3$ is not zero and the matrix $P_1$ is invertible, the rank of the matrix $\begin{pmatrix}
			0_{N\times 3} & g_3\end{pmatrix}P_1^{-1}$ is $1$ .
			Divided the matrix $\tilde{A}$ as follows
%Let the matrix $\tilde{A}$ be divided as follows
\begin{equation*}
	\begin{aligned}
		\tilde{A}(U;\varepsilon) 
		\triangleq \begin{pmatrix}
			\tilde{A}^{11}(\tilde{U};\varepsilon) & \tilde{A}^{12}(\tilde{U};\varepsilon)\\
			\tilde{A}^{21}(\tilde{U};\varepsilon) & \tilde{A}^{22}(\tilde{U};\varepsilon)
		\end{pmatrix},
	\end{aligned}
\end{equation*}
where $\tilde{A}^{11}(\tilde{U};\varepsilon) \in \mathrm{R}^{3\times 3}$, $\tilde{A}^{12}(\tilde{U};\varepsilon) \in \mathrm{R}^{3\times (N+1)}$, $\tilde{A}^{21}(\tilde{U};\varepsilon) \in \mathrm{R}^{(N+1)\times (3}$, $\tilde{A}^{22}(\tilde{U};\varepsilon) \in \mathrm{R}^{(N+1)\times (N+1)}$. Then it follows from \eqref{4eq:a-eq-tilde} that $\tilde{A}^{11}(\tilde{U}_{eq};0) =0 $ for all $\tilde{U}_{eq} \in \tilde{G}_{eq}$. Meanwhile, $\tilde{A}^{21}(\tilde{U}_{eq};0)$ is the matrix formed by the first three columns of the following $(N+1)\times 4$ matrix
\begin{equation*}
	\begin{pmatrix}
		0_{1\times 4}\\
		\begin{pmatrix}
			0_{N\times 3} & g_3\end{pmatrix}P_1^{-1}
	\end{pmatrix}
\end{equation*}
Thus, $\tilde{A}^{21}(\tilde{U}_{eq};0)$ is not full-rank matrix. 

Furthermore, we analyze $\tilde{A}^{11}_{\tilde{u}}(\tilde{U}_{eq};0)$. Firstly, we show that $\tilde{A}_u^{11}(\tilde{U}_{eq};0 )=0$. For any $u=(\rho, ~\rho v, ~\rho E)$, we have 
\begin{equation}\label{eq:deriva-u-a}
	\begin{aligned}
		\tilde{A}_u(\tilde{U}_{eq};0) &= \partial_{u}(D_{U}\tilde{U}) \begin{pmatrix}
			0&0\\
			0&\tilde{D}^{-1}\tilde{M}\tilde{D}
		\end{pmatrix}(D_{U}\tilde{U})^{-1} \\
		& + D_{U}\tilde{U} \begin{pmatrix}
			0&0\\
			0&\partial_{u} (\tilde{D}^{-1}\tilde{M}\tilde{D})
		\end{pmatrix}(D_{U}\tilde{U})^{-1} + D_{U}\tilde{U}\begin{pmatrix}
			0&0\\
			0&\tilde{D}^{-1}\tilde{M}\tilde{D}
		\end{pmatrix}\partial_{u}(D_{U}\tilde{U})^{-1}.
	\end{aligned}
\end{equation}
From the expression of $D_{U}\tilde{U}$ in \eqref{4eq:D_U_tilde{U}} and $\theta = \theta(u)$, we know that 
\begin{equation*}
	\partial_{u}(D_{U}\tilde{U})(\tilde{U}_{eq};0) = \begin{pmatrix}
		0_{3\times 3} & 0\\
		Y_1 & 0_{(N+1)\times (N+1)}
	\end{pmatrix}
\end{equation*}
with $Y_1$ is a non-zero matrix in $\mathrm{R}^{(N+1)\times 3}$. Therefore, we have 
\begin{equation*}
	\partial_{u}(D_{U}\tilde{U})(\tilde{U}_{eq};0) \begin{pmatrix}
			0&0\\
			0&\tilde{D}^{-1}\tilde{M}\tilde{D}(\tilde{U}_{eq};0)
		\end{pmatrix} = \begin{pmatrix}
		0_{3\times 3} & 0\\
		Y_1 & 0_{(N+1)\times (N+1)}
	\end{pmatrix}\begin{pmatrix}
			0&0\\
			0&\tilde{D}^{-1}\tilde{M}\tilde{D}(\tilde{U}_{eq};0)
		\end{pmatrix} = 0
\end{equation*}
In the second term in \eqref{eq:deriva-u-a}, since $\tilde{D}^{-1}\tilde{M}\tilde{D}$ is only depend on $w$, so $\partial_{u} (\tilde{D}^{-1}\tilde{M}\tilde{D})=0$, which yields the second term vanish. From above discussion, we know that
\begin{equation*}
	\begin{aligned}
		D_{U}\tilde{U}(\tilde{U}_{eq};0)\begin{pmatrix}
			0&0\\
			0&\tilde{D}^{-1}\tilde{M}\tilde{D}(\tilde{U}_{eq};0)
		\end{pmatrix} = \begin{pmatrix}
			0_{4\times 4} & P_1 \begin{pmatrix}
			0_{3\times N}\\g_2\end{pmatrix}\\
			\begin{pmatrix}
			0_{N\times 3}&g_3\end{pmatrix} & g_4
		\end{pmatrix}.
	\end{aligned}
\end{equation*}
A tedious calculation shows that the matrix of inverse transformation is
\begin{equation}\label{inver-trans}
	\begin{aligned}
		&(D_U\tilde{U})^{-1}(\tilde{U};0)\\
		&= \begin{pmatrix}
	 		\begin{pmatrix}
				1 & 0 & 0 & 0 & 0\\
				0 & 1 & 0 & 0 & 0 \\
				-\frac{b'\theta_{\rho}}{1 + b'\theta_{\rho E}} & -\frac{b'\theta_{\rho v}}{1 + b'\theta_{\rho E}} & \frac{1}{1 + b'\theta_{\rho E}}  & -\frac{1}{1 + b'\theta_{\rho E}}  & 0\\
				-\frac{b'\theta_{\rho}}{(1 + b'\theta_{\rho E})\kappa_{0,0}} & -\frac{b'\theta_{\rho v}}{(1 + b'\theta_{\rho E})\kappa_{0,0}} & -\frac{b'\theta_{\rho E}}{(1 + b'\theta_{\rho E})\kappa_{0,0}} & \frac{1}{(1 + b'\theta_{\rho E})\kappa_{0,0}} & -\frac{\kappa'_{0,0}f_0}{\kappa_{0,0}}\\
				0 & 0 & 0 & 0 & 1
			\end{pmatrix}
			& 0_{5\times(N-1)}\\
			0_{(N-1)\times 5} & I_{(N-1)\times (N-1)}
		\end{pmatrix}.
	\end{aligned}
\end{equation}
Thus we can obtain 
\begin{equation*}
	\partial_{u}(D_{U}\tilde{U})^{-1}(\tilde{U}_{eq};0) = \begin{pmatrix}
		Y_2 & 0_{4\times N} \\
		0_{N\times 4} & 0_{N\times N}
	\end{pmatrix}
\end{equation*}
with $Y_2 \in \mathrm{R}^{4 \times 4}$. So the third term of \eqref{eq:deriva-u-a} can be rewritten as
\begin{equation*}
	\begin{aligned}
		&D_{U}\tilde{U}(\tilde{U}_{eq};0)\begin{pmatrix}
			0&0\\
			0&\tilde{D}^{-1}\tilde{M}\tilde{D}(\tilde{U}_{eq};0)
		\end{pmatrix}\partial_{u}(D_{U}\tilde{U})^{-1}(\tilde{U}_{eq};0)\\
		& = \begin{pmatrix}
			0_{4\times 4} & P_1 \begin{pmatrix}
			0_{3\times N}\\g_2\end{pmatrix}\\
			\begin{pmatrix}
			0_{N\times 3}&g_3\end{pmatrix} & g_4
		\end{pmatrix}\begin{pmatrix}
		Y_2 & 0_{4\times N} \\
		0_{N\times 4} & 0_{N\times N}
	\end{pmatrix} = \begin{pmatrix}
		0_{4\times 4} & 0_{4\times N}\\
		Y_3 & 0_{N\times N}
	\end{pmatrix}
	\end{aligned}
\end{equation*}
with $Y_3$ is corresponding matrix.
Thus that the $3\times 3$ block in the upper left corner of matrix $\tilde{A}_u(\tilde{U}_{eq};0)$ is zero. This means $\tilde{A}_u^{11}(\tilde{U}_{eq};0 )=0$ for any $u=(\rho, ~\rho v, ~\rho E)$. 

% Since $\tilde{A}^{11}(\tilde{U}_{eq};0) =0$ for any $\tilde{U}_{eq} \in \tilde{G}_{eq} $, thus $\tilde{A}_u^{11}(\tilde{U}_{eq};0 )=0$ for any $u=(\rho, ~\rho v, ~\rho E)$. 
Since $\tilde{u} = \tilde{u}(u, w)$, we have
\begin{equation*}
	\begin{aligned}
		\partial_{\tilde{u}} \tilde{A}^{11}(\tilde{U}_{eq};0) &= \partial_{u} \tilde{A}^{11}(\tilde{U}_{eq};0) \frac{\partial u}{\partial \tilde{u}} + \partial_{w}\tilde{A}^{11}(\tilde{U}_{eq};0) \frac{\partial w}{\partial \tilde{u}}\\
		&=\partial_{w}\tilde{A}^{11}(\tilde{U}_{eq};0) \frac{\partial w}{\partial \tilde{u}}.
	\end{aligned}
\end{equation*}
According to expression of $(D_U\tilde{U})^{-1}$ in \eqref{inver-trans}, we know $\frac{\partial w}{\partial \tilde{u}}$ are zero except for $\frac{\partial f_0}{\partial\tilde{u}}$.
%, especially $\frac{\partial \alpha}{\partial \tilde{u}}=0$. 
Thus
\begin{equation*}
	\begin{aligned}
		\partial_{\tilde{u}} \tilde{A}^{11}(\tilde{U}_{eq};0) = \partial_{w}\tilde{A}^{11}(\tilde{U}_{eq};0) \frac{\partial w}{\partial \tilde{u}} = \partial_{f_0}\tilde{A}^{11}(\tilde{U}_{eq};0) \frac{\partial f_0}{\partial \tilde{u}}.
	\end{aligned}
\end{equation*}
Thanks to the equations of hydrodynamical variables \eqref{4eq:rhov-equ}, we set
\begin{equation*}
	\tilde{F}(\tilde{U};\varepsilon)=\bigg(\varepsilon \rho v, \quad \varepsilon (\rho v^2 +  p +  \kappa_{2,2}f_2 + \kappa_{2,0}f_0), \quad \varepsilon(\rho E v + p v) + \kappa_{1,0}f_0\bigg)^T,
\end{equation*}
with $\kappa_{2,2}$, $\kappa_{2,0}$, $\kappa_{1,0}$ are function of $\alpha$ such that $\tilde{A}^{11}(\tilde{U};\varepsilon)=\partial_{\tilde{u}}\tilde{F}(\tilde{U};\varepsilon)$. Thus
\begin{equation*}
	\begin{aligned}
		\partial_{f_0}\tilde{A}^{11}(\tilde{U};\varepsilon)&= \partial_{f_0}\big(\partial_{\tilde{u}}\tilde{F}(\tilde{U};\varepsilon)\big)=\partial_{\tilde{u}}\partial_{f_0}\tilde{F}(\tilde{U};\varepsilon) = \partial_{\tilde{u}}\big(0,~ \varepsilon \kappa_{2,0}(\alpha),~ \kappa_{1,0}(\alpha)\big)^T.
	\end{aligned}
\end{equation*}
And since $\frac{\partial w}{\partial \tilde{u}}$ are zero except for $\frac{\partial f_0}{\partial\tilde{u}}$, we have
\begin{equation*}
	\begin{aligned}
		\partial_{\tilde{u}}\kappa_{1,0}(\alpha) &= \partial_{u}\kappa_{1,0}(\alpha)\frac{\partial u}{\partial \tilde{u}} + \partial_{w}\kappa_{1,0}(\alpha)\frac{\partial w}{\partial \tilde{u}} = \partial_{f_0}\kappa_{1,0}(\alpha)\frac{\partial f_0}{\partial \tilde{u}} = 0
	\end{aligned}
\end{equation*}
Similarly $\partial_{\tilde{u}}\kappa_{2,0}(\alpha) = 0$.
Therefore $\partial_{\tilde{u}} \tilde{A}^{11}(\tilde{U}_{eq};0)=0$.
	
In conclude, on all equilibrium state $\tilde{U}_{eq}$, we see that
\begin{equation*}
	\tilde{A}^{11}(\tilde{U}_{eq};0) =0,\qquad \partial_{\tilde{u}} \tilde{A}^{11}(\tilde{U}_{eq};0)=0.
\end{equation*}
Moreover, $\tilde{A}^{21}(\tilde{U}_{eq};0)$ is not full--rank matrix.

%\clearpage
%\addcontentsline{toc}{chapter}{Reference}
%\bibliographystyle{unsrt}
%\bibliography{refs.bib}

\end{document}